\documentclass[letterpaper,10pt,reqno]{amsart}
\usepackage[usenames,dvipsnames]{xcolor}

\RequirePackage[OT1]{fontenc}
\usepackage{underscore}
\usepackage[portrait,margin=3.5cm]{geometry}
\usepackage{mathrsfs}
\usepackage[colorlinks,citecolor=blue,urlcolor=blue,linkcolor=blue, bookmarks=false, pdffitwindow=true, pdfpagetransition=scroll]{hyperref}
\usepackage[foot]{amsaddr}
\usepackage{amssymb,amsthm,amsfonts,amsbsy,latexsym,dsfont}
\usepackage[textsize=small]{todonotes}       
\usepackage{graphicx}
\usepackage[english]{babel}
\usepackage{subfigure}
\usepackage{appendix}
\usepackage{framed}

\usepackage{enumerate}
\newenvironment{enumeratei}{\begin{enumerate}[\upshape (i)]}{\end{enumerate}}

\usepackage{color}              
\usepackage{graphicx}
\usepackage[]{amsmath}
\usepackage{amsthm}
\usepackage{amssymb}
\usepackage{bbm}
\usepackage{hyperref}
\usepackage{comment}
\usepackage{microtype}

\definecolor{MyDarkblue}{rgb}{0,0.08,0.50}
\definecolor{Brickred}{rgb}{0.65,0.08,0}

\hypersetup{
colorlinks=true,       
    linkcolor=MyDarkblue,          
    citecolor=Brickred,        
    filecolor=red,      
    urlcolor=cyan           
}

\newtheorem{theorem}{Theorem}[section]
\newtheorem{lemma}[theorem]{Lemma}
\newtheorem{proposition}[theorem]{Proposition}
\newtheorem{corollary}[theorem]{Corollary}

\newtheorem{definition}[theorem]{Definition}
\newtheorem{assumption}[theorem]{Assumption}
\newtheorem{remark}[theorem]{Remark}
\newtheorem{claim}[theorem]{Claim}

\newcommand{\Pv}{\mathbb{P}}

\newcommand{\Ev}{\mathbb{E}}

\newcommand{\CC}{\mathcal{C}}

\newcommand{\sss}{\scriptscriptstyle}
\newcommand{\Var}{{\rm Var}}

\newcommand{\CE}{{\mathcal{E}}}

\newcommand*{\CMD}{{\mathrm{CM}}_n(\boldsymbol{d})}

\newcommand{\e}{{\mathrm e}}

\numberwithin{equation}{section}

\newcommand{\R}{\mathbb{R}}
\newcommand{\N}{\mathbb{N}}
\newcommand{\Z}{\mathbb{Z}}

\newcommand{\union}{\cup}

\newcommand{\CA}{\mathcal {A}}
\newcommand{\CB}{\mathcal {B}}
\newcommand{\CD}{\mathcal {D}}

\newcommand{\CG}{\mathcal {G}}

\newcommand{\CL}{\mathcal {L}}
\newcommand{\CM}{\mathcal {M}}

\newcommand{\CP}{\mathcal {P}}

\newcommand{\CR}{\mathcal {R}}

\newcommand{\CT}{\mathcal {T}}

\newcommand*{\wih}{\widehat}

\newcommand*{\vr}{\varrho}
\newcommand*{\la}{\lambda}

\newcommand*{\ve}{\varepsilon}

\newcommand*{\Vv}{{\text{\bf Var}}}

\newcommand*{\be}{\begin{equation}}
\newcommand*{\ee}{\end{equation}}
\newcommand*{\ba}{\begin{aligned}}
\newcommand*{\ea}{\end{aligned}}
\newcommand*{\barr}{\begin{array}{c}}
\newcommand*{\earr}{\end{array}}
\def \toinp    {\buildrel {\Pv}\over{\longrightarrow}}
\def \toindis  {\buildrel {d}\over{\longrightarrow}}
\def \toas     {\buildrel {a.s.}\over{\longrightarrow}}

\def \simp {\buildrel {\mathbb{P}} \over{\sim}  }

\newcommand*{\Yrn}{Y_r^{\scriptscriptstyle{(n)}}}
\newcommand*{\Ybn}{Y_b^{\scriptscriptstyle{(n)}}}
\newcommand*{\Ywn}{Y_1^{\scriptscriptstyle{(n)}}}
\newcommand*{\Yln}{Y_2^{\scriptscriptstyle{(n)}}}

\newcommand*{\bnr}{b_n^{\sss{(r)}}}
\newcommand*{\bnb}{b_n^{\sss{(b)}}}
\newcommand*{\anr}{\alpha_n^{\sss{(r)}}}
\newcommand*{\anb}{\alpha_n^{\sss{(b)}}}
\newcommand*{\dnr}{\delta_n^{\sss{(r)}}}
\newcommand*{\dnb}{\delta_n^{\sss{(b)}}}

\newcommand*{\benb}{\beta_n^{\sss{(b)}}}
\newcommand*{\wur}{\widetilde u^{\sss{(r)}}}
\newcommand*{\wub}{\widetilde u^{\sss{(b)}}}
\newcommand*{\wgar}{\widetilde \Gamma^{\sss{(r)}}}
\newcommand*{\wgab}{\widetilde \Gamma^{\sss{(b)}}}

\newcommand*{\CMBN}{\mathcal M^{\sss{(b)}}_n}
\newcommand*{\MBN}{M^{\sss{(b)}}_n}
\newcommand*{\wit}{\widetilde}
\newcommand*{\opt}{\mathcal{O}}
\newcommand*{\ops}{\mathrm{O}}

\newcommand*{\opss}{\mathrm{Opt}}
\newcommand*{\ind}{\mathbbm{1}}
\def\namedlabel#1#2{\begingroup
    #2%
    \def\@currentlabel{#2}%
    \phantomsection\label{#1}\endgroup
}

\begin{document}
	\title[Competition on Configuration model for $\tau\in (2,3)$]{Fixed speed competition on the configuration model with infinite variance degrees: equal speeds}

	\date{\today}

	\author[van der Hofstad]{Remco van der Hofstad}
	\address{Department of Mathematics and
	    Computer Science, Eindhoven University of Technology, P.O.\ Box 513,
	    5600 MB Eindhoven, The Netherlands.}
\author[Komj\'athy]{J\'ulia Komj\'athy}
	\email{rhofstad@win.tue.nl, j.komjathy@tue.nl}

\begin{abstract}
We study competition of two spreading colors starting from single sources on the configuration model with i.i.d.\ degrees following a power-law distribution with exponent $\tau\in (2,3)$.
In this model two colors spread with a fixed and equal speed
on the unweighted random graph.
 We analyse how many vertices the two colours paint eventually, in the quenched setting. We show that coexistence sensitively depends on the initial local neighbourhoods of the source vertices: if these neighbourhoods are `dissimilar enough', then there is no coexistence, and the `loser' color paints a polynomial fraction of the vertices with a random exponent.
   If the  local neighbourhoods of the starting vertices are `similar enough', then there is coexistence, i.e., both colors paint a strictly positive proportion of vertices. We give a quantitative characterisation of `similar' local neighborhoods: two random variables describing the double exponential growth of local neighborhoods of the source vertices must be within a factor $\tau-2$ of each other.

   This picture reinforces the common belief that location is an important feature in advertising.
   This paper is a follow-up of the similarly named paper that handles the case when the speeds of the two colors are not equal. There, we have shown that  the faster color paints almost all vertices, while the slower color paints only a random subpolynomial fraction of the vertices.
\end{abstract}

\maketitle

\section{Introduction and results}
\subsection{The model and the main result}
Let us consider the configuration model $\CMD$ on $n$ vertices, where the degrees $D_v, v\in \{1,2,\dots, n\}:=[n]$ are i.i.d.\ with a power-law tail distribution. That is, given the number of vertices $n$, to each vertex we assign a random number of half-edges drawn independently from a distribution $F$ and the half-edges are then paired randomly to form
edges. In case the total number of half-edges $\CL_n:=\sum_{v\in[n]} D_v$ is not even, then we add one half-edge to $D_n$ (see below for more details).
 We assume that
\be\label{eq::F} \frac{c_1}{x^{\tau-1}}\le 1- F(x)= \Pv(D>x) \le \frac{C_1}{x^{\tau-1}},\ee
 with $\tau \in (2,3)$, and all edges have weight $1$. We assume that $\Pv(D\ge 2)=1$ guaranteeing that the graph has almost surely a unique connected component of size $n(1-o_{\Pv}(1))$ see e.g.\ \cite[Vol II., Theorem 4.1]{H10} or \cite{MolRee95, MolRee98}.\\
 We further denote the mass function of  (the \emph{size-biased version} of $D)-1$ by
 \be\label{def::size-biased1}
f^\star_{j}:=\frac{(j+1)\Pv(D=j+1)}{\Ev[D]}\text{, }\quad j\geq 0.
\ee
We write $F^\star(x)$ for the distribution function $F^\star(x)=\sum_{j=0}^{\llcorner x\lrcorner}f^\star_{j}$.
It is not hard to see
that there exist $0<c_1^\star\le C_1^\star<\infty$, such that
\be \label{eq::size-biased2} \frac{c_1^\star}{x^{\tau-2}}\le 1- F^\star(x)\le \frac{C_1^\star}{x^{\tau-2}}.  \ee
 Pick two vertices $\mathcal R_0$ (the red source) and $\mathcal B_0$ (the blue source) uniformly at random in $[n]$,
 and consider these as two sources of spreading infections.
 Each infection spreads \emph{deterministically} on the graph: in general,
for color blue it takes $\la$ time units to pass through an edge, while color red needs $1$ unit of time for that.
  Without loss of generality we can assume that $\la\ge 1$. The $\la>1$ case was treated in \cite{BarHofKom14}, while in this paper we study the case where $\la=1$. Each vertex is painted the color of the infection that reaches it
  first, keeps its color forever, and starts coloring the outgoing edges at the speed of its color.  When the two colors reach a vertex at the \emph{same} time, the vertex gets color red
  or blue with  probability $1/2$ each, independently of everything else. In general, this rule could be modified to an \emph{arbitrary adapted rule}, i.e., a rule that does not depend on the future. Throughout this paper, we handle the case when the rule is such that at all times, there is a strictly positive probability for both colors to get the vertex under consideration, and this decision is independent of the decision about other vertices. However, we emphasise that other rules can be handled analogously with possibly different outcomes for coexistence.

	Let $\mathcal R_t:=\mathcal R_t(n)$ and $\mathcal B_t:=\mathcal B_t(n)$
  denote the number of red and blue vertices occupied up to time $t$, respectively. We denote by $\CB_{\infty}:=\CB_{\infty}(n)$ the number of vertices eventually occupied by blue.
   Roughly speaking, the first result of this paper, Theorems \ref{thm::main} and \ref{thm::main2} below, tell us that with high probability (whp), i.e., with probability tending to $1$ as the size of the graph tends to infinity, one out of two things can happen:

   1.) When the local neighbourhoods of the source vertices are \emph{dissimilar enough}, i.e., if one of them grows at a significantly faster speed than the other, then there is no coexistence, and the color with faster-growing local neighbourhood gets $n-o_{\Pv}(n)$ vertices.
 The number of vertices the other color paints is a polynomial of $n$ with a random exponent. More precisely, blue paints whp $n^{H_n(Y_r, Y_b)}$ many vertices, where the coefficient  $H_n(Y_r, Y_b)<1$ is a random function that depends on $n, \tau$, and two random variables $Y_r$ and $Y_b$, that can intuitively be interpreted as some measure of `how fast' the neighbourhoods of the source vertices grow: the faster the local neighbourhoods grow, the larger these variables are. Moreover, $H_n(Y_r, Y_b)$ does not converge in distribution: it has an oscillatory part that exhibits  `$\log\log$-periodicity'.

  2.) If the local neighbourhoods of the source vertices are \emph{similar enough}, i.e., the random variables $Y_r, Y_b$, describing the speed of growth of the neighbourhoods, are within a factor $\tau-2$ of each other, then both colors get a linear proportion of vertices, i.e., there is asymptotic coexistence in the model.
 More precisely, both $\CB_\infty/n, \CR_\infty/n$ stay strictly between $0$ and $1$ as $n\to \infty$. We also show that as the ratio of $Y_r, Y_b$ approaches $\tau-2$, the proportion of vertices painted by the color with the smaller $Y$ value tends to zero, and in this sense, the transition from coexistence to non-coexistence is smooth. These results are established by the analysis of a \emph{branching process random coloring scheme}, which we find interesting in its own right. In this coloring scheme, the BP is run until the maximal degree reaches some value $Q$, and then, the vertices in the last generation of the stopped BP are colored red, blue or stay uncolored according to their degrees. The colored vertices spread their color to earlier generation vertices on the BP tree, following a rule that is similar to the spreading dynamics for red and blue in the graph. We show that the root of the BP can get both colors with positive probability uniformly as $Q\to \infty$, which in turn implies coexistence in the configuration model. For more on this problem see Section \ref{sc::BP-color} below.

The heuristic interpretation of these results and the main result in \cite{BarHofKom14} in marketing terminology is as follows: if a company gains customers via word-of-mouth recommendations, then the one with faster spreading speed takes most of the network. On the other hand, if both of the companies gain customers at the same speed, then the one with better starting location can gain most of the network, at least if this location is significantly better than the starting location of the other company (companies).
This result reinforces the common sense about importance of location in (online) location-based  advertisement: there is even a slogan called `Location, location, location', see e.g.  \cite{BruKum07, Das05, LLL, XuOhTe09}. 

According to our knowledge this is the first random graph model that can both produce asymptotic coexistence and non-coexistence when the spreading speeds are equal (i.e., both outcomes happen with positive probability). Further, which one of the two outcome happens depends on the starting location of the spreading colors, and moreover, the asymptotic probability of these events is explicitly computable. 

  The other main result, Theorem \ref{thm::distances}, describes the distribution of the fluctuations of typical distances in the graph. More precisely, it was shown in \cite{HHZ07}  that the graph distance between two uniformly chosen vertices is concentrated around $2 \log \log n / |\log (\tau-2)|$, with bounded and non-converging fluctuations around this value.
  Here, we provide a different, in some sense more natural, proof of this fact, and provide a different representation of the fluctuation: we describe it as a simple function of $n$ and two independent random variables that describe the growth rate of the local neighborhood of the source vertices. This function contains integer parts, hence, the same $\log \log$-periodicity  phenomenon is present as the one in Theorems \ref{thm::main} and Theorem \ref{thm::main2}, coming from the fact that the edge weights are concentrated on a lattice. Although it is not apparent from their final forms, we emphasize that the result of \cite{HHZ07} and Theorem \ref{thm::distances} are \emph{the same}: we show this fact in Appendix \ref{appendix::distances}.

  To be able to state the main theorem precisely, let us define the following random variables:
  \begin{definition}[Galton-Watson limits]\label{def::limit-variables}
  Let $Z_k^{\sss{(r)}}, Z_k^{\sss{(b)}}$ denote the number
 of individuals in the $k$th generation of two independent copies of a Galton-Watson process described as follows: the size of the first generation has distribution $F$ satisfying \eqref{eq::F},
 and all the further generations have offspring distribution $F^\star$ from \eqref{def::size-biased1}.
 Then, for a fixed but small $\vr>0$ let us define
\be\label{def::yrn-ybn}\Yrn:=(\tau-2)^{t(n^\vr)} \log (Z^{\sss{(r)}}_{t(n^{\vr})}),  \quad \Ybn:=(\tau-2)^{ t(n^{\vr})} \log (Z^{\sss{(b)}}_{ t(n^{\vr})}),\ee
  where $t(n^\vr)=\inf_k\{\max\{ Z_k^{\sss{(r)}}, Z_k^{\sss{(b)}} \} \ge n^\vr\}$.
Let us further introduce
  \be\label{def::Y} Y_r:= \lim_{k\to\infty} (\tau-2)^k \log (Z_k^{\sss{(r)}}), \quad Y_b:=\lim_{k\to\infty} (\tau-2)^k \log (Z_k^{\sss{(b)}}).\ee
\end{definition}
  We will see in Section \ref{sc::BP} below that these quantities are well-defined and that $(\Yrn, \Ybn)\toindis (Y_r,Y_b)$ from \eqref{def::Y} as $n\to \infty$. To be able to state the results shortly, let us further define, for $j=r,b$,
\be\label{def::ti-bi} T_j:=\left\lfloor\frac{\log\log n -\log((\tau-1)Y_j^{\sss{(n)}})}{|\log (\tau-2)|}-1\right\rfloor, \quad
b_n^{(j)}:= \left\{\frac{\log\log n -\log((\tau-1)Y_j^{\sss{(n)}})}{|\log (\tau-2)|}\right\},\ee
where $\lfloor x\rfloor $ denotes the largest integer that is at most $x$ and $\{x\}= x-\lfloor x\rfloor$ denotes the fractional part of $x$.

Let us also introduce four events $E_<, E_>, O_<, O_>$ where $E,O$ stands for the events that $T_r+T_b-1$ is even or odd, respectively, and the subscript $<$ is added when $\tau-1<(\tau-2)^{\bnr}+(\tau-2)^{\bnb}$ and the subscript $>$  is added when $\tau-1>(\tau-2)^{\bnr}+(\tau-2)^{\bnb}$. We write $D_n^{\max}(t)$ for the degree of the maximal degree vertex occupied by the losing color at time $t$.

For sequences of random or deterministic variables $X_n, Y_n$ we write $X_n = o_{\Pv}(Y_n)$ and $X_n = O_{\Pv}(Y_n)$ if the sequence $X_n/Y_n \toinp 0 $ and is tight, respectively.

 Recall that $\CB_\infty, \CR_\infty$ denotes the number of vertices eventually occupied by the blue and red colors, respectively. With these notations in mind, we can state our main results:

\begin{theorem}[Total number of vertices painted by the losing color]\label{thm::main} Let us assume wlog that $\Yrn>\Ybn$ in Definition \ref{def::limit-variables}, that is, the `losing' color is blue.
When $\Ybn/\Yrn \le \tau-2$, then $\CR_\infty/n\toinp 1$ whp, and
\[ \frac{\log (\CB_\infty)}{\log n\cdot (\tau-1)^{-1} f_n(\Yrn, \Ybn) }  \toindis \sqrt{\frac{Y_b}{Y_r}}, \]
where $f_n(\Yrn, \Ybn)$ is an oscillating random variable  given by
\be\label{eq::main-f} f_n(\Yrn, \Ybn) \!=\! (\tau\!-\!2)^{(\bnr-\bnb - 1 - \ind_{O})/2}
\!\!\left(\! (\tau-2)^{\bnb\! + \ind_{O_<}}\! +\! (\tau\!-\!1\!-\!(\tau\!-\!2)^{\bnr})(\tau\!-\!2)^{ \ind_{O_>}}\! \right)\! ,  \ee
where $O=O_< \cup O_>$.
\end{theorem}
\begin{remark}\normalfont
Note that in Theorem \ref{thm::main}, the function $f_n(\Yrn, \Ybn)$ filters out the oscillations coming from $\log\log$-periodicity, and hence, it depends on $\Yrn, \Ybn$.  We emphasise that in general it is not true that $f_n(\Yrn, \Ybn) (\tau-1)^{-1}<1$. However, it is true that
\[ \sqrt{\Ybn/\Yrn} f_n(\Yrn, \Ybn) (\tau-1)^{-1} < 1, \]
hence we get the statement that $\CB_\infty=o_{\Pv}(n)$. See Lemma \ref{lem::no-coexistence} below for the proof of this fact.
\end{remark}
The next theorem holds under the following (technical) assumption:
\begin{assumption}\label{assume::abs-cont}
The limiting random variable $Y=\lim_{n\to \infty} (\tau-2)^k\log Z_k$ of the BP described in Definition \ref{def::limit-variables} has an absolutely continuous distribution function with support containing an interval of the form $(0, K), \ K\in \R^+ \cup\{\infty\}$.
\end{assumption}
 The criteria on $F$ required for this assumption to hold are not obvious: according to our knowledge, no necessary and sufficient condition for absolute continuity can be found in the literature. We provide some necessary criterion based on the work \cite{Sene73, Sene74}   below in Assumption \ref{assume::convex}. This assumption is not tight, though, milder criteria on the slowly varying function hidden in \eqref{eq::F} can also guarantee the statement. We improve the already existing criteria in an upcoming short note to be published elsewhere \cite{HofKom15}.
\begin{theorem}[Asymptotic coexistence when $Y_b/Y_r>\tau-2$]\label{thm::main2}
Let us assume wlog that $\Yrn>\Ybn$ in Definition \ref{def::limit-variables}, that is, the `losing' color is blue, and that Assumption \ref{assume::abs-cont} holds. When $q:=\Ybn/\Yrn > \tau-2$, then both colors can paint a linear proportion of the vertices. More precisely, there exists deterministic constants $0<c(q)<C(q)<1$ such that whp as $n\to \infty$
\[ c(q)\le \frac{\CB_\infty}{n}\le C(q). \]
Further, we have $c(q), C(q)\to 0$ as $q \searrow \tau-2$.
\end{theorem}
\begin{remark} For the specific values of $c(q)$ and $C(q)$ see the proof of the theorem on page 49.
\end{remark}
For an event $A$, we denote $\Pv_n(A):=\Pv(A| D_1, D_2, \dots, D_n)$. We say that the two competing spreading processes asymptotically coexist on a sequence of finite graphs indexed by $n$, if the limiting ratios $\CR_\infty(n)/n$ and $\CB_\infty(n)/n$ are both strictly positive with strictly positive probability, as $n\to \infty$.
An immediate consequence of Theorem \ref{thm::main2} is the following result:
\begin{corollary}[Probability of coexistence] In the above competition model with equal speeds, under Assumption \ref{assume::abs-cont}, whp
\[ \lim_{n\to \infty }\Pv_n( \text{ coexistence occurs } )  = \Pv\,(\, Y_b/ Y_r \in ( \tau-2, (\tau-2)^{-1}) ). \]
\end{corollary}

\begin{remark}\normalfont The core of the proof of Theorem \ref{thm::main2} is Proposition \ref{prop::BP-color} below. This proposition establishes that the probability that the root of a BP described in Definition \ref{def::limit-variables} gets both colors with strictly positive probability in a random coloring scheme, see Section \ref{sc::BP-color} below.
\end{remark}
\begin{remark}\label{rem::coexistence-precise}\normalfont
 We conjecture that the statement of Theorem \ref{thm::main2} can be further sharpened, namely, we suspect that there exists a random variable $\widehat h_n (\Yrn, \Ybn) \in (0,1)$ so that on the event $\{\Ybn/\Yrn \in (\tau-2,1)\}$,  we have
 $\CB_\infty/(n\!\cdot\! \widehat h_n(\Yrn, \Ybn)) \toinp 1$ as $n\to \infty$, where $\widehat h(n, \Yrn, \Ybn)$ does not converge, but oscillates between two constants $ c_2(q)< C_2(q)$ with $n$, $c_2(q), C_2(q)$ satisfying $c(q)\le c_2(q)< C_2(q)\le C(q)$, where $q=\Ybn/\Yrn$, and both $c_2(q), C_2(q) \to 0$ as $q \searrow \tau-2$.
(Here, $c(q),C(q)$ are from Theorem \ref{thm::main2}.)
\end{remark}

\begin{remark} \normalfont The statements of Theorems \ref{thm::main} and \ref{thm::main2} remain valid also if the red and blue processes are started from either ends of a uniformly chosen edge. In this case, the laws of $Y_r, Y_b$ are limits of branching processes as in Definition \ref{def::limit-variables}, where also the root has offspring distribution $F^\star$.
 \end{remark}
All the consecutive results hold again \emph{without} Assumption \ref{assume::abs-cont}.
Let us define $D_n^{\max}(t)$ as the degree of the maximal degree vertex that blue has colored before or at time $t$.
The next theorem is about the degree of the maximal degree vertex that each color can eventually paint:

\begin{theorem}[Maximal degree of the `losing' color]\label{thm::maxdegree}  Let us assume wlog that $\Ybn<\Yrn$ in Definition \ref{def::limit-variables}, i.e., the losing color is blue. Then, with high probability,
the degree $D_n^{\max, \mathrm{red}}(\infty)$ of the maximal degree vertex that red can eventually paint always tends to the maximal degree in the graph, that is,
\[  D_{n}^{\max, \mathrm{red}}(\infty) = n^{1/(\tau-1)(1+o_{\Pv}(1))}. \]
The maximal degree vertex that blue can eventually paint satisfies the following:

(i) When $\Ybn/\Yrn\le \tau-2$,
\[ \frac{\log D_n^{\max}(\infty)}{\log n \cdot (\tau-1)^{-1} h_n(\Yrn, \Ybn)} \toindis \sqrt{\frac{Y_b}{Y_r}}. \]
where $h_n(\Yrn, \Ybn)\le 1$ is an oscillating random variable  given by
\be\ba\label{eq::h1} h_n(\Yrn, \Ybn) =& \ind_{ E_< \cup O_>} (\tau-2)^{(\bnb+\bnr-1 - \ind_{O_>})/2} +\\
&+\ind_{ E_> \cup O_<} (\tau-2)^{(\bnr-\bnr-1 - \ind_{O_<})/2} ((\tau-1)-(\tau-2)^{\bnr}).    \ea \ee

(ii) When $\Ybn/\Yrn > \tau-2$, the maximal degree of blue sensitively depends on the fractional parts $\bnr, \bnb$. With $T_r, T_b$ as in \eqref{def::ti-bi},
\begin{enumerate}[(a)]
\item if $T_b-T_r=0$, then whp
 \[ D_{n}^{\max}(\infty) = n^{1/(\tau-1)(1+o_{\Pv}(1))},\]
 \item if $T_b-T_r=1$ and $\tau-1<(\tau-2)^{\bnr}+(\tau-2)^{\bnb}$, then whp
 \[ D_{n}^{\max}(\infty)  = n^{(\tau-2)^{\bnb}/(\tau-1)(1+o_{\Pv}(1))},\]
 \item if $T_b-T_r=1$ and $\tau-1>(\tau-2)^{\bnr}+(\tau-2)^{\bnb}$, then whp
 \[ D_{n}^{\max}(\infty) = n^{ ((\tau-1)-(\tau-2)^{\bnr})/(\tau-1)(1+o_{\Pv}(1))}.\]
 \end{enumerate}

\end{theorem}

\begin{remark}\normalfont Compare Theorem \ref{thm::maxdegree} to Theorem \ref{thm::main} to see that Case (ii) corresponds to coexistence.
Also, we emphasise that there is a coupling of the graphs $\CMD$ for $n\ge 1$ so that even the stronger statement $(\Yrn, \Ybn) \toinp (Y_r, Y_b)$ is valid. In this coupling construction, the event $\{ \Ybn/\Yrn> \tau-2\}$ converges to the event $\{Y_b/Y_r > \tau-2\}$. Case (ii) in Theorem \ref{thm::maxdegree} heuristically says that even under this limiting event, the maximal degree of blue shows some oscillation with $n$. \end{remark}

As a side result of the proof of Theorem \ref{thm::maxdegree}, we get a new description of typical distances in the graph:

\begin{theorem}[Fluctuations of typical distances]\label{thm::distances}
In the configuration model with i.i.d. degrees from distribution $D$ satisfying \eqref{eq::F} with power-law exponent $\tau\in(2,3)$, the typical distance $\CD_n(u,v)$ between two uniformly picked vertices $u:=\CR_0,v:=\CB_0$ can be described using $\Yrn$ and $\Ybn$ in \eqref{def::yrn-ybn} as
\be\label{eq::duv-1} \ba   \CD_n(u,v)=&\left\lfloor\frac{\log\log n -\log((\tau-1)\Ybn)}{|\log (\tau-2)|}\right\rfloor + \left\lfloor\frac{\log\log n -\log((\tau-1)\Yrn)}{|\log (\tau-2)|}\right\rfloor  \\
&\, -1 + \ind\{(\tau-2)^{b_n^{(r)}}+(\tau-2)^{b_n^{(b)}}<\tau-1 \} + o_\Pv(1), \ea\ee
whp, where $b_n^{(j)}=\left\{ \frac{\log\log n -\log((\tau-1)Y_j^{\sss{(n)}})}{|\log (\tau-2)|}\right\}$ for $j=r,b$. Equivalently,
\[\CD_n(u,v)- \frac{2\log\log n}{|\log (\tau-2)|}+1 +b_n^{(r)}+b_n^{(b)} - \ind\{\tau-1>(\tau-2)^{b_n^{(r)}}+(\tau-2)^{b_n^{(b)}}\} \toindis \frac{-\log((\tau-1)^2 Y_rY_b)}{|\log (\tau-2)|}. \]
\end{theorem}
\begin{remark}\normalfont  Note that Theorem \ref{thm::distances} implies that the typical distances in the graph are concentrated around $2 \log\log n / |\log (\tau-2)|$ with bounded fluctuations, a result that already appeared in \cite{HHZ07} under weaker assumptions on $F$. Our proof is considerably simpler than that in \cite{HHZ07}. The second statement of the theorem `filters out' the oscillations arising from fractional part issues: it is not hard to see that
\[ 1 +b_n^{(r)}+b_n^{(b)} - \ind\{\tau-1>(\tau-2)^{b_n^{(r)}}+(\tau-2)^{b_n^{(b)}}\}\in \Big[\frac{2 \log\tfrac{2}{\tau-1} } { |\log (\tau-2)|}, 2\Big)\]
oscillating with $n$. We emphasise here that the essential statement of Theorem \ref{thm::distances} and \cite[Theorem 1.2]{HHZ07} are the same, however, they provide a different description of typical distances.
\end{remark}


\subsection{Random coloring schemes for branching process trees}\label{sc::BP-color}
The crucial ingredient in the proof of Theorem \ref{thm::main2} boils down to the analysis of the following problem, that we find interesting in its own right. The specific version of the problem, which we solve in this paper, can be described as follows:

Suppose we have an infinite-mean Galton-Watson BP with offspring distribution given in \eqref{eq::size-biased2}. We let this BP grow until a vertex with degree at least $Q$ appears in the process.

Then, there is a  \emph{starting rule}: We fix a parameter $\gamma \in (1, 1/(\tau-2))$. In the last generation of the stopped BP we paint every vertex  with degree in the interval $[Q, Q^\gamma)$ red and vertices with degree in the interval $[Q^\gamma, Q^{1/(\tau-2)})$ red or blue with equal probability (if any).

After this, we sequentially color earlier generations, using a \emph{flow rule}: if a vertex has both red and blue children, it gets painted red or blue with equal probability; if it has children of only one color, then it takes that color; if it has no colored children, it stays uncolored, independently for each vertex in the same generation.

\begin{proposition}\label{prop::BP-color}
Fix a $\gamma\in (1,1/(\tau-2))$ and consider the above described coloring scheme of a branching process described in Definition \ref{def::limit-variables}. Assume further that Assumption \ref{assume::abs-cont} holds for the limiting random variable $Y$ of this BP.
Then there exist constants $0<c(\gamma)\le C(\gamma)<1$ such that
\[ c(\gamma)\le\liminf_{Q \to \infty} \Pv( \emph{root is painted blue}) \le \limsup_{Q \to \infty} \Pv( \emph{root is painted blue}) \le C(\gamma).\]
Further, $C(\gamma) \searrow 0$ as $\gamma\nearrow 1/(\tau-2)$.
\end{proposition}

This result is the core of the proof of the coexistence in Theorem \ref{thm::main2}. Note that this proposition is itself non-trivial since the proportion of blue vertices among all colored vertices in the last generation tends to zero as $Q \to \infty$, and further, the generation where the process is stopped also tends to infinity as $Q\to \infty$. As a result of these two effects, a smaller and smaller proportion of blue vertices have to `make their way' down to the root that is further and further away. Heuristically speaking, the rule that a vertex flips a coin that \emph{does not depend on
the number of its red and blue children} saves the blue color: this effect `exaggerates' the proportion of blue vertices as the generation number decreases towards the root.

To gain a more precise result on the proportion of blue vertices in the graph in Theorem \ref{thm::main2} (in particular, to prove the conjecture in Remark \ref{rem::coexistence-precise}), one has to gain a deeper understanding of the probability that the root is painted blue in this coloring scheme. In particular, the dependence of $\Pv(\text{root is blue})$ on the parameter $\gamma$ directly translates to the dependence of $\CB_{\infty}/n$ on the ratio $q=\Ybn/\Yrn \in (\tau-2, 1)$. However, for a more detailed analysis of $\Pv(\text{root is blue})$, one has to know more about the shape of the density function of the limiting variable $Y=\lim_{k\to \infty} (\tau-2)^k \log Z_k$ similar as in \eqref{def::Y}. Without more specific  assumptions on the offspring distribution than the one in \eqref{eq::F}, this is beyond our reach.

The random coloring scheme above might be generalised as follows: Suppose we have a branching process (a simple discrete time Galton-Watson BP in the case of this paper, but not necessarily in general) that we let grow until the (random) time when an individual with degree at least $Q$ appears, for some number $Q\gg 1$. Then, there is a \emph{starting rule} that colors some individuals in the BP, where the color depends on the degree of the individual. Further, this dependence is so that one color (say red) paints significantly more vertices than the other color (say blue), and the difference is exaggerated as $Q\to \infty$. Plus, the starting rule is so that only high-enough degrees get colored, the rest of the vertices in the BP remains unpainted.

Then, we prescribe a \emph{flow rule}: any vertex in the BP takes the color of one of its children according to some rule that might depend on the number of children of each color, but is independent for different vertices.

The question is: under what circumstances can the root be painted by both colors with strictly positive probability, as $Q \to \infty$? If this is possible, then how does this probability depend on the parameter of the starting rule?

We suspect that the property that the offspring distribution has infinite mean is crucial, as well as the strictly positive probability of taking each color in the presence of children of both colors. For instance, we conjecture that coexistence might not occur with other very natural degree dependent coloring rules, e.g.,  when a vertex takes the color of one of its neighbors proportional to the number of neighbors of that color.

\subsection{Discussion and open problems}
We gave an overview of related literature on first passage percolation, competitive spreading processes on lattices, applications such as word-of-mouth recommendations in online and offline marketing and epidemiology references in \cite[Section 1.2]{BarHofKom14}. Hence, we omit repetition and refer the reader for references about related mathematical and applied models there.

Here we review only results on competition on random graphs, and state our conjectures about possible generalizations.

Antunovic, Dekel, Mossel and Peres \cite{ADMP11}
give a detailed analysis of competition on random regular graphs (degree at least $3$) on $n$ vertices with i.i.d. exponential edge weights. They analyse scenarios where the speed of the two colors $\lambda_{1},\lambda_{2}$ might differ, and also the initial number of colored vertices might grow with $n$. They show that whp the color with higher rate occupies $n-o_{\Pv}(n)$ vertices and the slower color paints approximately $n^\beta$ vertices for some deterministic function $\beta(\la_1, \la_2)$.  When the speeds are equal, they show coexistence starting from single sources. We conjecture that their result can be generalised for the configuration model with i.i.d.\ continuous edge weights as long as the second moment of the degree distribution is finite, i.e., $\tau >3$ holds in \eqref{eq::F}.

Next, van der Hofstad and Deijfen \cite{DH13} studied competition of two colors from uniformly picked single source vertices with i.i.d.\  exponential edge weights, on the configuration model with i.i.d.\ degrees satisfying \eqref{eq::F} with $\tau\in (2,3)$. They prove that even if the speeds are not equal, the `winner' color is \emph{random}, and the winning color paints all but a finite number of vertices. The randomness of the `winner' color comes from the fact that the underlying Markov branching process explodes in finite time, and the slower color has a positive chance to explode earlier than the faster color.

Then, in \cite{BarHofKom14} we treated \emph{fixed speed} competition on the configuration model with i.i.d. degrees satisfying \eqref{eq::F}, when the speeds of the two colors are not equal. Suppose it takes $1$ and $\la>1$ unit of time for red and blue to spread across an edge, respectively. We have shown that in case the two colors start to spread from uniformly chosen single source vertices, the red color paints $n-o_{\Pv}(n)$ vertices, while blue paints
\[ \exp \left\{  (\log n)^{2/(\la +1)} g_n(\Yrn, \Ybn)\right\}\]
many vertices, where $g_n(\Yrn, \Ybn)$ is a random variable that shows $\log\log$-periodicity, and can be rewritten in a similar ($\la$-dependent) form as the one in the  statement of Theorem \ref{thm::main}.

Recently Cooper \emph{et al.} \cite{CooElsOgiRad15} have analysed a similar fixed speed competition model, also on the configuration model with power-law degrees with exponent $\tau\in(2,3)$. In what they call Model 2, one of the colors (say red) is called a `malicious information' (i.e., it might model the spread of a virus), while the other color (say blue) `immunizes' vertices. The main difference from the spreading rules of this paper is that in their model the infection does not spread to all the neighbors of a vertex, only a fixed subset of the nodes. Further, the immunization process starts from the not-yet infected neighbors of infected vertices (with a delay), and spreads then in a similar manner as the blue color in this paper. This implies a dependence between the two processes beyond the obvious `blocking' effect: if the red color paints a high-degree vertex (a hub), the blue color automatically reaches these hubs as well.
The authors show that in this competing scheme, the immunization process can block the spread of the infection, i.e., the infection can only spread to $o(n)$ many vertices.

Finally, this paper finishes the description of fixed speed competition when the speeds are equal.
Let us here compare the results  to those in \cite{BarHofKom14}.
 Theorem \ref{thm::main} and the first part of Theorem \ref{thm::maxdegree} correspond to  \cite[Theorem 1.2]{BarHofKom14} and \cite[Theorem 1.4]{BarHofKom14}, respectively, and they could be interpreted as follows: if $\Ybn/\Yrn< \tau-2$, then the local neighbourhoods of the source vertices differ enough to build up a significant difference between the spread of the two colors. Hence, both the maximal degree and the total number of vertices occupied by blue can be obtained by substituting $\la=1$ in the formulas in \cite[Theorem 1.4]{BarHofKom14} and \cite[Theorem 1.2]{BarHofKom14}. However, if $\Ybn/\Yrn> \tau-2$, then the growth of the two colors does not differ enough, hence, both colors can paint a positive proportion of the hubs. Once the highest degree vertices are coloured, their coloring `rolls down' to smaller and smaller degree vertices to eventually occupy the whole graph in a manner that maintains the fact that `significantly many' of the vertices are coloured both blue and red. Eventually, there is asymptotic coexistence in the model.

\subsubsection*{Open problems}
 The analysis of competition on the configuration model is far from complete. One can for instance ask about different spreading dynamics (edge lengths) and different power-law exponents. Further, one can ask what happens if the colors have entirely different passage time distributions (e.g.\ one is explosive and the other is not), or what happens if one of the colours have a main advantage by starting from one or many initial vertices of very high degree. These can correspond to e.g.\ competition advantage of different product on the network or to different marketing strategies.

Here we list some conjectures for competition on $\CMD$ with i.i.d. power law degrees of distribution $D$ with exponent $\tau$.
We further assume that the time to passage times can be represented as i.i.d. random variables on edges, from distribution $I_r$, $I_b$ for red and blue, respectively.

\emph{Uniformly chosen source vertices, $\tau\in (2,3)$:}

A.  If the spreading dynamics are so that the underlying
branching processes defined by $D^\star,I_r$ and $D^\star, I_b$ are both explosive, then we conjecture that there is \emph{never} coexistence and either of the two colors can win. We conjecture that the number of vertices the losing color can paint depends on the behavior of density of the explosion time around $0$.  A step towards proving this is to understand typical distances in $\CMD$ for arbitrary i.i.d. edge weights: this is done in an upcoming paper \cite{BarHofKomdist}.

B. If the underlying BP for one color is explosive while the other one is not, than we suspect that the explosive one always wins and the number of vertices the other color paints is tight.

C. If the edge weights are \emph{separated away from $0$}, in the sense that they can be written in the form $c+X$ for some random variable $X\ge 0$ and some constant $c>0$, then we conjecture that the different speed case  ($\la >1$) will be similar to the results in \cite{BarHofKom14}: the faster color wins. We conjecture that even the number of vertices that the slower color paints should be the same as the result given in \cite{BarHofKom14} as long as the fluctuation of typical distances are \emph{tight} around $c \cdot 2 \log \log n/ |\log (\tau-2)|$. We investigate typical distances in this setting in the upcoming paper \cite{BarHofKomdist} and tightness in a subsequent paper. Tightness depends sensitively on the precise behavior of the random variable $X$ around the origin.

D. If the edge weights are so that the support of the distribution contains an interval $[0,\ve]$ for some $\ve>0$, and the underlying BP is not explosive, then even typical distances in the graph are not understood. This is mainly due to a lack of literature about conservative infinite mean BPs: a precise understanding of the time it takes in the BP to reach $m$ individuals would be necessary.

\emph{Special source vertices, $\tau \in (2,3)$:}

It would be interesting to study scenarios where at least one of the sources has a `big head start' in the sense that it is started from a vertex with degree that grows with $n$. Since the number of initial half-edges is the important parameter here, this problem is essentially equivalent to starting from multiple source vertices, where the number of source vertices grows with $n$. The first question that we might ask: how many half-edges are needed for a process to start from to guarantee the winning of that color? If it does not have the necessary amount of initial half-edges for winning, what is the number of painted vertices, and how does it depend on the initial size of the source set and on the speeds of the two colors?

Based on the results of this paper, it is reasonable to conjecture that at least in the fixed speed setting, if the slower color can occupy all the highest degree vertices  (degrees larger than $n^{(\tau-2)/(\tau-1)}$), then it blocks the way of the other color from spreading, and as a result, it might flip the outcome and paint almost all vertices. If this is not the case, but the slower color has a head start, then similar estimates such as the maximal degree it can paint and the number of half-edges with this degree should determine the size it can eventually paint.

\emph{Uniformly chosen sources, $\tau>3$:}
As mentioned above, our conjecture for $\tau>3$ is mostly based on the result in \cite{ADMP11}.

 We suspect that if the transmission times $I_r, I_b$ both have continuous distribution, and the branching process approximations of them have different Malthusian parameters, then there is no coexistence, and the number of vertices painted by the slower color is $n^{\beta}$ for some $\beta\in(0,1)$.
When the Malthusian parameters agree, we suspect that there is asymptotic co-existence.

\emph{Uniformly chosen sources, $\tau=3$:}

In this case $\Pv(D>x) = L(x)/ x^2$, with $L(x)$ a slowly varying function at infinity. We suspect that $L(x)$ and the transmission distributions $I_r, I_b$ jointly determine what happens:
 In the case when one or more of the BPs might be explosive, we suspect that the outcome is similar to cases A and B above. The non-explosive cases might be harder, at least for the case when $L(x)$ is so that the $\Ev[D^\star] = \infty$.

\subsection{Overview of the proof and structure of the paper}
The heuristic idea of the initial parts of proof is the same as for the $\la>1$ case.  The main idea is  to extensively use the fact that in the configuration model, half-edges can be paired in an arbitrarily chosen order. This allows for a joint construction of the graph with the growing of the two colored clusters. The growth has six phases, out of which the first one (Section \ref{sc::climbup}) is essentially the same as for $\la>1$. In Section \ref{sc::peak}, the two cases, i.e., $\Ybn/\Yrn<\tau-2$ or $\Ybn/\Yrn>\tau-2$, separate and the proof of coexistence in the latter case is entirely new. The methodology of the proof for Theorem \ref{thm::main} for the $\Ybn/\Yrn<\tau-2$ remains in essence the same as the proof of \cite[Theorem 1.2]{BarHofKom14}, with some adjustments needed to handle larger error terms due to $\la=1$ instead of $\la>1$.
Wherever we can, we try to keep the overlap with \cite{BarHofKom14} minimal, and for a more detailed overview of the methodology for $\Ybn/\Yrn<\tau-2$ we refer the reader to \cite[Section 1.3]{BarHofKom14}.

We use the shorthand notation $X_n \simp n^{a}$ if there exists a constant $b\ge 0$ such that $\Pv( X_n \in ( (\log n)^{-b} n^a, (\log n)^b n^a) ) \to 1$. We call vertices with degree at least $\simp\!n^{(\tau-2)/(\tau-1)}$ \emph{hubs}.
\begin{enumeratei}
\item \label{ph::bp} \emph{Branching process phase.}\\  We couple the initial stages of the growth to two independent branching processes. The coupling fails when  one of the colors (wlog we assume it is red)
    reaches  size $n^\vr$ for some $\vr>0$ sufficiently small.
 \item \label{ph::montain_up}\emph{Mountain climbing phase.}\\
    After the coupling fails, we build a path through higher and higher degree vertices to a vertex with degree at least $\simp n^{(\tau-2)/(\tau-1)}$. The length of this path is of constant order. We denote the total time to reach such a vertex by red by $T_r$, and the time it would take for blue (if red would not be present at all) by $T_b$. While doing so, we arrange the vertices in layers of nested sets that can be though of as level sets of a (imaginary) mountain where the height function is linear in the $\log \log$( degree), hence the name of the phase.

 \item \label{ph::peak}\emph{Crossing the peak of the mountain.} \\
 The degree of the  maximal degree vertex in the graph is $\simp\!n^{1/(\tau-1)}$, and vertices of approximately this degree form a complete graph whp. Hence, when red occupies one of these hubs, in the next step it occupies all of them. We very carefully handle how red crosses the peak. If blue is at much lower degrees at time $T_r$ (meaning, $T_b\ge T_r+2$), then the proof follows the same method as for the $\la>1$ case, and red will occupy all vertices with degree higher than  $\simp\!n^{(\tau-2)^{\dnr}/(\tau-1)}$ for some $\dnr \in (0,1)$.

 On the other hand, if $T_b=T_r$, then both colors arrive to the hubs at the same time, hence, in the next step, they arrive to almost all hubs at the same time, and hence, these vertices are painted with equal probability red and blue, respectively. Hence, we get Theorem \ref{thm::maxdegree} part 2(a).
 Further, if $T_b=T_r+1$, then at time $T_r+1$ red can cross the peak, and occupies almost all vertices with degree higher than $\simp\! n^{(\tau-2)^{\dnr}/(\tau-1)}$ for some $\dnr \in (0,1)$, while blue occupies a few vertices of degree higher than $\simp\!n^{(\tau-2)/(\tau-1)}$, leading to Theorem \ref{thm::maxdegree} part 2(b), 2(c).
  The question at this point becomes what happens at time $T_r+2$. We say that at time $T_r+2$, blue can also `cross over' the already red peak, if there is a $\dnb>\dnr$ so that blue occupies approximately half of the vertices between $\simp\! n^{(\tau-2)^{\dnb}/(\tau-1)}$ and $\simp\! n^{(\tau-2)^{\dnr}/(\tau-1)})$. We show that this happens if and only if $\Ybn/\Yrn> \tau-2$.
 From here, the proofs separate for $\Ybn/\Yrn>\tau-2$ and  $\Ybn/\Yrn<\tau-2$.

 As a side-result, the analysis of the crossing of the mountain phase leads to the proof of Theorem \ref{thm::distances}.
 Note that for typical distances, it is not necessary to let red and blue grow simultaneously. Hence, we let red grow $T_r$ steps, blue $T_b$ steps, show that they whp do not meet until this time, but they reach the top of the mountain, and then the same analysis as the one for crossing the peak shows that the typical distance is either $T_r+T_b+1$ or $T_r+T_b+2$.

 \end{enumeratei}
 The proof for $\Ybn/\Yrn<\tau-2$:
 \begin{enumerate}
 \item[(iv)(1)] \label{ph::mountain_down}\emph{Red avalanche from the peak.}\\
   After crossing the mountain, red starts occupying all vertices of less and less degree. We call this the \emph{avalanche-phase of red}. This is similar to the $\la>1$ case.
   \item[(v)(1)] \label{ph::meeting}\emph{At the collision time.}\\
Meanwhile, blue did its mountain climbing phase as well, and at time $T_r+1$, it is exactly $T_b-T_r-1$ many steps away from reaching the top of the mountain.
Then, approximately at time $T_r + 1 + (T_b- T_r-1)/2$ the red avalanche has sloped down to the same degree vertices as blue has climbed up to, hence they meet by arriving to vertices at the same time. Since this expression is not necessarily an integer, we have to investigate their meeting time\footnote{Note that this method provides and alternative proof for typical distances in Theorem \ref{thm::distances}. Of course, the two methods yield the exact same result, but the proof presented in Section \ref{sc::meetingtime} is much shorter.} and the maximal degree of blue more carefully.  \item[(vi)(1)] \label{ph::after_meeting}\emph{Competing with the avalanche.}\\
 After the meeting time, blue cannot occupy higher degree vertices anymore, since those are already all red.
 Note that at this time most of the graph is still not reached by any color.
 We estimate how many vertices blue can still paint  in two steps:  first we calculate the size of the `optional cluster of blue', i.e.\ we calculate the size of the $k$-neighborhood of blue half-edges via path-counting methods, yielding an upper bound.
Vertices in the optional cluster of blue are `close' to a blue half edge, hence, they will be blue unless they are occupied by red simply because they are `accidentally' also `close' to some red half-edge. In the second step we estimate the size of the intersection between the optional cluster of  blue and the red cluster. The two steps together provide matching upper and lower bounds for the number of vertices that blue occupies after the intersection.   \end{enumerate}

 The proof for $\Ybn/\Yrn>\tau-2$:
 \begin{enumeratei}
  \item[(iv)(2)]\label{ph::mountain_down-mixed}\emph{Mixed avalanche from the peak.}
   While most vertices close to the mountain-top are all red, there is an interval $(\dnr, \dnb)$ in the $\log\log$( degrees) that has `mixed' coloring (see the interval in (iii)). We show that this pattern `rolls down' the mountain, i.e., for each $m\le \nu \log\log n/|\log(\tau-2)|$ for some $\nu<1$, each vertex with degree in between  $\simp\! n^{(\tau-2)^{m+\dnb}/(\tau-1)}$ and $\simp\!n^{(\tau-2)^{m+\dnr}/(\tau-1)}$ is again painted red and blue with equal probability. We stop this process after approximately $\nu \log\log n/|\log(\tau-2)|$ steps and denote the degree of vertices `at the bottom of' the last colored interval by $Q$.
    \item[(v)(2)] \label{ph::meeting}\emph{Coloring the neighbourhood of a random vertex.}\\
    To determine the proportion of blue vertices, we look at a uniform random vertex $w$. We couple its local neighborhood in the graph to a branching process that is a copy of the BPs described in the branching process phase.
We introduce the random stopping time that is the number of generations needed for this BP to reach at least one colored vertex with degree at least $Q$. We look at the vertices in the last generation of the stopped BP with degree higher than $Q$. We color such a vertex red or blue with equal probability if its degree falls into a mixed interval, and color it red if it falls into an all-red interval, providing a partial coloring of the last generation of the stopped BP.

 \item[(vi)(2)] \label{ph::random bootstrap}\emph{Random bootstrap percolation to the root.}\\
 After the last generation of the BP tree with root $w$ has been partially colored, we do the following recursive procedure to (partially) color the earlier generations of the BP:
 if a vertex has children of both colors, then we paint it blue and red with equal probability. If it has children of only one color, then we paint it in that color deterministically. If it has no colored children, then it stays uncolored. It is not hard to see that this is exactly the procedure how the two colors  reach the local neighbourhood of a vertex $w$.
 We show that in this random bootstrap procedure, the root gets color blue or red each with strictly positive probability (summing up to $1$) as $n\to \infty$. A second moment method - investigating two vertices instead of only one - finishes the proof of coexistence.

  \end{enumeratei}

\subsubsection*{Notation}
We write $[n]$ for the set of integers $\{1,2,\dots, n\}$. We denote by the same name and add a superscript $(r), (b)$ to random variables, sets or other quantities belonging to the red and blue processes, respectively.  We write $E(\CMD)$ for the set of edges in $\CMD$. Note that multiple edges might also occur. For any set of vertices $S\subset [n]$, we write $N(S)$ for the set of their neighbors, i.e.,
\be \label{def::ns}N(S)=\{y\in [n]: \exists x\in S, (x,y) \in E(\CMD)\},\ee
where $(x,y)$  might not be unique if there are multiple edges between a vertex $x\in S$ and $y\in N(S)$.
 For any event $A$, $\Pv_n(A):=\Pv(A| D_1, D_2, \dots, D_n)$. As usual, we write i.i.d.\ for independent and identically distributed, lhs and rhs for left-hand side and right-hand side. We write $\lfloor x\rfloor, \lceil x \rceil$ for the lower and upper integer part of $x\in \R$, and $\{x\}$ for the fractional part of $x\in \R$. Slightly abusing the notation, we use curly brackets around set elements, events and long exponents as well.

 We use $\toindis, \toinp, \toas$ for convergence in distribution, in probability and almost surely, respectively. We use the Landau symbols $o(\cdot), O(\cdot), o_{\Pv}(\cdot), O_{\Pv}(\cdot)$ in the usual way.

 We say that a sequence of events $\mathcal E_n$ occurs with high probability (whp) when $\lim_{n\to \infty}\Pv(\mathcal E_n) = 1.$ In this paper, constants are typically denoted by $c$ in lower and $C$ in upper bounds (possible with indices to indicate which constant is coming from which bound), and their precise values might change from line to line. Typically, all the whp-events hold whp under the event $\{\CL_n\in [ \Ev[D]n /2, 2 \Ev[D] n] \}.$

\section{The branching process phase}\label{sc::BP}
In this section we briefly summarise the key idea of this phase. In the construction of the configuration model, we start pairing the half-edges in an arbitrary order, and each time we pair a half-edge, we can pick the next one to pair as we wish. Hence, we can do the pairing in an order that corresponds to the spread of the two colors. More precisely, first we pair all the outgoing half-edges from the source vertices (time $t=1$), then we pair the outgoing half-edges from the neighbors of the source vertices (time $t=2$), and so on, in a breadth-first search manner. Whenever we finish pairing all the half-edges attached to vertices at a given graph distance from the source vertices, we increase the spreading time in the competitive coloring process by $1$. This process of joint construction of the competition and graph building is often called the \emph{exploration process} in the literature.

The key idea is that cycles in the exploration process are improbably as long as the total number of half-edges attached to colored vertices is small. Hence, the initial stage of the exploration can be coupled to a random tree, i.e., a branching process. For more details on the exploration process specific to this particular model, see \cite[Section 2]{BarHofKom14}.

Let $B_{i}$ stand for the \emph{forward-degree} of the $i$-th colored vertex $v_{i}$ for $i\ge 2$ in this exploration process, i.e., the number of half-edges incident to $v_i$ that are not paired yet when $v_i$ is reached in the exploration process.  Since the probability of picking a half-edge that belongs to a vertex with degree $j+1$ is approximately equal to $(j+1)\Pv(D=j+1)/\Ev[D]$, we get the size-biased distribution \eqref{def::size-biased1}
 as a natural candidate for the forward degrees of the vertices $v_i$ in the exploration process.

More precisely, \cite[Lemma 2.2]{BarHofKom14}  based on \cite[Proposition 4.7]{BHH10} states that the number of vertices and their forward degrees in the exploration process can be coupled to i.i.d.\ degrees having distribution function $F^\star$ from \eqref{def::size-biased1}, as long as the total number of vertices of the colored clusters does not exceed $n^{\vr'}$ for some small $\vr'>0$.

An immediate consequence of this lemma is that locally we can consider the growth
 of $\mathcal{R}_{t}$ and $\mathcal{B}_t $ as independent branching processes $(Z_{k})_{k>0}$  with offspring distribution $F^\star$ for the second and further generations, and with offspring distribution given by $F$ for the first generation.

The following theorem by Davies \cite{D78} describes the growth rate of a similar branching process:

\begin{theorem}[Branching process with infinite mean \cite{D78}]\label{thm::davies} Let $\wit Z_k$ denote the $k$-th generation of a branching process with offspring distribution given by the distribution function $F^\star$. Suppose there exists an $x_{0}>0$ and a function $x\mapsto\kappa(x)$ on $\R^+$ that satisfies the following conditions:
\begin{enumeratei}
\item  $\kappa(x)$ is non-negative and non-increasing,
\item $x^{\kappa(x)}$ is non decreasing,
\item $\int\limits_0^\infty\kappa\left(\mathrm e^{\mathrm e^x}\right)\mathrm d x<\infty$.
\end{enumeratei}
 Let us assume that the tail of the offspring distribution satisfies that for some $\tau\in(2,3)$ and for all $x\ge x_0$,
\begin{equation}\label{eq::davies_cond}
    x^{-(\tau-2)-\kappa(x)}\leq 1-F^\star(x)\leq x^{-(\tau-2)+\kappa(x)}.
\end{equation}
Then $(\tau-2)^{k}\log(\wit{Z}_{k}\vee 1)$ converges almost surely to a random variable $\wit Y$. Further, the variable $\wit Y$ has exponential tails: if $J(x):=\Pv(\wit Y\le x)$, then
\be\label{eq::exp-tails-Y}   \lim_{x\to \infty} \frac{- \log (1-J(x))}{x} =1.\ee

 \end{theorem}

It is an elementary calculation to check that $F^\star$ in \eqref{def::size-biased1} satisfies the criterions of this theorem, and a simple modification yields that the same convergence holds for the branching process where the size of the first generation has distribution $F$ instead of $F^\star$, see \cite[Lemma 2.4]{BarHofKom14}.
Hence, we get that there exists a random variable $Y$ with exponentially decaying tail, such that this BP satisfies
\be\label{eq::limitY} \lim_{k\to \infty} (\tau-2)^k \log Z_k \toas Y,\ee
since $Z_k\ge 1$ a.s., so $Z_k\vee 1 = Z_k$.

The random variables defined in Definition \ref{def::limit-variables} are finite time approximations of two independent copies of $Y$, denoted by $Y_r$ and $Y_b$ (standing for red and blue), respectively, at the time when the number of vertices colored by one of the colors reaches size $n^{\vr}$ for some small $\vr>0$. From now on, and without loss of generality, we assume that this color is red.



\section{Mountain-climbing phase}\label{sc::climbup}
In this section we briefly summarise the key ideas in  the mountain-climbing phase and recall notation that will be needed later on. For a more detailed description we refer the reader to \cite[Section 3]{BarHofKom14}. Again, we assume that red is the color that reaches size $n^\vr$ first.

Recall that the coupling of the number of vertices in the BP and in the growing cluster of the colors fails when one of the colors reaches $n^{\vr'}$ many vertices. For our purposes later on, we rather want to guarantee that the number of vertices in the last generation of the BP is \emph{at least} some power of $n$. We can assure this by stopping the BP a bit earlier: let us first set some $\vr<\vr' (\tau-2)$ and define
\[ t(n^{\vr})=\inf \{k:  Z_k^{\sss{(r)}}  \ge n^{\vr} \}.\]
Recall Definition \ref{def::limit-variables}, i.e.,
\be \label{def::Y_r*}\Yrn:=(\tau-2)^{t(n^{\vr})}\log Z_{t(n^{\vr})}. \ee
An elementary rearrangement yields that, taking $\Yrn$ as given and with $\{x\}=x-\lfloor x \rfloor$,
\be\label{eq::an2} t(n^{\vr}) = \frac{\log(\vr/\Yrn) + \log\log n}{|\log(\tau-2)|}+ 1-a_n^{\sss{(r)}},\ee
where
 \be\label{eq::an} a_n^{\sss{(r)}}=  \left\{ \frac{\log(\vr/\Yrn) + \log\log n}{|\log(\tau-2)|}\right\}.  \ee
Note that $1-a_n^{\sss{(r)}}$ is there to make the expression on the rhs of $t(n^{\vr})$ equal to its upper integer part. Due to this effect, the last generation has a bit more vertices than $n^{\vr}$, so let us introduce the notation $\vr^{\sss{(r)}}$ for the random exponent of the overshoot
\be \label{eq::rho_0} Z_{t(n^{\vr})}^{\sss{(r)}}= n^{ \vr (\tau-2)^{a_n^{\sss{(r)}}-1} }:= n^{\vr^{\sss{(r)}}},\ee
We get this expression by rearranging \eqref{def::Y_r*} and using the value $t(n^{\vr})$ from \eqref{eq::an2}.
 The property $\vr< \vr'(\tau-2)$ and $a_n^{\sss{(r)}} \in [0,1)$ implies that $\vr^{\sss{(r)}}< \vr'$, which in turn guarantees that the coupling is still valid, i.e., we can also couple the \emph{degrees of vertices} in the $t(n^{\vr})$th generation of the branching process to i.i.d.\ size-biased degrees.

The second step in this section is to decompose the high-degree vertices in the graph into the following sets, that we call \emph{layers}:
\be\label{def::Gamma_i} \Gamma_i^{\sss{(r)}}:=\{ v: D_v>u_i^{\sss{(r)}} \}, \ee
where $u_i^{\sss{(r)}}$ is defined recursively by
\be\label{eq::ui_recursion} u_{i+1}^{\sss{(r)}} =\left(\frac{u_{i}^{\sss{(r)}}}{C\log n}\right)^{1/(\tau-2)}, \ \ \qquad u_0^{\sss{(r)}}:= \bigg(\frac{n^{\vr^{\sss{(r)}}}}{C\log n}\bigg)^{1/(\tau-2)} \ee
for a large enough constant $C>0$ (e.g. $C= 8/c_1$ is sufficient, where $c_1$ is introduced in \eqref{eq::F}). It is not hard to see that
\be\label{def::ui} u_{i}^{\sss{(r)}}= n^{\vr^{\sss{(r)}} (\tau-2)^{-(i+1)}} (C\log n)^{-e_i}\quad \mbox{with} \quad
 e_i =  \frac{1}{3-\tau}\bigg( \Big( \frac{1}{\tau-2}\Big)^{i+1}-1\bigg). \ee
Note that since $(\tau-2)^{-1}>1$, $u_i^{\sss{(r)}}$ is growing, hence $\Gamma_0 \supset \Gamma_1 \supset \Gamma_2 \supset \dots$.

The third step is to show that $Z_{t(n^{\vr})}$ has a nonempty intersection with the initial layer $\Gamma_0$.
The proof of this is based on the following lemma, that we cite here, since we repeatedly use it later on.
\begin{lemma}\label{lem::maxdegree}
Let $X_i, \ i=1, \dots, m$ be i.i.d.\  random variables  with power-law exponent $\alpha$, i.e., the distribution function of $X_i$ satisfies \eqref{eq::F} with $\tau-1$ replaced by any $\alpha>0$.
Then for any $K>0$
\be\label{eq::logwhp} \Pv\bigg(\max_{i=1,\dots, m} X_i < \Big(\frac{ m}{K\log m}\Big)^{1/\alpha} \bigg) \le \frac{1}{m^{c_1K}}, \ee
where $c_1$ is introduced in \eqref{eq::F}.
\end{lemma}
Now we can apply this lemma since the last generation in the BP is an i.i.d.\  collection of $n^{\vr^{\sss{(b)}}}$ many power-law random variables, and the maximum of these behaves approximately as  $n^{\vr^{\sss{(r)}}/(\tau-2)}$. The definition of $u_0^{\sss{(r)}}$ adds a logarithmic correction term to this, hence, whp there is a vertex with degree at least $u_0^{\sss{(r)}}$ in the last generation of the BP.  For more details, see \cite[Lemma 3.1]{BarHofKom14}.

The fourth step is to show the existence of a red path from $\Gamma_0\cap \CR_{t(n^{\vr})}$ to a vertex that has degree at least $\simp\!n^{(\tau-2)/(\tau-1)}$.
The existence of such a path is guaranteed by  \cite[Lemma 3.4]{BarHofKom14}, based on simple concentration of binomial random variables, stating that whp
\be\label{eq::gamma-i-connectivity} \forall i\ \  \Gamma_i \subset N(\Gamma_{i+1}), \ee
where recall that  $N(S)$ stands for the set of neighbors of $S$.
\begin{figure}\label{fig::mountain}
\includegraphics[width=0.5\textwidth]{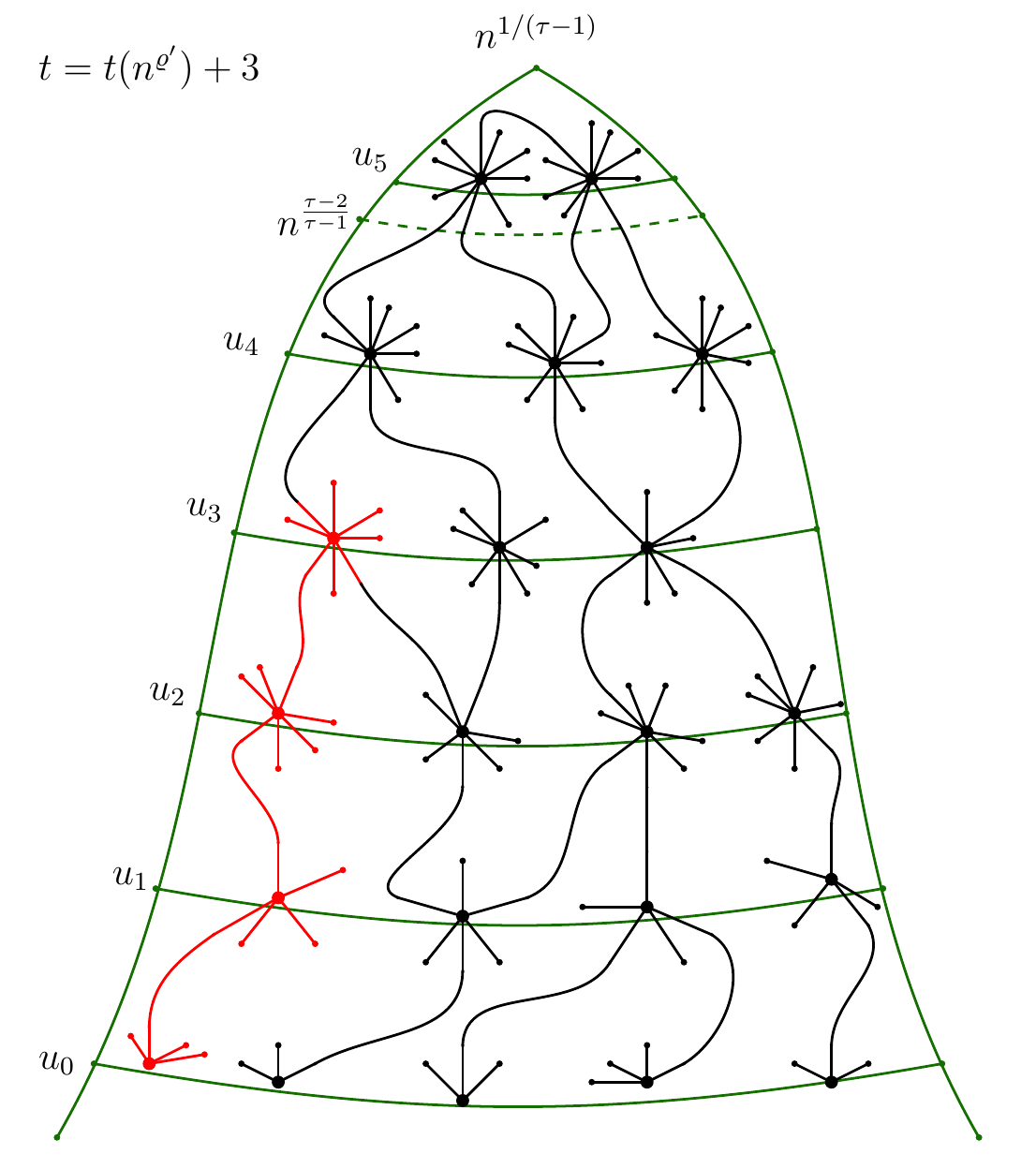}
\caption{An illustration of the layers and the mountain climbing phase at time $t(n^{\vr})+3$. Disclaimer: the degrees on the picture are only an illustration.}
\end{figure}

With \eqref{eq::gamma-i-connectivity} in hand, we can determine how long it takes to climb up through the layers $\Gamma_i^{\sss{(r)}}$ to the highest-degree vertices in the graph.  Lemma \ref{lem::maxdegree} with $X_i=D_i\sim F$, $\alpha=\tau-1$ shows that the maximal degree in $\CMD$ is of order $n^{1/(\tau-1)}$. We write $i_{\star\sss{(r)}}$ for the last index\footnote{We put the brackets $(r)$ now in the subscript only because we want to avoid that quantities with both subscripts and superscripts look messy: compare $u_{ i_{\star\sss{(r)}}}^{\sss{(r)}}$ to $ u_{ i_{\star}^{\sss{(r)}}}^{\sss{(r)}}$.} when $\Gamma_{i}^{\sss{(r)}}$ is whp nonempty, i.e.,
\be\label{def::i*}i_{\star\sss{(r)}}:=\inf \{ i: u_i^{\sss{(r)}} \le n^{1/(\tau-1)} < u_{i+1}^{\sss{(r)}}\}. \ee

An easy calculation using \eqref{def::ui} shows that
\be\label{eq::value_i*} i_{\star \sss{(r)}}= -1+ \frac{-\log ((\tau-1)\vr^{\sss{(r)}})}{|\log(\tau-2)|}-\bnr,
\quad \mbox{ with }\quad
 \bnr=\left\{ \frac{-\log ((\tau-1)\vr^{\sss{(r)}})}{|\log(\tau-2)|}\right\}. \ee
Using the value of the overshoot exponent $\vr^{\sss{(r)}}$ in \eqref{eq::rho_0} and then the value $a_n^{\sss{(r)}}$ in \eqref{eq::an}, plus the fact that $\{x - 1+\{y\}\}=\{x+y\}$,  we get that
\be\label{eq::bn} \bnr = \left\{ \frac{-\log ((\tau-1)\vr)}{|\log(\tau-2)|}+ a_n^{\sss{(r)}}-1\right\} = \left\{ \frac{-\log ((\tau-1)\Yrn)+\log \log n}{|\log(\tau-2)|}\right\}.\ee
From \eqref{def::ui} one can easily calculate that
\be\label{eq::ui*} \ba u_{i_{\star\sss{(r)}}}^{\sss{(r)}}&=n^{\frac{(\tau-2)^{\bnr}}{\tau-1}} (C \log n)^{-e_{i_{\star\sss{(r)}}}},  \quad \mbox{ with} \\
e_{i_{\star\sss{(r)}}}&=\frac{1}{3-\tau}\left(\frac{(\tau-2)^{\bnr}}{(\tau-1)\vr^{\sss{(r)}}} -1\right)\le\frac{1}{(3-\tau)}\left(\frac{1}{(\tau-1)\vr^{\sss{(r)}}}-1\right). \ea \ee
In what follows, we will often use the total time to reach the top, so let us introduce the notation
\be\label{eq::k*+i*}  T_r:=t(n^{\vr})+i_{\star\sss{(r)}}=\frac{\log\log n-\log \left((\tau-1) \Yrn\right)}{|\log(\tau-2)|} -1-\bnr,\ee
which only depends  on $\vr$ through the approximating $\Yrn$, and $\bnr$ is exactly the fractional part of the expression on the rhs of $T_r$. Since also $\Yrn\toindis Y_r$ irrespective of the choice of $\vr$, this establishes that the choice of $\vr$ is not relevant in the proof.

Note that if color red would not be present in the graph, we could repeat the same procedure for blue, yielding the definitions \eqref{def::ti-bi}.  It is not hard to see using \eqref{eq::an2} that
\be\label{eq::BPblue} Z_{t(n^{\vr})}^{\sss{(b)}} = n^{ \vr (\tau-2)^{a_n^{\sss{(r)}}-1} \Ybn/\Yrn },\ee
hence, using $\vr^{\sss{(b)}}:=\vr (\tau-2)^{a_n^{\sss{(r)}}-1} \Ybn/\Yrn$ instead of $\vr^{\sss{(r)}}$, we can define
\be\label{eq::b-definitions}\ba u_0^{\sss{(b)}}&:= Z_{t(n^{\vr})}^{\sss{(b)}} = n^{ \vr^{\sss{(b)}} }, \qquad \ \ \
 u_{i+1}^{\sss{(b)}}:=\left(\frac{u_{i}^{\sss{(b)}}}{C\log n}\right)^{1/(\tau-2)}, \\
 \Gamma_i^{\sss(b)}&:= \{ v: D_v>u_i^{\sss{(b)}} \},  \qquad i_{\star \sss{(b)}} := -1+ \frac{-\log ((\tau-1)\vr^{\sss{(b)}})}{|\log(\tau-2)|}-\bnb\ea \ee
 where $ \bnb=\left\{ \frac{-\log ((\tau-1)\Ybn)+\log \log n}{|\log(\tau-2)|}\right\}$ similarly as in \eqref{eq::bn}.

Note also that establishing the existence of a path to the top of the degree mountain only provides an \emph{upper bound} on how long it takes for each color to reach the top. However, with a bit of work we can turn these bounds into a \emph{matching lower bound} as well. This will be relevant later, since it also determines \emph{an upper bound on the maximal degree} of the colors in their climbing phase as well as the number of vertices in each layer they occupy.

Similarly as in \eqref{def::Gamma_i} and \eqref{eq::ui_recursion}, let us define, for $j=r,b$,
\be\ba\label{eq::uibar}
   \widehat u_0^{\sss(j)}&:= (Z^{\sss{(j)}}_{t(n^{\vr})}\cdot C \log n)^{1/(\tau-2)}, \\
  \widehat u_{i+1}^{\sss{(j)}}&:= (\widehat u_i^{\sss{(j)}}\cdot C \log n)^{1/(\tau-2)},  \\
\widehat\Gamma_i^{\sss{(j)}}&:=\{v \in \CMD: d_v \ge \widehat u_i^{\sss{(j)}}\},
\ea\ee

Note that $\widehat\Gamma_i^{\sss{(j)}}$ grows faster than $\Gamma_i^{\sss{(j)}}$ since there is always an extra $(C\log n)^2$ factor causing an initial `gap' of order $(\log n)^2$ between $u_0^{\sss{(r)}}, \widehat u_0^{\sss{(r)}}$ and `opening up' as $i$ gets larger.

The next lemma handles the upper bound on the maximal degree of red or blue at any time $t(n^{\vr})+i$, but first some definitions. We say that a sequence of vertices and half-edges $(\pi_0, s_0, t_1, \pi_1, s_1, t_2,  \dots,  t_k, \pi_k)$ forms a \emph{path} in $\CMD$, if for all $0< i\le k$, the half edges $s_i, t_i$ are incident to the vertex $\pi_i$ and $(s_{i-1}, t_i)$ forms an edge between $\pi_{i-1},\pi_i$.
Let us denote the vertices in a path starting from a half-edge in $Z^{\sss{(j)}}_{ t(n^{\vr})}$ by $\pi_0, \pi_1, \dots $. We say that a path is \emph{good} if $\deg(\pi_i)\le\widehat u_i^{\sss{(j)}}$ holds for every $i$. Otherwise we call it \emph{bad}. We decompose the set of red and blue bad paths in terms of where they turn bad, i.e.\  we say that a bad path of color $j$ is belonging to  $\CB ad \CP_k^{\sss{(j)}}$ if it turns bad at the $k$th step (for $j=r,b$):
\be\label{def::badpaths} \ba  \CB ad\CP_k^{\sss{(j)}} := &\{ (\pi_0, s_0, t_1, \pi_1, s_1 \dots, t_k, \pi_k) \text{ is a path, } \\
 &\quad \pi_0\!\in\!\CC^{\sss{(j)}}_{ t(n^{\vr}) },\   \deg(\pi_i)\!\le\! \widehat u_i^{\sss{(j)}} \ \forall i\le  k-1,\ \deg(\pi_k)\!>\!\widehat u_k^{\sss{(j)}}   \},\ea\ee where $\CC^{\sss{(r)}}=\CR, \CC^{\sss{(b)}}=\CB$.
The following lemma tells us that the probability of having a bad path is tending to zero:
Note that $u_{i}^{\sss{(j)}}, \widehat u_{i}^{\sss{(j)}}$ and hence $\Gamma_i^{\sss{(r)}}, \widehat \Gamma_{i}^{\sss{(j)}}, i_{\star \sss{(j)}}$ are random variables that depend (only) on $n, \Yrn, \Ybn$. For a short notation for conditional probabilities and expectation let us introduce
\be\label{def::py-ey} \Pv_{Y,n}(\,\cdot\,):=\Pv(\,\cdot\, | \Yrn, \Ybn, n) \quad \Ev_{Y,n}[\,\cdot\,]:= \Ev[ \,\cdot\, | \Yrn, \Ybn, n].\ee

\begin{lemma}\label{lem::badpaths}
The following bound on the probability of having any bad paths holds for color $j=r,b$:
\be \Pv_{Y,n}( \exists k\le i_{\star\sss{(j)}}:  \CB ad\CP_k^{\sss{(j)}} \neq \varnothing) \le \frac{2}{C\log n}.\ee
\end{lemma}
\begin{proof}
The statement is a direct consequence of the proof of \cite[Lemma 5.2]{BarHofKom14} in the Appendix of that paper. In that proof, the estimate gets worse as the total number of layers grows, (which is of order $\log\log n$ there): here the total number of layers is bounded for color red and whp less then any increasing function for color blue.  More precisely, for red,  $\vr\le \vr^{\sss{(r)}}\le \vr(\tau-2)^{-1}$, $\bnr\in [0,1)$, hence $i_{\star \sss{(r)}}$ is bounded (see \ref{eq::value_i*}). For $i_{\star \sss{(b)}}$, pick any function $\omega(n)$ such that $\log \omega(n)=o(\log\log n)$ that tends to infinity (e.g. $\omega(n) = \log\log n$ will do). Then,
\[ \Pv( \vr^{\sss{(b)}}=\vr' (\tau-2)^{a_n^{\sss{(r)}}-1} \Ybn/\Yrn < \omega(n)^{-1}) \to 0 \]
as $n\to \infty$, since $a_n^{\sss{(r)}}\in [0,1)$. Then, this implies that $i_{\star \sss{(b)}} = -1+ \frac{-\log ((\tau-1)\vr^{\sss{(b)}})}{|\log(\tau-2)|}-\bnb$ is whp $o(\log\log n)$ and hence the proof of  \cite[Lemma 5.2]{BarHofKom14} is also valid for the blue process.   \end{proof}
An immediate consequence is the following:
\begin{corollary}\label{core::time-to-top-who}
For color red whp it takes time $T_r$ to reach a vertex with degree at least $n^{(\tau-2)/(\tau-1)}$, and for color blue it would take time $T_b$ whp to do so if red would not be present.
\end{corollary}
\begin{proof} We have seen that $T_j$, $j=r, b$,  is a whp an upper bound to reach the top (due to \eqref{eq::gamma-i-connectivity} yielding the existence of a path of length $T_j$). Now we argue that $T_j$ is also whp a lower bound to reach the top, that is, there is whp no path to the top shorter than $T_j$.

On the event $\{ \CB ad\CP_k^{\sss{(r)}} = \varnothing, \CB ad\CP_k^{\sss{(b)}} \neq \varnothing \}$, which occurs whp, we can use the upper bound $\widehat u_i^{\sss{(j)}}$ on the degrees at time $t(n^{\vr})+i$ for all $i\ge 0$, hence
we get that the time it takes to reach a vertex of  degree at least $n^{(\tau-2)/(\tau-1)}$ is at least
\[ \widehat i_{\star, \sss{(j)}} :=\inf \{ i: \widehat u_i^{\sss{(j)}} \le n^{1/(\tau-1)} \le \widehat u_{i+1}^{\sss(j)} \}. \]
Combine this with the fact that
\be\label{eq::wideui-ui} \widehat u_i^{\sss{(j)}} = u_i^{\sss{(j)}} (C\log n)^{\tfrac{2}{3-\tau}\Big( \big( \tfrac{1}{\tau-2}\big)^{i+1}-1\Big)},\ee and we get that $\widehat i_{\star, \sss{(j)}}=i_{\star, \sss{(j)}} $ whp.
Hence, $T_b$ and $T_r$ are both upper and lower bounds for this quantity.
\end{proof}
Later, we will need to estimate the number of vertices in each layer $\Gamma_i^{\sss{(j)}}$ at the time of occupation (without the presence of the other color).
Let us  denote the set and number of color $j$ vertices in the $i$th layer $\Gamma_i^{\sss{(j)}}$ right at the time when color $j$ (denoted by $\CC^{\sss{(j)}}$ below) reaches it, by
\be\label{def::Ai} \CA_i^{\sss{(j)}}:=\CC_{t(n^{\vr})+ i}^{\sss{(j)}}\cap\Gamma_i^{\sss{(j)}}, \qquad A_i^{\sss{(j)}}:=|\CA_i^{\sss{(j)}}|. \ee

\begin{lemma}\label{lem::numberofverticesinGamma}  For $j=r,b$, on the event $\{\CB ad\CP_k^{\sss{(j)}}=\varnothing \ \forall k\le  i_{\star \sss{(j)}} \}$, whp for all $ i\le i_{\star \sss{(j)}}$,
 \be \label{eq::aifinal} \ A_i^{\sss{(j)}}\le\exp\left\{ \log (C\log n) \cdot \frac{2 (\tau-2)^{-i}}{(3-\tau)^2} \right\}. \ee
\end{lemma}
\begin{proof}  The proof is based on concentration of binomial random variables, applied in two settings: (1) In the last generation of the BP approximation the degrees are i.i.d., hence the number of vertices with degree at least $u_0^{\sss{(j)}}$ is binomial\footnote{In the graph $\CMD$ is hypergeometric, but in the coupling to i.i.d degrees in the BP we simply have $Z_{t(n^{\vr})}^{\sss{(j)}}$ many i.i.d. variables with distribution $D^\star$, so the number of vertices with degree at least $u_0^{\sss{(j)}}$ is binomial in the BP}, and so whp there are only $A_0^{\sss{(j)}}:=C\log n$ many vertices with degree at least $u_0^{\sss{(j)}}$. (2) Given that there are $A_i^{\sss{(j)}}$ many color $j$ vertices in layer $\Gamma_i^{\sss{(j)}}$ with degree at most $\widehat u_i^{\sss{(j)}}$, the number of vertices that they connect to in $\Gamma_{i+1}^{\sss{(j)}}$ is stochastically dominated by a \[ {\sf Bin}\left(A_i^{\sss{(j)}} \widehat u_{i}^{\sss{(j)}}, \frac{\CE_{\ge u_{i+1}^{\sss{(j)}}}}{\CL_n(1+o(1))}\right)\]
random variable, where $\CE_{\ge y}$ denotes the total number of half-edges incident to vertices with degree at least $y$. More details are worked out in the proof of \cite[Lemma 5.4]{BarHofKom14}. 
\end{proof}

\section{Crossing the peak of the mountain}\label{sc::peak}
Next we investigate what happens when the path through the layers reaches the highest degree vertices. This is also the part where the two proofs for $\Ybn/\Yrn < \tau-2$ or $\Ybn/\Yrn >\tau-2$ separate. To be able to handle all the cases, we add the superscript $(r)$ and $(b)$ to quantities in the previous section, meaning that they belong to the description of the growth of the red and the blue cluster, respectively.

For red, we have just seen that $u_{i_{\star \sss{(r)}}}^{\sss{(r)}} \simp n^{\frac{(\tau-2)^{\bnr}}{\tau-1}}$, where the exponent is $ \in \left( \frac{\tau-2}{\tau-1}, \frac{1}{\tau-1}\right)$. Since the maximum degree in the graph is $\simp n^{\frac{1}{\tau-1}}$ whp, we have $\Gamma_{1+i_{\star \sss{(r)}}}^{\sss{(r)}}=\varnothing$.

Following the lines of \cite{BarHofKom14}, we cite the lemma from \cite[Volume II., Chapter 5]{H10}:
 \begin{lemma}\label{lem::direct_connect} Consider two sets of vertices $A$ and $B$. If the number of half-edges $\mathcal S_A=o(n)$ and $\mathcal S_B$ satisfy
 \[ \frac{\mathcal S_A \mathcal S_B}{n}> h(n),\]
 for some function $h(n)$, then conditioned on the degree sequence with $\CL_n\le 2\Ev[D] n$, the probability that the two sets are not directly connected can be bounded from above by
 \[ \Pv_n(A\nleftrightarrow B)< \e^{-\tfrac{h(n)}{4\Ev[D]}}.\]
  \end{lemma}
 \begin{proof} See the proof of \cite[Lemma 4.1]{BarHofKom14}.
 \end{proof}
 Let us introduce
 \be\label{def::alpha} \alpha_n^{\sss{(r)}}:=1-\frac{(\tau-2)^{b_n^{\sss{(r)}}}}{\tau-1}, \quad
 \beta_n^{\sss{(r)}}:=1+\frac{1}{(3-\tau)}\left(\frac{1}{(\tau-1)\vr^{\sss{(r)}}}-1\right), \ee
\be\label{eq::wideu1} \wur_1:= (C \log n) n/u_{i_{\star\sss{(r)}}}^{\sss{(r)}}=n^{ \anr} (C\log n)^{\beta_n^{\sss{(r)}}},\ee
and the following layer:
\be\label{def::witgamma}\wgar_1:=\{v\in \CMD, D_v > \wur_1\}.\ee
Then, we can apply Lemma \ref{lem::direct_connect} by picking $A$ to be a single red vertex with degree higher than $u_{i_{\star\sss{(r)}}}^{\sss{(r)}}$, and $B$ any vertex in $\wgar_1$. As a result, \emph{as long as blue does not interfere}, \cite[Lemma 4.2]{BarHofKom14} stays valid, stating that all the vertices in $\wgar_1$ are occupied by red at time $T_r+1$, i.e.,
\be\label{eq::slopestart}  \wgar_1\subset \mathcal R_{T_r+1} \quad \mbox{whp.}\ee
Now we investigate what happens if blue \emph{does} interfere with this stage.
To be able to do so, we define the similar quantities for blue as around \eqref{eq::BPblue}, and define
\be\label{eq::uistar-b} u_{i_{\star \sss{(b)}}}^{\sss{(b)}}:= n^{(\tau-2)^{\bnb}/(\tau-1)} (C\log n)^{-e_{i_{\star\sss{(b)}}}},  \ee
with $e_{i_{\star \sss{(b)}}}$ as in \eqref{eq::ui*} but $\vr^{\sss{(r)}}$ replaced by $\vr^{\sss{(b)}}$, and finally
\be \label{eq::wideu1-b} \wub_1:= (C\log n) n / u_{i_{\star\sss{(b)}}}^{\sss{(b)}} =: n^{\anb} (C\log n)^{\benb},\ee
where $\anb, \benb$ are as in \eqref{def::alpha}, but the superscript $(r)$ and $\vr^{\sss{(r)}}$ replaced by $(b)$ and $\vr^{\sss{(b)}}$, respectively.

First of all, note that blue would occupy some vertices with degree at least $n^{(\tau-2)^{\bnb}/(\tau-1)(1+o_{\Pv}(1))}$ at time $T_b$, and it is working its way towards this degree by increasing its maximal degree by a factor of $1/(\tau-2)$ in the exponent with each step. Hence, at time $T_r$, blue occupies some vertices with $\log$(degree)/$\log n$ that is
\[ \frac{(\tau-2)^{\bnb}}{\tau-1}(\tau-2)^{T_b-T_r}(1+o_{\Pv}(1)).\]
Recall that $D_n^{\max}(t)$ stands for the degree of the maximal degree vertex that blue paints at or before time $t$. We have to distinguish three cases.
\begin{description}
\item[\namedlabel{case::tb-tr>2}{$(1)$}] $T_b-T_r\ge2$. In this case, at time $T_r+1$, blue occupies some vertices with $\log$(degree)/$\log n$
that is \[ \frac{\log (D_n^{\max}(T_r+1))}{\log n} = \frac{(\tau-2)^{b_n^{\sss{(b)}}}}{\tau-1}(\tau-2)^{T_b-T_r-1}(1+o_{\Pv}(1)),\]
while red crosses the peak and  tries to occupy every vertex in $\wgar_1$ by \eqref{eq::slopestart}. It is not hard to see in \eqref{def::alpha} that $\anr>(\tau-2)/(\tau-1)$, hence
\[ \frac{\log \wur_1}{\log n} = \anr (1+ o_{\Pv}(1)) >  \frac{\log (D_n^{\max}(T_r+1))}{\log n}. \]
In other words, if $T_b\ge T_r+2$ then blue does not reach any red vertex at time $T_r+1$. Red can cross the peak, and \eqref{eq::slopestart} is valid.
\item[\namedlabel{case::tb=tr}{$(2)$}] $T_b-T_r=0$. In this case, red and blue arrive to the top of the mountain at the same time. By the assumption that $\Yrn>\Ybn$, we have the easy statement
\be\label{eq::frac-compare} \frac{\log\log n - \log ((\tau-1)\Yrn)}{|\log (\tau-2)|}-1 < \frac{\log\log n - \log ((\tau-1)\Ybn)}{|\log (\tau-2)|}-1 \ee
Note that $T_r$ and $T_b$ are the integer parts of these expression. Hence, we conclude that $T_b=T_r$ is only possible if the two expressions in \eqref{eq::frac-compare} differ by at most $1$. This translates exactly to $\Ybn/\Yrn>\tau-2$. Secondly, also note that \eqref{eq::frac-compare} can be rewritten as
\be\label{eq::key-1} T_r + \bnr < T_b + \bnb. \ee
Note that on the event $T_r=T_b$, at time $T_r+1$, both colors try to cross the mountain at the same time, and red tries to occupy every vertex with degree at least $n^{\anr(1+o_{\Pv}(1))}$, while blue tries to occupy every vertex with degree at least $n^{\anb(1+o_{\Pv}(1))}$.
Now, \eqref{eq::key-1} implies that if $T_r=T_b$, we must have $\bnr<\bnb$, which in turn implies that $\anr<\anb$.
Let us introduce $\delta_n^{\sss{(j)}}$ for $j=r,b$ implicitly by
\be\label{def::delta}  \alpha_n^{\sss{(j)}} =: \frac{(\tau-2)^{\delta_n^{\sss{(j)}}}}{\tau-1}. \ee
It is not hard to see from \eqref{def::alpha} that $\delta_n^{\sss{(j)}}\in(0,1]$, and $\dnr>\dnb$ when \eqref{eq::key-1} holds together with $T_r=T_b$.
Recall the rule that if red and blue arrive at a vertex at the same time, then the vertex gets each color with equal probability, independently of everything else. This rule, combined with the previous observation, implies that at time $T_r+1$, approximately half of the vertices with degree at least $\simp\! n^{\anb}$ are red and blue, respectively, while almost all\footnote{In fact, a few vertices might be blue even in this interval, but the proportion of these vertices is negligible. An estimate on how many vertices might be blue is given in Lemma \ref{lem::red-intervals} below. For the purpose of understanding, at this point it is enough that `almost all' vertices are red.} vertices between degree
$(\simp\!n^{\anr}, \simp\!n^{\anb})$ are red.
We introduce the shorthand notation for this as
\be\label{eq::shorthand}(0, \dnb] \in \text{Mix}, \quad (\dnb, \dnr]\in \CR. \ee  For an illustration see Fig.\ \ref{fig::cross-1}. \begin{figure}
\includegraphics[width=0.6\textwidth]{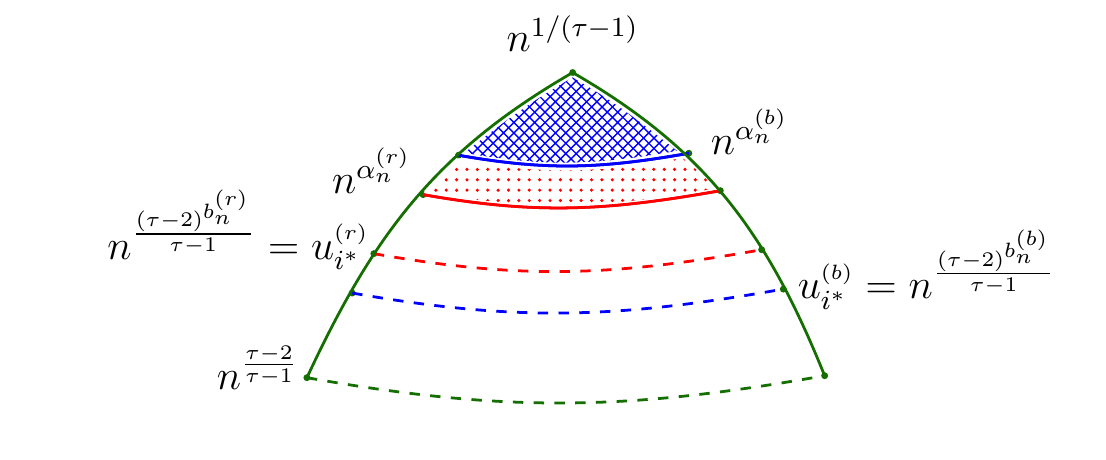}
\caption{Case (2): An illustration of the crossing of the mountain peak when $T_b=T_r$ at time $T_r+1$. The crosshatched (darker) and dotted (lighter) areas indicate the part of the mountain where vertices are red and blue with equal probability, and almost all red, respectively.}\label{fig::cross-1}
\end{figure}
\item[\namedlabel{case::tb-tr=1}{$(3)$}]\label{case::tb-tr=1}$T_b=T_r+1$. In this case, red arrives to the top of the mountain just one step before blue. At time $T_r+1$, blue occupies some vertices of $\log$(degree)/$\log n$ approximately $(\tau-2)^{\bnb}/(\tau-1)(1+o_{\Pv}(1))$.  Note that  by Lemma \ref{lem::numberofverticesinGamma} blue occupies only a few vertices of exactly this degree, but quite a few vertices slightly below this value, i.e., we can add a slightly larger logarithmic correction term at this step and have plenty of vertices that blue tries to color at time $T_r+1$. So, even if the degree of these vertices is higher than $\wur_1$ (by the equal probability upon same arrival rule), blue will whp occury some vertices around this value. Meanwhile, red occupies every other vertex with $\log$(degree)/$\log n$ at least $\anr$.
The question is what happens at time $T_r+2$ in this case.

Apply Lemma \ref{lem::direct_connect} for the blue color at time $T_r+1$: one can see that the blue vertices with $\log$(degree)/$\log n$ at least $(\tau-2)^{\bnb}/(\tau-1)(1+o_{\Pv}(1))$ are whp connected to \emph{every vertex} with degree higher than $n^{\anb}$ (the proof is identical to that of \eqref{eq::slopestart}). Here, there are two cases: either $\anb>\anr$, and then these vertices are already all red, hence blue cannot occupy more vertices. Or, $\anb<\anr$, and in this case blue and red arrives to vertices with $\log$(degree)/$\log n$ in the interval $[\anb, \anr)$ at the same time. Meanwhile, red occupies every vertex in  the interval $[(\tau-2)\anr, \anb)$.\footnote{This fact is verified in \eqref{eq::widetildegamm} below.}

To check which case of the two scenarios can happen we turn to \eqref{eq::frac-compare} and \eqref{eq::key-1}.
Since $T_r, T_b$ are the integer parts of the two sides of \eqref{eq::frac-compare}, $T_b=T_r+1$ can only happen if
$\Ybn/\Yrn> (\tau-2)^2$. Let us write shortly $q:=\Ybn/\Yrn$. There are two cases:
\begin{description}
\item[\namedlabel{casetb-tr=1,nc}{$(3)(a)$}] If $\Ybn/\Yrn \in ((\tau-2)^2, \tau-2]$, then $-\log q/|\log(\tau-2)| \in  [1,2).$ Hence, using $\Ybn=q\Yrn$, the right hand side of \eqref{eq::frac-compare} can be rewritten in two ways as
\be\label{eq::key-2} T_b+ \bnb = T_r + \bnr -\log q/|\log(\tau-2)|.\ee
Since the last term is between $1$ and $2$, $T_b=T_r+1$ is only possible if $\bnb=(-\log q/|\log(\tau-2)|-1) + \bnr$, more importantly, if $\bnb>\bnr$. This in turn implies $\anb>\anr$. This means that at time $T_r+2$ blue tries to occupy vertices that are already colored red. Hence, blue cannot increase its degree anymore and is left with the few maximal degree vertices with $\log$(degree)/$\log n$ approximately $(\tau-2)^{\bnb}/(\tau-1)$. This means that similarly to Case \ref{case::tb-tr>2}, \emph{only red} can cross the peak and \eqref{eq::slopestart} is valid.
\item[\namedlabel{casetb-tr=1,coex}{$(3)(b)$}]  If $\Ybn/\Yrn \in ((\tau-2), 1]$ then $-\log q/|\log(\tau-2)| \in  [0,1).$ Using the rewrite \eqref{eq::key-2} again, we see that $T_b=T_r+1$ in this case is only possible if $-\log q/|\log(\tau-2)| + \bnr >1$, and in this case $\bnb<\bnr$. (Otherwise, $T_r=T_b$ and we are back to Case (2) above.)  This in turn implies $\anb<\anr$, hence, at time $T_r+2$,
blue and red both try to occupy vertices with $\log$(degree)/$\log n$ in the interval $[\anb,\anr)$: here, each vertex has an equal chance to be colored red or blue. On the other hand, at the same time red also occupies almost all vertices in the interval $[(\tau-2)\anr, \anb)$.\footnote{Again, by \eqref{eq::widetildegamm} below. A few vertices might be blue here as well, see Lemma \ref{lem::red-intervals} for error bounds.} Recall $\dnb, \dnr$ from \eqref{def::delta}.  In the shorthand notation introduced in \eqref{eq::shorthand}, we can write in this case
\be\label{eq::shorthand-2} (\dnr, \dnb] \in \text{Mix}, \quad (\dnb, \dnr+1] \in \CR.  \ee For an illustration see Fig. \ref{fig::cross-2}.
\end{description}
\end{description}

\begin{figure}
\includegraphics[width=0.6\textwidth]{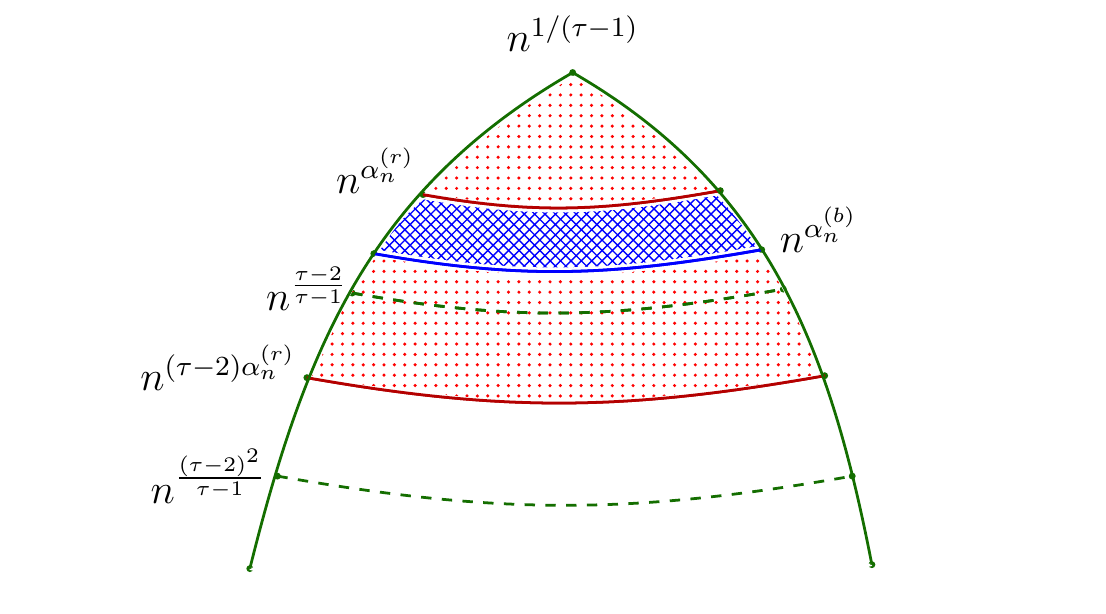}
\caption{Case (3)(b): An illustration of the crossing of the mountain peak when $T_b=T_r+1$ at time $T_r+2$, and $\Ybn/\Yrn>\tau-2$. The crosshatched (darker) and dotted (lighter) areas indicate the parts of the mountain where vertices are red and blue with equal probability, and almost all red, respectively.}\label{fig::cross-2}
\end{figure}

This completes the \emph{crossing the peak of the mountain} phase.

\section{Red or mixed avalanche from the peak}\label{sc::slopedown}
In this section we describe how the (potentially) two colors roll down the mountain. As much as possible, we try to keep together the proofs for Cases (1) -- (3) above. From now on, we write $\Pv_Y(A):=\Pv(A| \Yrn, \Ybn, n)$. Note that all the quantities below are actually random variables that depend on $\Yrn, \Ybn$. We omit the dependence on $n$ in the notation.

Recall $\wgar_1$ and $\wur_1, \wub_1$ from \eqref{def::witgamma} and \eqref{eq::wideu1} and \eqref{eq::wideu1-b}, as well as  define for $j=r,b$
\be \label{eq::wideui_recursion} \widetilde u_{\ell+1}^{\sss{(j)}}=C \log n\!\cdot\!(\widetilde u_{\ell}^{\sss{(j)}})^{\tau-2}. \ee
and also the increasing sequence of sets
\be\label{eq::wgar} \widetilde \Gamma_\ell^{\sss{(j)}}:=\{ v: D_v > \widetilde u_\ell^{\sss{(j)}}\},\ee
i.e., now $\wgar_1 \subset \wgar_2 \subset \dots$ holds, and in Case (2)  $\wgab_1 \subset \wgar_1 \subset \wgab_2 \subset \wgar_2 \dots$, (see Fig.\ \ref{fig::avalanche-2}),  while in Case (3)(b) we have $\wgab_1 \subset \wgar_2 \subset \wgab_2 \subset \wgar_3 \dots$ (see Fig. \ref{fig::avalanche-3}).

Since \eqref{eq::wideui_recursion} is the very same as the recursion in \eqref{eq::ui_recursion}
with indices exchanged, we can apply equation \eqref{eq::gamma-i-connectivity} to $\big(\wgar_\ell\big)_{\ell\ge 1}$, and to $\big(\wgab_\ell\big)_{\ell\ge 1}$  now yielding that, for some $\nu<1$, for all $\ell< \nu \log\log n / |\log(\tau-2)|+ O_{\Pv}(1)$\footnote{This inequality on $\ell$ guaranties that the logarithmic correction terms in $\widetilde u_\ell$ are of less order than the main factor.},
\be\label{eq::widetildegamm} \wgar_{\ell+1}\subset N(\wgar_\ell)\quad\mbox{ and }\quad \wgab_{\ell+1} \subset N(\wgab_{\ell}) \quad \mbox{whp.} \ee
This means that
\begin{enumerate}
\item[(1)] if $T_b-T_r\ge 2$, whp red occupies whp \emph{all} vertices in $\wit\Gamma_\ell$ at time $T_r+\ell$. More precisely, as long as $\wur_\ell>D_{\max}(T_r+\ell)$, every vertex in $\wgar_\ell$ is colored red. If this inequality is not true anymore, at time $T_r+\ell$, red still occupies every vertex in $\wgar_\ell$ that is not in the $\ell$-neighborhood of a vertex that is already blue at time $T_r$. We work out the number of such vertices in Section \ref{sc::path-counting}.
\item[(2)] if $T_b-T_r=0$, we have seen in Section \ref{sc::peak} that at time $T_r+1$, every vertex in $\wgab_1$ is  red or blue with equal probability, and every vertex in $\wgar_1\setminus \wgab_1$ is red, see Fig. \ref{fig::cross-1}. We show  in Lemma \ref{lem::independence} below that whp every vertex in $\wgab_2\setminus \wgar_1$ has plenty of neighbors in  $\wgab_1$, hence, the probability that none of them is blue or none of them is red is negligible. As a result, these vertices are again painted red and blue with equal probability at time $T_r+2$. On the other hand, almost all vertices  in $\wgar_2\setminus \wgab_2$ whp only have neighbors\footnote{For an error bound on the number of blue vertices, see Lemma \ref{lem::red-intervals} below.} in $\wgar_1\setminus \wgab_2$, hence, most vertices in $\wgar_2\setminus \wgab_2$ are painted red whp. In the shorthand notation, we have
\[ (\dnr, \dnb+1] \in \mathcal{M}ix, \quad (\dnb+1, \dnr+1] \in \CR ed\]
One can continue inductively to see that at time $T_r+\ell$, every $m\le \ell$
\be\label{shorthand-case2-ell} (\dnr+m-1, \dnb+m)\in \CM ix, \quad (\dnb+m, \dnr+m] \in \CR ed,\ee
that is, the `mixed-and-red' interval pattern `rolls down' the mountain. For the precise statement, see Lemmas \ref{lem::independence} and \ref{lem::red-intervals} below.

\item[(3)(a)] If $T_b-T_r=1$ and $\Ybn/\Yrn<\tau-2$, only a small number of vertices are blue. At time $T_r+\ell$, red occupies every vertex in $\wgar_\ell$ that is not in the $\ell$-neighborhood of a vertex that is already blue at time $T_r$. This case is exactly the same as Case (1) after red and blue have already met.
\item[(3)(b)] If $T_b-T_r=1$ and $\Ybn/\Yrn>\tau-2$, we have seen that at time $T_r+1$, red paints almost all vertices in $\wgar_1$ and blue reaches some vertices with degree at least $u_{i^\star\sss{(b)}}^{\sss{(b)}}$. Then, at time $T_r+2$, the previous section's Case (3)(b) analysis implies that $\wgar_1\subset\wgab_1$. Combining this with the fact that $\wgab_1\subset \wgar_2$, we see that all vertices in $\wgab_1\setminus \wgar_1$ are coloured red or blue with equal probability, while most vertices in $\wgar_2\setminus \wgab_1$ are whp red. This verifies \eqref{eq::shorthand-2}, and Fig. \ref{fig::cross-2}.
From here, analogously as the reasoning for Case (2), we have that at time $T_r+\ell$, each vertex in the set $\wgab_{\ell-1}\setminus \wgar_{\ell-1}$ is coloured red and blue with equal probability whp, while most vertices in the set $\wgar_{\ell}\setminus \wgab_{\ell-1}$ are red whp.
In shorthand notation, at time $T_r+\ell$, for every positive integer $m\le \ell$,
\be\label{eq::shorthand-case3b-ell}  (\dnr+m-1, \dnb+m-1] \in \CM ix, \quad (\dnb+m-1, \dnr+m] \in \CR ed.\ee
Note that both in Case (2) and (3)(b), at any particular time, the coloring always `ends' with a red layer, that is, the layer with smallest degree that is colored is red.
\end{enumerate}
\begin{figure}[ht]
\centering
\subfigure[Case (2): $T_b=T_r$]{\label{fig::avalanche-2}
\includegraphics[keepaspectratio,width=6cm]{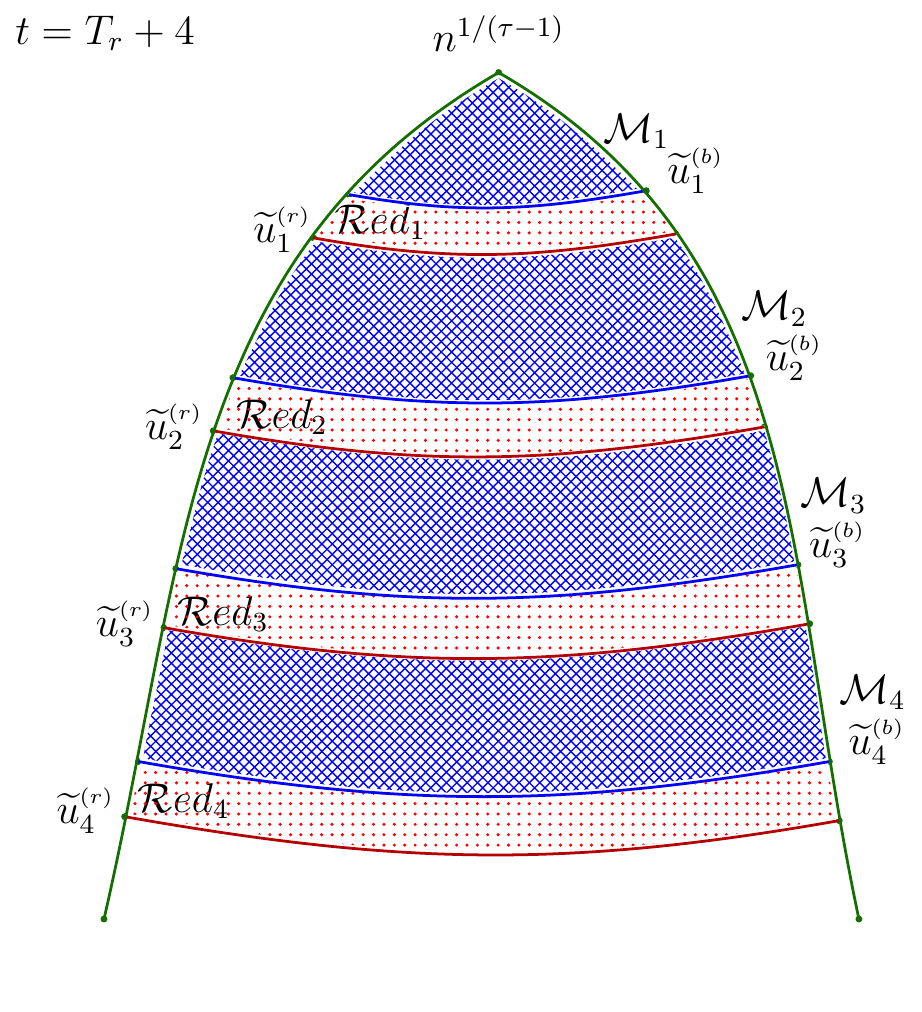}}
\subfigure[Case (3)(b): $T_b=T_r+1$, $\Ybn/\Yrn>\tau-2$]{\label{fig::avalanche-3}
\includegraphics[keepaspectratio,width=6cm]{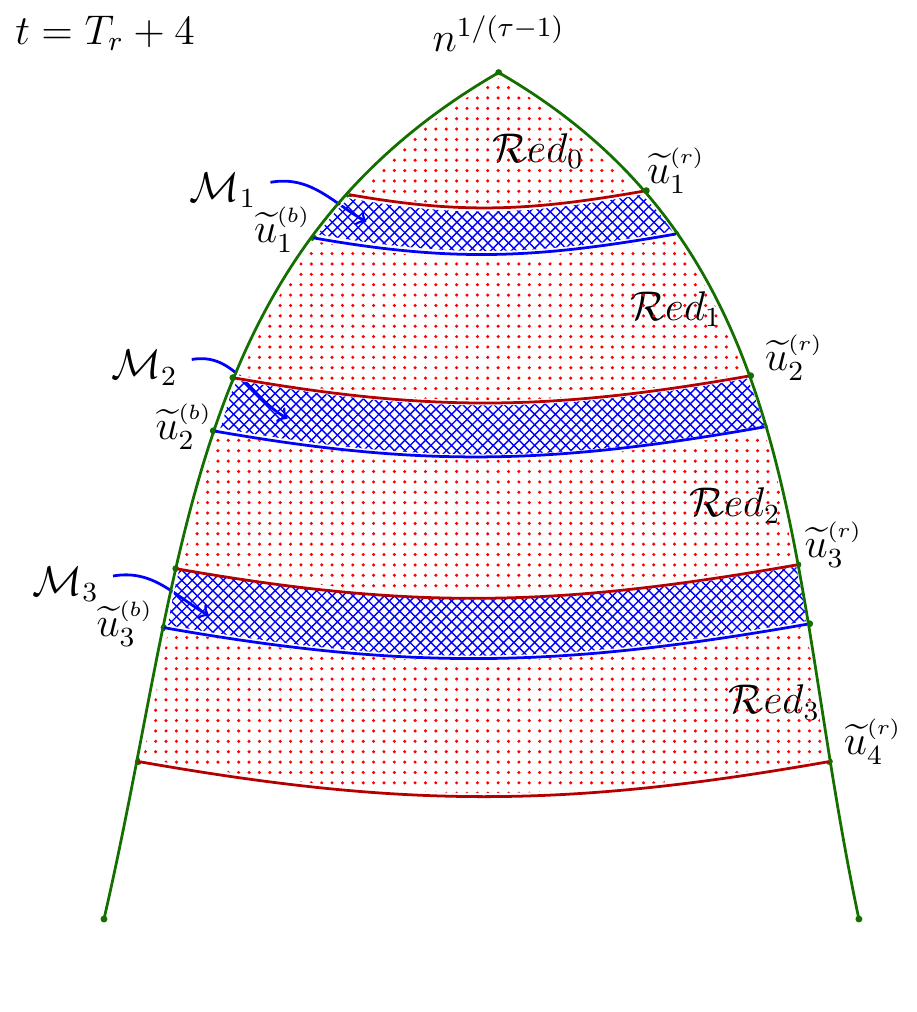}}
\caption{On these pictures, the mixed avalanche is illustrated when $\Ybn/\Yrn>(\tau-2)$, and it shows the state of the avalanche at time $T_r+4$. Dotted (lighter) areas indicate vertices that are almost all coloured red, while cross-hatched (darker) areas indicate vertices with equal probability to be blue or red. On both pictures, $\CM_k$ and $\CR ed_k$ are occupied at time $T_b+k$, $\CM_k$ is mixed while $\CR ed_k$ is mostly red.}\label{fig::avalanches}
\end{figure}

Solving the recursion \eqref{eq::wideui_recursion} yields that
\be\label{eq::ul} \wur_\ell= n^{\anr(\tau-2)^{\ell-1}} (C\log n)^{\bnr (\tau-2)^{\ell -1} + \frac{1}{3-\tau}\left(1-(\tau-2)^{\ell-1}\right)},\ee
where $\anr$ and $\bnr$ were defined in \eqref{def::alpha}.

To shorten the too complicated notation for the next two lemmas, let us introduce the following notation for the all-red and equal probability (mixed) layers in the avalanche:
For $T_b=T_r$, for $\ell\ge 1$ let
\[    \CR{ed}_\ell:= \wgar_\ell \setminus \wgab_{\ell}, \quad \CM_\ell:= \wgab_\ell \setminus \wgar_{\ell-1},  \]
where we mean $\wgar_0:=\varnothing$.
For $T_b=T_r+1$, for $\ell \ge 0$ we let
\[    \CR{ed}_\ell:= \wgar_{\ell+1} \setminus \wgab_{\ell}, \quad \CM_\ell:= \wgab_\ell \setminus \wgar_{\ell},  \]
with $\wgab_0:=\varnothing$. This means that for $T_b=T_r+1$ the red layer at the very top of the avalanche is denoted by $\CR ed_0$, while $\CM_0$ can be ignored.
See Figure \ref{fig::avalanches} for this notation.
More importantly, note that in both settings, $\CM_\ell$ and $\CR ed_\ell$ are colored at the same time, at time $T_b+\ell$. In what follows, we write $d_w$ for the degree of vertex $w$.
\begin{lemma}[Independence of mixed coloring]\label{lem::independence}
Suppose $\Ybn/\Yrn>\tau-2$, and let us condition on the values $\Yrn, \Ybn$, and write $\wit\ind=\ind_{\{T_b=T_r\}}$. Define $\wur_0:=n^{1/(\tau-1)}C \log n$. Then, the two colors red and blue whp reach the vertices in $\wgab_\ell \setminus \wgar_{\ell-\wit\ind}$ at the same time (at $T_b+\ell$), for any $1\le \ell< \nu \log\log n / {|\log (\tau-2)|} + O_{\Pv}(1)$, for some $\nu<1$.

Further, the color of vertices with degree in the interval $[ \wub_\ell)^{1+\ve}, \wur_{\ell-\wit\ind}] $ for some $\ve>0$, can be described as i.i.d. random variables taking value red and blue with equal probability.
\end{lemma}
\begin{proof}
We proceed by induction. Recall that if a not-yet-colored vertex is reached by the red and blue clusters at the same time, then its color is decided by a fair coin flip independently of everything else. Also recall the arguments in Section \ref{sc::peak}, that the vertices in $\CM_1:=\wgab_1 \setminus \wgar_{1-\wit \ind}$ are reached at the same time by the two colors  (see \eqref{eq::shorthand} and \eqref{eq::shorthand-2}). Hence, independent coin flips will decide the color of these vertices, yielding that the statement holds for $\ell=1$.

Next we advance the induction.
Suppose the statement holds for all indices at most $\ell-1$, then we show that it also holds for $\ell$. For this, we show that whp any vertex $w$ satisfying the criterion of the lemma is connected to at least one red and at least one blue vertex that is in $\CM_{\ell-1}$. Hence, a coin flip will decide its color, so its coloring is red and blue with equal probability, and moreover, this coloring is independent of the coloring of other vertices in the same interval, that is, in $\CM_\ell$.

First we show that whp $w$ is connected to many vertices in $\CM_{\ell-1}$.
Using the fact that $d_w$ is at least a power $\ve$ away from the boundary of $\CM_\ell$, let us fix an $\ve_\ell <\ve$ and define
\be\label{eq::i-ell} I_{\ell-1}:=[(d_w)^{(1-\ve_{\ell})/(\tau-2)},(d_w)^{1/(\tau-2)}] \subset \CM_{\ell-1}.\ee
By a concentration of binomial random variables (see e.g. \cite[Lemma 3.2]{BarHofKom14}) and the fact that the degrees are i.i.d. in $\CMD$, using \eqref{eq::F}, the number of vertices in $I_{\ell-1}$
is whp within the interval
\be\label{eq::interval-vertices} [ n (d_w)^{\frac{1-\ve_{\ell}}{\tau-2}(1-\tau)} c_1/2  , 2 C_1 n (d_w)^{\frac{1-\ve_{\ell}}{\tau-2}(1-\tau)}],\ee
as long as $(d_w)^{\ve_\ell (1-\tau)} \to 0$.
Hence, the number of half-edges incident to these vertices  is at least $n (d_w)^{\ve_{\ell}-1} c_1/2$ whp. Now, on the event $\{\CL_n/n \in (\Ev[D]/2, 2\Ev[D])\}$, the expected number of half-edges connecting $w$ to vertices in  $I_{\ell-1}$ is at least
\[  d_w \frac{n (d_w)^{\ve_{\ell}-1} c_1/2 }{\Ev[D]n/2}= \wit C(d_w)^{\ve_{\ell}}, \] for some constant $\wit C$. Next we show that most of these half-edges connect to disjoint vertices:
the probability that there are $2$ half-edges of $w$ that connect to the same vertex is at most
\[\ba  {d_w \choose 2} \sum_{v\in I_{\ell-1} } \frac{d_v^2}{(n\Ev[D]/2)^2} &\le \wit C (d_w)^{2} \frac{(d_w)^{\frac{2}{\tau-2} }}{n^2}  n (d_w)^{\frac{1-\ve_{\ell}}{\tau-2}(1-\tau)}\\
&\le \wit C (d_w)^{\frac{\tau-1}{\tau-2}(1+\ve_{\ell})}/n, \ea\]
which is $o(1)$ as long as $d_w< n^{(\tau-2)(1-\ve')/(\tau-1)}$ for some $\ve'>0$. In particular, this holds for every $\ell\ge 3$,  since then $d_w< \wub_2=n^{(\tau-2)^{\dnb+1}/(\tau-1)(1+o_{\Pv}(1))}$, and also holds for $\ell=2, \wit\ind=0$, since in this case $d_w< \wur_2 = n^{(\tau-2)^{\dnr+1}/(\tau-1)(1+o_{\Pv}(1))}$. If $\ell=2, \wit\ind=1$, and $n^{(\tau-2)/(\tau-1)}<d_w< \wur_1$, then $w$ is whp connected to \emph{all} the vertices in $\wgab_1$:
by Lemma \ref{lem::direct_connect}, since the sum of the exponents of $n$ of the degree of $w$ and any vertex in $\wgab_1$ is
\[ \frac{(\tau-2)^{\dnr} + (\tau-2)}{ \tau-1} > 1,  \]
since $\dnr \in [0,1)$. The number of vertices in $\wgab_1$ is at least $n^{1-(\tau-2)^{\bnb} (1+ o_{\Pv}(1))}$, which is still plenty.
Summarizing, we get that $w$ is connected to at least $\wit C(d_w)^{\ve_{\ell}}$ many vertices in $\CM_{\ell-1}$.

By the induction hypothesis, these vertices are colored red and blue with equal probability. Hence, the probability that all of them are blue or all of them are red is at most $2\cdot 2^{-\wit C(d_w)^{\ve_{\ell}}}$, which tends to zero as long as $d_w> (C\log n)^\sigma$ for some power $\sigma>0$.

Hence, $w$ has whp at least one red and at least one blue neighbor that is colored a time-step earlier, so an independent coin flip will decide the color of $w$.

Note that for the induction to hold true, we need to use a decreasing sequence of $\ve'>\ve_\ell> \ve_{\ell-1}>\ve_{\ell-2} \dots$ to reach higher and higher intervals. First, \eqref{eq::interval-vertices} needs to hold true, and also $(\wub_i)^{\ve_i} \to \infty$ for all $i$. These are all guaranteed if all $\ve_i>\ve'/2$, for instance.
\end{proof}
In the proof of the next lemma, we repeatedly use the following claim:
\begin{claim}\label{claim::Sbound} Let $\CE_{\ge y_n}$ denote the total number of half-edges incident to vertices with degree at least $y_n$, and  $V_{\ge y_n}$ the total number of vertices of degree at least $y_n$, for a sequence $y=y_n$.  Then, for any $\omega(n)\to \infty$, and a large enough constant $C<\infty$,  a small enough constant $0<c$ and for some constant $0<c_2<\infty$,
\be\label{eq::Sbound} \ba \Pv ( \CE_{\ge y_n} \ge C \omega(n) \!\cdot\! n\!\cdot\! y_n^{2-\tau} ) &\le c_2/ \omega(n), \\
\Pv ( V_{\ge y_n} \ge C\!\cdot\!  n\!\cdot\! y_n^{1-\tau}) &\le \exp\{- c_2\cdot  n y_n^{1-\tau}\},\\
\Pv ( V_{\ge y_n} \le c\!\cdot\!  n\!\cdot\! y_n^{1-\tau}) &\le \exp\{- c_2\cdot  n y_n^{1-\tau}\}
 \ea
\ee
\end{claim}
\begin{proof} Note that
\[ \CE_{\ge y_n}= \sum_{i=1}^n D_i \ind_{\{ D_i \ge y_n\}} \le \sum_{i=1}^n \sum_{k=1}^\infty y_n 2^{k} \ind_{\{  2^{k-1} y_n< D_i \le 2^k y_n\}}. \]
Exchanging sums we can write
\[  \CE_{\ge y_n} \le \sum_{k=1}^{\infty} y_n 2^k X_k^{(n)},  \]
where the marginals of the random variables on the rhs are $X_k^{(n)}\ {\buildrel d \over \le }\  \mathrm{Bin}(n, C_1 (y_n 2^{k-1})^{1-\tau}  )$, where $C_1$ is from \eqref{eq::F}. Calculating the expected value and using Markov's inequality yields \eqref{eq::Sbound}. The proof for $V_{\ge y_n}$ is easier and directly follows from the fact that $V_{\ge y_n} \sim \mathrm{Bin}(n, 1-F(y_n))$ and usual concentration of Binomial random variables (see e.g. \cite[Lemma 3.2]{BarHofKom14}).
 \end{proof}
Recall that $\simp, \overset{\Pv}{\lesssim}$ means whp equality/ inequality up to a multiplicative factor of finite powers of $C \log n$, and that
 we assume that the event  $\{ \CL_n/n \in [ \Ev[D]/2, \Ev[D] ] \}$ holds.
\begin{lemma}[Red coloring in red intervals]\label{lem::red-intervals}
Recall $\wub_0:=n^{1/(\tau-1)}C \log n$, and set some $0<\nu<1$.
Then, for any $\le \ell< \nu \log\log n / {|\log (\tau-2)|} + O_\Pv(1)$, `almost every' vertex $w$ that satisfies
$\wur_{\ell+\wit\ind}<d_w < (\wub_{\ell})^{1-\ve}$ for some $\ve>0$,  is whp painted red, where again $\wit\ind=\ind_{\{T_b=T_r\}}$. More precisely, for a uniformly picked vertex $v$ in $\CR ed_{\ell}$
\[  \Pv_{Y,n}( v \in \CR ed_\ell \text{\ is blue\ }) \le (C \log n)^{x+\ell}\,  \wur_{\ell+\wit \ind }/\wub_\ell ,\]
for some finite $x\in \N$ that depends only on $\Yrn, \Ybn, n$.
\end{lemma}
\begin{proof}[Proof of Lemma \ref{lem::red-intervals} when $T_b=T_r+1$.]

 Due to the shift of indices that will become visible later, we handle the cases $T_b=T_r$ and $T_b=T_r+1$ separately. We start with $T_b=T_r+1$.

We have seen in Section \ref{sc::peak}, that if $T_b=T_r+1$ then all vertices in $\CR ed_0$ are red at time $T_r+1$, except maybe those that are maximal degree blue vertices. (In case $\widehat u_{i_{\star \sss{(b)}} }^{\sss{(b)}} > \wur_1$.) Let us call the set of blue vertices (if any) in $\CR ed_0$ by $\CB ad_0$. By Lemma \ref{lem::numberofverticesinGamma}, there are at most $(C\log n)^{x_0}$ many such vertices (where $x_0$ is a rv that depends only on $\Yrn, \Ybn, n$), which is much smaller than the total size of $\CR ed_0$ ($\simp\! n (\wur_1)^{1-\tau}$).

  To some (minor) extent, $\ell=1$ is different from $\ell\ge 2$, so let us calculate one more step. We want to give an estimate on the total number of blue vertices.  Any vertex in this interval or in $\CM_1$ is colored at time $T_r+2$: for each vertex in $\CR ed_1$, there is at least a neighbor in $\CR ed_0$ (since $\wgar_2 \subset N(\wgar_1)$), but some vertices in $\CR ed_1$ might be connected to maximal degree blue vertices as well (i.e., vertices in $\CB ad_0$.)
Hence, a vertex in $\CR ed_1$ might be blue if it is a neighbor of $\CB ad_0$.
Thus we define as the potentially blue set
\[ \CB ad_1:=  N( \CB ad_0) \cap \CR ed_1.\]
Since $\CR ed_1$ has minimal degree $\wur_2$, by Claim \ref{claim::Sbound}, with $\omega(n):=C\log n$ the total number of half-edges incident to vertices in $\CR ed _1$ is  at most $ n (\wur_2)^{2-\tau} C\log n$.
Then, the expected number of half-edges from $\CB ad_0$ paired to $\CR ed_1$ is bounded by
\be\label{eq::b1-r2}   \Ev_Y[  \#\{ \CB ad_0 \leftrightarrow \CR ed_1\}   ] \le  (C \log n)^{x_0+1}  \widehat u_{i_{\star \sss{(b)}} }^{\sss{(b)}}  \frac{n (\wur_2)^{2-\tau}}{n \Ev[D]/2} \le (C \log n)^{x_0+x_1+1} \frac{n (\wur_2)^{2-\tau}}{\wub_1},  \ee
where we used that $\widehat u_{i_{\star \sss{(b)}} }^{\sss{(b)}}= (C \log n)^{x_1} u_{i_{\star \sss{(b)}} }^{\sss{(b)}} = n (C\log n)^{x_1+1}/ \wub_1$, for some $x_1$ depending on $\Yrn, \Ybn, n$, see \eqref{eq::wideu1}.
We need to show that the right hand side of \eqref{eq::b1-r2} is of smaller order than the total number of vertices in $\CR ed_1$, which is  $\simp n (\wur_2)^{1-\tau}$ by Claim \ref{claim::Sbound}.  Then we can write
\[   \frac{n (\wur_2)^{2-\tau}}{\wub_1} =  \frac{\wur_2}{\wub_1} n (\wur_2)^{1-\tau}\simp \frac{\wur_2}{\wub_1} |\CR ed_1|,\]
Note that $\wur_2 =o(\wub_1)$, (see Figure \ref{fig::avalanche-3}), since if $T_b=T_r+1$ then $\wur_2\le n^{(\tau-2)/(\tau-1)}<\wub_1$.  The statement of the lemma follows with $x=x_0+x_1+2$ for $\ell=1$, where we added the extra $C \log n$ factor for a Markov's inequality.

In what follows we show that for $\CB ad_\ell$, the blue vertices in $\CR ed_\ell$, for every $\ell<\nu \log\log n /|\log (\tau-2)|+O_{\Pv}(1)$,
\be \label{eq::induction-step-canceled}|\CB ad_\ell| \le n (C \log n)^{x+2\ell}(\wur_{\ell+1})^{2-\tau}/\wub_{\ell},
\ee
where again, $x= x(n)=x_0+x_1+1$ as for $\ell=1$, is a rv that depends only on $\Yrn, \Ybn, n$, and forms a tight sequence of rv-s.
  By the arguments above, the statement holds for $\ell=1$. Next we show it for general $\ell$. Let us write $N_k(S)$ for the set of vertices $k$ edges away from the set $S$ ($N(S)=N_1(S)$ in this notation).

For this, we decompose $\CB ad_\ell$, $\ell\ge k$ into subsets, according to how many `generations' a vertex has to go back to an $\CM_i$, $i\le \ell$ as follows:
\be \CB ad_\ell= \bigcup_{k=1}^{\ell} \CB ad_{\ell}(k), \ee
where
\be \ba \CB ad_\ell(1)&=N(\CM_{\ell-1}) \cap \CR ed_\ell, \\
\CB ad_\ell(k)&= N (\CB ad_{\ell-1} (k-1))  \cap \CR ed_\ell, \quad k\ge 1. \ea \ee
This means that we divide bad vertices according to `how many generations' are they already bad: for instance, a bad vertex $v$ in $\CB ad_3$ might be `first generation' bad, being a neighbor of a vertex in $\CM_1$. Or, second generation bad, being a neighbor of a first generation bad vertex in $\CB ad_2$, when a path $\CM_1 \to \CB ad_2 (1) \to v$ is present in the graph. Or `third generation' bad being a neighbor of a second generation bad vertex in $\CB ad_2$, when a path $\CB ad_0 \to \CB ad_1 \to \CB ad_2(2) \to v$ is present in the graph\footnote{Note that it is also possible to have a blue vertex in $\CB ad_3$ that is the first or second neighbor of $\CB ad_0$, or the first neighbor of $\CM_1$, but, using the same methods as in \eqref{eq::badellk-bound} below it is easy to show that these, even summed up, are negligible compared to the main terms, represented by $\CB ad_\ell(k)$.}.
\begin{figure}\label{fig::badsets}
\includegraphics[width=0.7\textwidth]{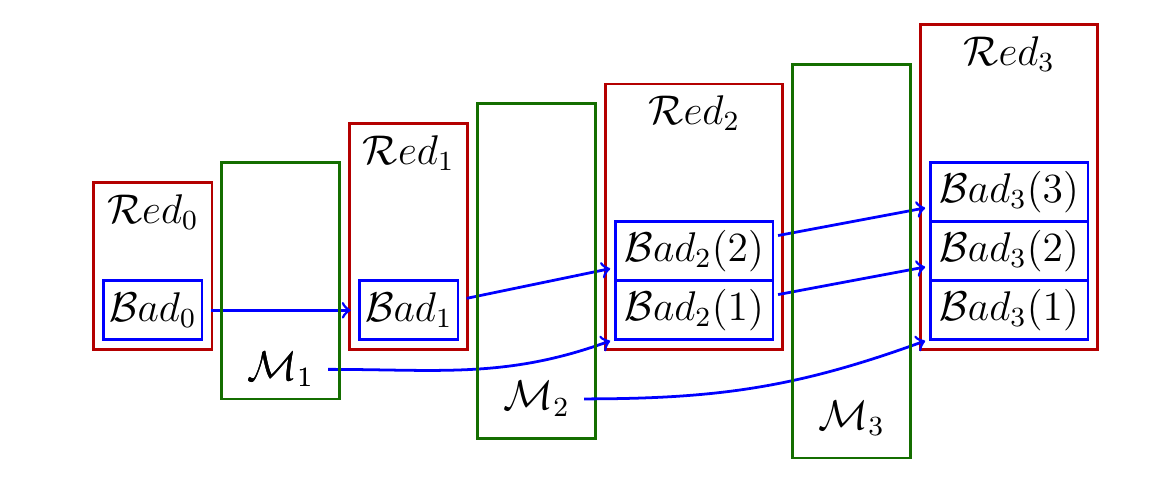}
\caption{An illustration of the structure of bad sets up to $\ell=3$.}
\end{figure}
Next we calculate the expected size of each $\CB ad_\ell(k).$ For a vertex in $\CB ad_\ell(k)$, there must be a path of length $k$
\be  \CM_{\ell-k} \to \CB ad_{\ell-k+1}(1) \to \CB ad_{\ell-k+2}(2) \to \dots \CB ad_{\ell-1}(k-1) \to \CB ad_\ell(k)\ee present in $\CMD$.
Since the total number of half edges $\Ev_Y[H(\CM_{\ell-k})]\le C  n (\wub_{\ell-k})^{2-\tau} $ for some constant $C$ by Claim \ref{claim::Sbound}, the expected number of such paths and hence the size of $\CB ad_{\ell}(k)$ can be estimated , for all $1\le k < \ell$ as
\be\label{eq::badellk} \Ev_{Y,n}[|\CB ad_\ell(k)| ] \le n (\wub_{\ell-k})^{2-\tau} \prod_{i=1}^{k-1} \left(\sum_{v \in \CR ed_{\ell-k+i}} \frac{d_v (d_v-1)}{n}\right) \cdot \sum_{v \in \CR ed_{\ell}} \frac{ d_v}{n},\ee
while for $k=\ell$ we have an extra factor $(C \log n)^{x_0+x_1}$ on the right hand side.
  Now, it is not hard to see that using \eqref{eq::size-biased2}, the sums inside the product can be approximated by
  \be(2\Ev[D])^{-1} \Ev[D(D-1) \ind_{\{\wur_
  {\ell-k+i+1}<D< \wub_{\ell-k+i}\}}] \le C (\wub_{\ell-k+i})^{3-\tau}, \ee
  and hence, multiplying the rhs here with $C \log n$ implies
  \be\label{eq::sum-to-expected} \Pv_{Y,n}\left( \sum_{v \in \CR ed_{\ell-k+i}} \frac{d_v (d_v-1)}{n} \ge  C \log n (\wub_{\ell-k+i})^{3-\tau}  \right) \le \frac{1}{C \log n}.\ee
  The last sum in \eqref{eq::badellk} is at most $C \log n (\wur_{\ell+1})^{2-\tau}$ whp by Claim \ref{claim::Sbound}.
Thus, the rhs of \eqref{eq::badellk} is at most, as $n\to \infty$,
\be \label{eq::badellk-2}  \Ev_{Y,n}[|\CB ad_\ell(k)| ] \le (C \log n)^{1+\ind_{k=\ell }(x_0+ x_1)} n (\wub_{\ell-k})^{2-\tau} \prod_{i=1}^{k-1} \left(C\log n(\wub_{\ell-k+i})^{3-\tau}\right) \cdot (\wur_{\ell+1})^{2-\tau}.  \ee
Now we can repeatedly apply \eqref{eq::wideui_recursion} in the form $(\wub_s)^{2-\tau}= C \log n / \wub_{s+1}$,
and then \eqref{eq::badellk-2} simplifies to
\be\label{eq::badellk-bound} \Ev_{Y,n}[|\CB ad_\ell(k)| ] \le n (C\log n)^{2k+1+\ind_{k=\ell}(x_0+ x_1) } \frac{(\wur_{\ell+1})^{2-\tau}}{\wub_{\ell}}.  \ee
Hence, we get that for some $\hat C$ constant,
\be \Ev_{Y,n}[|\CB ad_\ell|] \le \sum_{k=1}^{\ell} \Ev[\CB ad_{\ell}(k)] \le \hat C (C \log n)^{x-1+2\ell} n \frac{(\wur_{\ell+1})^{2-\tau}}{\wub_{\ell}}, \ee
with $x=x_0+x_1+2$.
Adding an extra factor of $C \log n$ to the rhs, the inequality holds whp without the expected value on the lhs as well.

Finally, by the concentration of binomial random variables, $|\CR ed_\ell| \ge c n  (\wur_{\ell+1})^{1-\tau}$ whp, for some constant $c$.
As a result, the probability that a vertex in $\CR ed_\ell$ is in $\CB ad_\ell$,  conditioned on $\Yrn, \Ybn$
\be\label{eq::prob-of-bad}\Pv_{Y,n}( v \in \CB ad_\ell | v\in \CR ed_\ell ) \le \hat C (C \log n)^{2\ell+x} \frac{\wur_{\ell+1}}{\wub_\ell},  \ee
 whp.
 If $\ell\le \nu \log \log n/ |\log (\tau-2)|$, for some $\nu<1$, then using \eqref{eq::ul}
 \be\label{eq::prob-of-bad-bound}  \wur_{\ell+1}/\wub_\ell = \exp\left\{ - (\log n)^{1-\nu} (\anb-(\tau-2)\anr) \right\} (C \log n)^{1/(3-\tau)} (1+ o(1)),  \ee
 where recall that $\anb>(\tau-2)\anr$ when $T_b=T_r+1$. On the other hand, $(C \log n)^{2\ell}\le \exp\{ c (\log \log n)^2\}$, hence the contribution of this term is negligible if $n$ large enough.

 As a result, the rhs of \eqref{eq::prob-of-bad} tends to zero for all $\ell< \nu \log \log n /|\log (\tau-2)|$, as $n\to \infty$.
\end{proof}
\begin{proof}[Proof of Lemma \ref{lem::red-intervals} when $T_b=T_r$.]

 For $T_b=T_r$, the induction starts slightly differently. Namely, in this case $\wub_1>\wur_1$ and the highest degree vertices belong to $\CM_1$, see Figure \ref{fig::avalanche-2}. $\CM_1$ and $\CR ed_1$ are occupied at the same time: all vertices in $\CR ed_1$ are connected to maximal degree red vertices, but some of them might be also connected to maximal degree blue vertices. These maximal degree blue vertices might or might not be inside the set $\CR ed_1$: let us call them $\CB ad_0$. Further, let us write $\CB ad_1$ for the set of blue vertices in $\CR ed_1$,  then the size of this set can be bounded by the expected number of half-edges from maximal degree blue vertices to $\CR ed_1$, similarly as it was done in \eqref{eq::b1-r2}:
 \be\label{eq::b0-r1-2} \Ev_{Y,n}[  \#\{ \CB ad_0 \leftrightarrow \CR ed_1\}   ] \le  (C \log n)^{x_0} \widehat u_{i_{\star \sss{(b)}} }^{\sss{(b)}}  \frac{n (\wur_1)^{2-\tau}}{n \Ev[D]/2}  \le (C \log n)^{x_0+x_1+1}\frac{n (\wur_1)^{2-\tau}}{\wub_1}.
  \ee
 On the other hand, the total size of $\CR ed_1$ is $\simp n (\wur_1)^{1-\tau}$  by Claim \ref{claim::Sbound}, and since $\wur_1=o(\wub_1)$ if $T_r=T_b$ (see Figure \ref{fig::avalanche-3}), the rhs of  \eqref{eq::b0-r1-2} is of less order than that.

 From here, the exact same proof as for $T_b=T_r+1$ can be repeated, with the indices of $\wur_i$-s shifted by $-1$. In this case $\CR ed_\ell$ has total size of $n (\wur_\ell)^{1-\tau}$, and also $\wur_\ell=o(\wub_\ell)$. This implies that the rhs of equation \eqref{eq::prob-of-bad-bound} contains $\anb-\anr>0$ instead of $\anb-(\tau-2) \anr$. As a result, the corresponding version of \eqref{eq::prob-of-bad} tends to zero also in this case, meaning that the statement of the lemma implies that the proportion of blue vertices in $\CR ed_\ell$ tends to zero also when $T_b=T_r$.

\end{proof}

\section{Typical distances and  the maximal degree of blue}\label{sc::meetingtime}
In this section we describe how the two colors meet and prove Theorem \ref{thm::distances}. As a result of the analysis, we also prove Theorem \ref{thm::maxdegree}.

Recall that the time to reach the top of the mountain for the two colors for $j=r,b$ is denoted by
\be\label{eq::Tj} T_j = \left\lfloor\frac{\log\log n-\log \left((\tau-1) Y_j^{\sss{(n)}}\right)}{|\log(\tau-2)|}-1\right\rfloor = \frac{\log\log n-\log \left((\tau-1) Y_j^{\sss{(n)}}\right)}{|\log(\tau-2)|} -1-b_n^{\sss{(j)}}. \ee

\begin{proof}[Proof of Theorem \ref{thm::distances}]
For the upper bound, we show that there is whp a path that connects red and blue in at most
\be\label{eq::distance-try}  \CD(\CR_0, \CB_0) = T_r + T_b + 1 + \ind\{ \tau-1>(\tau-2)^{\bnr} + (\tau-2)^{\bnb}\} \ee
many steps.
We have seen in Section \ref{sc::climbup} that whp there exists a path of length at most $T_r$ that connects the red source to the top of the mountain, i.e., to some vertex with $\log$(degree)/$\log n$ that is $(\tau-2)^{\bnr}/(\tau-1)(1+o_{\Pv}(1))$.
The crucial observation is, that in terms of the distance of the sources, the timing of the colors does not matter, that is, imaginarily, we can stop the spread of red at this moment. Now, the same method for blue shows that there exists a path of length at most $T_b$ steps, and blue occupies some vertices with $\log$(degree)/$\log n$
that is $(\tau-2)^{\bnr}/(\tau-1)(1+o_{\Pv}(1))$. For an upper bound on typical distances, we can simply assume that these `climbing clusters'  are disjoint, that is, we assume $\CR_{T_r} \cap \CB_{T_b} = \varnothing$.

Now we let the two colors jump, and we reduced the problem to the analysis of Case (2):  If we let them do one jump, then whp the maximal degree red vertex connects to every vertex in layer $\wgar_1$, while the maximal degree blue vertex connects to every vertex in $\wgab_1$ (see Fig. \ref{fig::cross-1}). Note that the distance is then $T_r+T_b + 1$ if and only if the maximal degree blue vertex is in $\wgar_1$ or the maximal degree red vertex is in $\wgab_1$, that is, if
\[ u_{i_{\star \sss{(r)}}}^{\sss{(r)}}  \ge \wit u_1^{\sss{(b)}} \quad \mbox{or} \quad  u_{i_{\star \sss{(b)}}}^{\sss{(b)}}  \ge \wit u_1^{\sss{(r)}}. \]
 Otherwise, we can pick an arbitrary vertex in $\CM_1=\wgab_1\cap \wgar_1$, and both the maximal degree blue and the maximal degree red vertex connect to that vertex whp, hence, the distance is $T_r+T_b+2$.
Since for large enough $n$, the logarithmic factors in $u_{i_{\star \sss{(j)}}}^{\sss{(j)}}$ and $\wit u_1^{\sss{(j)}}$ are negligible, to have distance $T_r+T_b+2$ we need for large enough $n$
\[  \frac{(\tau-2)^{\bnb} }{\tau-1}< \anr \quad \mbox{and}\quad \frac{(\tau-2)^{\bnr}}{\tau-1} < \anb.  \]
 Using \eqref{def::alpha}, we see that both criteria are equivalent to
 \be\label{eq::cond-larger}\tau-1>(\tau-2)^{\bnr}+(\tau-2)^{\bnb}.\ee The region at which this is satisfied is above the red curve on Fig~\ref{fig::contourplot}. Hence, the distance between the two sources is at most \eqref{eq::distance-try}.

Note that the form at which the theorem is presented is a simple rewrite of this equation.
With this, we have finished the upper bound of the proof of Theorem \ref{thm::distances}.

\begin{figure}[ht]
\includegraphics[height=6cm]{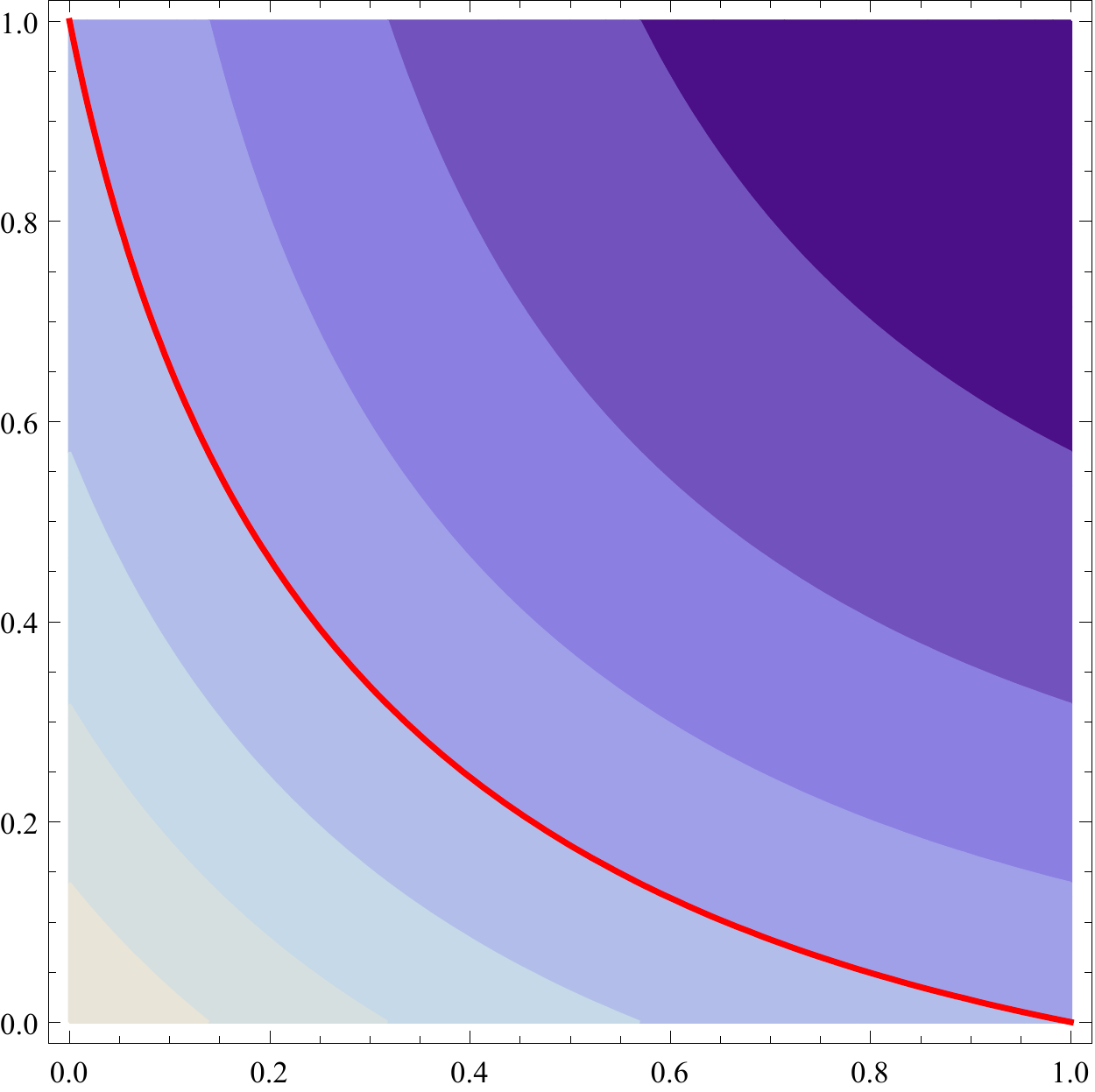}
\caption{The contour plot of the function $0.2^x+0.2^y$ with a red line indicating $0.2^x+0.2^y=1.2$, that is $\tau=2.2$. Darker colors represent smaller values, that is, $\bnb >\dnr$ is satisfied above the red curve. }\label{fig::contourplot}
\end{figure}
For the lower bound, we argue as follows: we have seen in Lemma \ref{lem::badpaths} that $T_r, T_b$ are also lower bounds for the time to reach the top of the mountain. We need to show that these `climbing clusters' are disjoint whp, i.e., any vertex that is distance at most $T_b$ away from the blue source is different from the vertices  that are distance at most $T_r$ away from the red source:
\be \label{eq::no-early-meeting} \CR_{T_r}\cap \CB_{T_b} = \varnothing. \ee
Let us postpone the proof of this statement till Claim \ref{cl::no-early-meeting} and assume that it holds, implying that \be\label{eq::distance-lowerbound} \CD(\CR_0, \CB_0 ) \ge T_b+T_r+1. \ee
So, we only have to show that if
\be\label{eq::equivalent-assumption} \tau-1> (\tau-2)^{\bnr} + (\tau-2)^{\bnb},\ee
then we need one extra edge to connect the two clusters. Recalling the definitions \eqref{eq::ui*}, \eqref{eq::uistar-b}, it is not hard to see that  \eqref{eq::equivalent-assumption} for large $n$  is equivalent to
 \be\label{eq::ui-wideu1} u_{i_{\star \sss{(r)}}}^{\sss{(r)}}  \le \wit u_1^{\sss{(b)}}= \frac{C\log n\cdot n}{u_{i_{\star \sss{(b)}}}^{\sss{(b)}} }  \quad \mbox{and } \quad  u_{i_{\star \sss{(b)}}}^{\sss{(b)}}  \le \wit u_1^{\sss{(r)}}=\frac{C\log n \cdot n}{u_{i_{\star \sss{(r)}}}^{\sss{(r)}} }.\ee
Now, recall Lemma \ref{lem::badpaths}: on the whp event that $\{ j=r,b, \forall i\le i_{\star \sss{(j)}}\  \CB ad \CP^{\sss{(j)}} =\varnothing\}$, $\widehat u_{i}^{\sss{(j)}}  $ as in \eqref{eq::uibar}  serves as an upper bound on the maximal degree of red and blue at time $t(n^{\vr})+ i_{\star \sss{(j)}}, j=r,b$.
Note also that by Lemma \ref{lem::numberofverticesinGamma}, there are only $\exp\{ \log (C \log n) \hat C (\tau-2)^{-i_{\star, \sss{(j)}}}\}$ many vertices with degree at most $\widehat u^{\sss{(j)}}_{i_{\star \sss{(j)}}}$ for $j=r,b$ with  $\hat C=2/(3-\tau)^2$.

Note that we need only one edge to connect $\CR_{T_r}$ to $\CB_{T_b}$ if some maximal degree red vertex is connected to some maximal degree blue vertex. Hence, on the whp event that the total number of half-edges satisfies $\{ \CL_n/n \in (\Ev[D]/2, \Ev[D] )\}$, by a simple union bound, the probability that any of the maximal degree red vertices is connected to any of the maximal degree vertices by an edge is at most
\be\label{eq::top-connect-bound} \Pv_{Y,n}(\CR_{T_r} \leftrightarrow \CB_{T_b}) \le \frac{2}{\Ev[D] n} (C \log n)^{ \hat C (\tau-2)^{-i_{\star \sss{(b)}}}} \widehat u^{\sss{(j)}}_{i_{\star \sss{(b)}}} \cdot
(C \log n)^ {\hat C (\tau-2)^{-i_{\star \sss{(r)}}}} \widehat u^{\sss{(j)}}_{i_{\star \sss{(r)}}},\ee
with $\hat C=2/(3-\tau)^2$.
Since $i_{\star \sss{(j)}}$, $j=r,b$ are  tight random variables that do not grow with $n$, by \eqref{eq::wideui-ui},
\be\label{eq::ui-ui-estimate}\widehat u_{i_{\star \sss{(j)}}}^{\sss{(j)}}\simp u_{i_{\star \sss{(j)}}}^{\sss{(j)}} \simp n^{(\tau-2)^{b_n^{\sss{(j)}}}/(\tau-1)} \ee

Note that the extra powers of $C \log n$  in \eqref{eq::top-connect-bound}, only depend on $i_{\star,\sss{(j)}}$ and hence do not grow with $n$. So, picking $n$ large enough and using \eqref{eq::ui-ui-estimate}, the expression on the rhs of  \eqref{eq::top-connect-bound} is $\simp (u_{i_{\star \sss{(b)}}}^{\sss{(b)}} u_{i_{\star \sss{(r)}}}^{\sss{(r)}})/n$. This is $o(1)$ when \eqref{eq::equivalent-assumption} holds.

Hence, under the condition \eqref{eq::equivalent-assumption} we need at least $T_r+T_b+2$ edges to connect, while we need at least $T_r+T_b+1$ edges in either case by \eqref{eq::no-early-meeting}.
To finish the lower bound of the proof of Theorem \ref{thm::distances}, it is left to show \eqref{eq::no-early-meeting}, that we handle in the following lemma.
\end{proof}

\begin{lemma}\label{cl::no-early-meeting} On the event $\{  \CB ad \CP^{\sss{(j)}} =\varnothing \  \forall i\le i_{\star \sss{(j)}} \mbox{ for } j=r,b\} $, the event $\{\CR_{T_r} \cap \CB_{T_b} = \varnothing\}$ holds whp.
\end{lemma}
In the proof of this lemma, and also later on in Section \ref{sc::coexistence}, we use the following technical, rather easy claim:
\begin{claim}\label{cl::stoch-dom-alpha-stable} Recall that $\widehat u_0^{\sss{(r)}} = (C \log n \cdot Z_{t(n^{\vr})}^{\sss{(r)}}    )^{1/(\tau-2)} $. Then there exist a $0<c<\infty$, so that
\[ \Pv\left(  \sum_{i=1}^{Z_{t(n^{\vr})}} D_i^\star \ge \widehat u_0^{\sss{(r)}} \right)  \le \frac{c}{(\log n)^{\tau -2}}. \]
\end{claim}
\begin{proof}
We prove the claim conditioned on the value $Z_{t(n^{\vr})}^{\sss{(r)}}:=m$. Note that $m=n^{\vr^{\sss{(r)}}}$ is a polynomial of $n$.
First, notice that we can pick a $b>0$ so that $D^\star\ {\buildrel {d}\over{\le}} \ b+X$, where $X$ is a continuous random variable with distribution function $\Pv(X\le x)= 1- C_1^\star\!\cdot\! x^{2-\tau}$ on $[0, \infty)$, where $C_1^\star$ is from \eqref{eq::size-biased2}. Then, $b+X$ is a totally asymmetric stable random variable with skewness $\kappa=1$, shift $b$ and some scale parameter $c$.
As a result, the moment generating function of $X$ is of the form
\[  \Ev[\e^{-\theta (b+X)}] = \exp\{ - \theta b + c' \theta^{\tau-2} \}. \]
Then, calculating the moment generating function of $\sum_{i=1}^m X_i$ for i.i.d.\  $X_i\sim X$ gives that
\be\label{eq::sum-domination} \sum_{i=1}^m D_i^\star \ {\buildrel {d}\over{\le}} \  \sum_{i=1}^m b+X_i \ {\buildrel {d}\over{=}}\  b m + m^{1/(\tau-2)}X' \ {\buildrel {d}\over{\le}\ }m^{1/(\tau-2)} (b + X'),\ee
where $X'\sim X$.
Then, by the stochastic domination and the tail distribution of $X'$, for an arbitrary $0<C'\le \infty$ there is a $c'<\infty$ so that
\[  \Pv\left( \sum_{i=1}^m D_i^\star \ge m^{\frac{1}{\tau-2}} C' \log m\right) \le \Pv( b + X' \ge C' \log m  ) \le \frac{c'}{(\log m)^{\tau-2}}.\]
The statement of the claim follows by noticing that $\log m=\vr^{\sss{(r)}} \log n$, with $\vr'(\tau-2)<\vr^{\sss{(r)}}<\vr'$ which modifies the constant $C'$, $c'$ to $C$ and $c$, respectively.
\end{proof}
\begin{proof}[Proof of Lemma \ref{cl::no-early-meeting}]
Recall that we write $\Pv_{Y,n}(\cdot), \Ev_{Y,n}[\cdot]$ for probabilities of events and expectations of random variables conditioned on the values $\Yrn, \Ybn$.

To prove the lemma  we first calculate the total number of free (unpaired) half-edges going out of the set $\CR_{T_r}$, which we denote by $H(\CR_{T_r})$. We do this via counting the number of paths \emph{with free ends}, that is, now  we say that a sequence of vertices and half-edges $(\pi_0, s_0, t_1, \pi_1, s_1, t_2,  \dots,  t_k, \pi_k, s_k)$ forms a \emph{path with free end} in $\CMD$, if for all $0< i\le k$, the half edges $s_i, t_i$ are incident to the vertex $\pi_i$ and $(s_{i-1}, t_i)$ forms an edge between vertices $\pi_{i-1},\pi_i$. Clearly, since the same vertex might be approached on several paths, the total number of free half-edges in $\CR_{T_r}$ can be bounded from above by the number of paths with free end of length $T_r$, starting from the red source vertex $\CR_0$. Now, on the event $\{  \CB ad \CP^{\sss{(j)}} =\varnothing, \ \forall i\le i_{\star \sss{(j)}}\
\mbox{ for } j=r,b,\} $, at time $t(n^{\vr}) + i$ $\widehat u_{i}^{\sss{(j)}}$ (see definition in \ref{eq::uibar})  is an upper bound on the degrees of color $j$ vertices.
Let us note that
\be\label{eq::d-star-indicator}  \Ev_{Y,n}[D^\star \ind_{\{D^\star < \widehat u_{i}^{\sss{(j)}}\}} ] \le C_1^\star (\widehat u_{i}^{\sss{(j)}})^{3-\tau}\ee
 by \eqref{eq::size-biased2}. We write $N_k(\CA, \text{free})$ for the total number of $k$-length paths with an unpaired half-edge starting from set $\CA$. Then, since $T_r=t(n^{\vr})+i_{\star \sss{(r)}}$ (see \eqref{eq::k*+i*}),
\be\label{eq::half-edge-to-path}H(\CR_{T_r}) \le  N_{i_{\star \sss{(r)}}}(\CR_{t(n^{\vr})}, \text{free}), \ee
and recall $\CR_{t(n^{\vr})}$ is coupled to the branching process described in Section \ref{sc::BP}. Hence, the degrees in the last generation of the BP phase are i.i.d.\ having distribution $D^\star$ satisfying \eqref{eq::size-biased2}.
By Claim \ref{cl::stoch-dom-alpha-stable}, the total number of half-edges in this last generation is whp
\[   \sum_{i=1}^{Z_{t(n^{\vr})}} D_i^\star \le \widehat u_0^{\sss{(r)}}. \]
The path counting method described in \cite[Appendix]{BarHofKom14}
gives that the expected number of paths with free ends of length $i_{\star \sss{(r)}}$ under the assumption of the claim satisfies
\be\label{eq::nkab} \Ev_{Y,n}[  N_{i_{\star \sss{(r)}}}(\CR_{t(n^{\vr})}, \text{free})] \le \widehat u_0^{\sss{(r)}} \left(\prod_{i=1}^{i_{\star \sss{(r)}}}\frac{\CL_n}{\CL_n-2i+1}\right)\cdot\sideset{}{^\star}\sum_{\substack{ \pi_1,\dots,\pi_{  i_{\star \sss{(r)}}} \\ \forall i\ \pi_i \in \Lambda_i}} \left(\prod_{i=1}^{i_{\star \sss{(r)}}    }\frac{d_{\pi_i} (d_{\pi_i}-1)}{\CL_n}\right) \ee
where $\sideset{}{^\star}\sum$ means that we sum over distinct vertices, $d_{\pi}$ denotes the degree of vertex $\pi$, and $\Lambda_i=\{v \in [n]: D_v \le  \widehat u_i^{\sss{(r)}} \}$.

Now, we can apply \eqref{eq::sum-to-expected} $i_{\star\sss{(r)}}$ many times (with  a union bound) and get, that on the event $\CL_n/n \in (\Ev[D]/2, \Ev[D])$, whp,
 \be \Ev_{Y,n}[  N_{i_{\star \sss{(r)}}}(\CR_{t(n^{\vr})}, \text{free})] \le (C \log n)^{i_{\star\sss{(r)}}}\cdot \widehat u_0^{\sss{(r)}} \cdot \left(\prod_{i=1}^{i_{\star \sss{(r)}} } (\widehat u_i^{\sss{(r)}})^{3-\tau}  \right) \e^{2 i_{\star \sss{(r)}}^2/ (\Ev[D] n) }. \ee
Note that $i_{\star \sss{(r)}}$ is a tight sequence of rv-s, see \eqref{eq::value_i*}. The value of $\widehat u_i^{\sss{(r)}}$ can be calculated from \eqref{eq::uibar} and is the same as the rhs of \eqref{def::ui} with $+e_i$ in the exponent of $C \log n$. Hence, the product of the first three factors on rhs of the previous formula, after some longish but elementary calculations, equals
\[    n^{\vr^{\sss{(r)}}  (\tau-2)^{ - (i_{\star \sss{(r)}} +1) } }(C\log n)^{   \frac{1}{3-\tau}  ( (\tau-2)^{ - (i_{\star \sss{(r)}}^{\sss{(r)}} +1) } -1  )  } = \widehat u_{i_{\star \sss{(r)}}}.  \]
Recall that the value of  $\widehat u_{i_{\star \sss{(r)}}}^{\sss{(r)}}$ is the same as the rhs of  \eqref{eq::ui*} with $+e_{i_{\star \sss{(r)}}}$ in the exponent of $C \log n$, so by \eqref{eq::half-edge-to-path}, whp,
\[ \Ev_{Y,n}[ H(\CR_{T_r})] \le n^{(\tau-2)^{\bnr} / (\tau-1) (1+o_{\Pv}(1)) }.  \]
Multiplying the rhs with a constant $C \log n$ factor to allow for an application of Markov's inequality, we also have
\be\label{eq::half-edge-red} \Pv_{Y,n}\left(H(\CR_{T_r}) \ge  C \log n \cdot n^{(\tau-2)^{\bnr} / (\tau-1) (1+o_{\Pv}(1)) } \right)  \le 1/ (C\log n). \ee

Now, let us apply the exact same technique for $H(\CB_{T_b-i})$, the half edges from the blue cluster at $T_b-i$, for all $i=1, 2, \dots, i_{\star \sss{(b)}}$. Since now we stop the paths at $t(n^\vr)+i_{\star \sss{(b)}}-1, t(n^\vr)+i_{\star \sss{(b)}}-2, \dots, t(n^{\vr})$, we get that
\[ \Pv_{Y,n}\left(H(\CB_{T_b-i}) \ge (C \log n)^{i} n^{(\tau-2)^{\bnb+i}/(\tau-1) (1+o_{\Pv}(1)) } \right)  \le 1/ (C\log n)^{i}. \]
Summing the error terms up, we get
\be\label{eq::half-edge-blue} \Pv_{Y,n}\left(H(\CB_{T_b-i}) \le (C \log n)^{i} n^{(\tau-2)^{\bnb+i}/(\tau-1) (1+o_{\Pv}(1)) } \ \forall i=1, \dots, i_{\star \sss{(b)}},  \right)  \ge 1-1/ (C\log n). \ee
Now, to see that $\CR_{T_r}$ and $\CB_{T_b}$ are disjoint, we apply the following procedure:
It is easy to see that $H(\CR_{T_r-i})$ is maximised at $i=0$. Hence, we grow the red cluster first till time $T_r$, and then stop it. Then, we grow the blue cluster step by step, looking at $\CB_1, \CB_2, \dots, \CB_{T_b}$, and at each step we check if any of the half edges paired are actually paired to a red half-edge. If this happens for any time before or at $T_b-1$, then an early connection happens and the distance is at most  $T_b+T_r$. (Note that the distance is $T_r+i$ if we pair a blue half-edge that is at the end of a path with free end of length $i-1$, to a red half-edge.  Hence, to have a connection of $T_r+T_b$, we need to find this connection at the latest when pairing the half-edges in $H(\CB_{T_r-1})$. )
The probability that there is a connection before or at $t(n^{\vr})$ is of the same order as the probability that there is a connection at time $t(n^{\vr})$, since the total degree in the whole BP is the same order of magnitude as the total degree in the last generation.
The probability that $H(\CB_{T_b-i})$ connects to $H(\CR_{T_r})$ is then at most, by a union bound,
\[   \Pv_{Y,n}( H(\CB_{T_b-i}) \leftrightarrow H(\CR_{T_r}) ) \le \frac{H(\CB_{T_b-i}) H(\CR_{T_r})}{\CL_n}.  \]
On the event that $\{\CL_n \in (\Ev[D]/2, \Ev[D)]\}$, using \eqref{eq::half-edge-blue} and \eqref{eq::half-edge-red}, we sum up the error terms for $i=1, \dots, i_{\star \sss{(b)}}$,
\be \label{eq::early-error} \Pv_{Y,n}( \CR_{T_r} \cap \CB_{T_b}\neq \varnothing ) \le \frac{2}{\Ev[D]n} n^{(\tau-2)^{\bnr}(\tau-1)(1+o_{\Pv}(1)) } \sum_{i=1}^{i_{\star \sss{(b)}}}n^{(\tau-2)^{\bnb+i} /(\tau-1)(1+o_{\Pv}(1)) }. \ee
 Note that the exponent of $n$ in the dominant term in the numerator is
 \[ \frac{(\tau-2)^{\bnr}+ (\tau-2)^{\bnb+1}}{\tau-1} \le 1, \]
 since $\bnb, \bnr \in [0,1)$. Note that $\bnb=\bnr=0$ happens with probability tending to zero.
 If $n$ is so enough that that the logarithmic factors (hidden in the $(1+o_\Pv(1))$ factor in the exponent) are negligible, the rhs of \eqref{eq::early-error} tends to zero with $n$.
This finishes the proof of the claim, and hence, the proof of the lower bound of Theorem \ref{thm::distances}.

\end{proof}

We continue to prove Theorem \ref{thm::maxdegree}, for which we need to investigate the maximal degree of a vertex occupied by blue.

\begin{proof}[Proof of Theorem \ref{thm::maxdegree}]
Before we start, recall the various notations of `$u$'-s: $u_{i}^{\sss{(j)}}$ denotes the climbing degree of color $j$ at time $t(n^\vr) +i$, while $\widetilde u_\ell^{\sss{(j)}}$ denotes the avalanche degree of color $j$  at time $T_j+\ell$ (we need $j=r$ in this section).

First we handle Case (1), i.e., we assume $T_b-T_r\ge2$. Recall the mountain climbing phase, i.e., the fact that $u_{i+1}^{\sss{(b)}} = (u_{i}^{\sss{(b)}})^{1/(\tau-2)}(1+o_{\Pv}(1))$. At time $T_r+k$, since blue would reach the top of the mountain at time $T_b$, blue is $T_b-T_r-k$ steps away from the top of the mountain, i.e., it is at vertices with $\log$(degree)/$\log n$ equal to
\be\label{eq::bluedeg-k}\frac{(\tau-2)^{\bnb}}{\tau-1}(\tau-2)^{T_b-T_r-k}(1+o_{\Pv}(1))\ee
while red at this moment is at $\log$(degree)/$\log n$ that is $\anr(\tau-2)^k(1+o_{\Pv}(1))$. Recall that $\anr=(\tau-2)^{\dnr}$ and that $\dnr\in[0,1)$. From here, the mountain climbing phase for blue and the avalanche phase for red can be continued, as long as the lowest degree where red occupies all the vertices is still higher than the maximal degree of blue. Let us define the \emph{real} time $t_c$ when this happens as the solution of the equation
\be\label{eq::intersect}(\tau-2)^{{\dnr}}(\tau-2)^{t_c-1}=(\tau-2)^{\bnb}(\tau-2)^{T_b-T_r-t_c}\ee
which yields\footnote{We remark that formula \eqref{eq::tc2} is the same as applying \cite[Equation 6.5]{BarHofKom14}  with $\la=1$ and using the identity $\log (\anr (\tau-1))/|\log (\tau-2)|={\dnr}$.}
\be\label{eq::tc2} t_c=\frac{T_b-T_r+1}{2} + \frac{\bnb-{\dnr}}{2}=\frac{1}{2}\frac{\log(\Ywn/\Yln)}{|\log (\tau-2)|} + \frac{\bnr + 1-{\dnr}}{2}. \ee
Due to the shifts $\dnr,\bnb$, we note that $t_c$ is typically not an integer. Note that the definition of $t_c$ is so that at time $\lfloor t_c\rfloor$, the maximal degree of blue, given by $u_{\lfloor t_c\rfloor}^{\sss{(b)}}$ is just `slightly less' than the location of the red avalanche at degrees $\wur_{\lfloor t_c\rfloor}$.
In what follows, we determine what happens at the next jump, at time $\lfloor t_c\rfloor +1$.

Note that
writing $\lfloor t_c\rfloor= t_c-\{t_c\}$ and using \eqref{eq::tc2} combined with \eqref{eq::bluedeg-k}, the maximal degree of blue  at $\lfloor t_c\rfloor$ and $\wur_{\lfloor t_c\rfloor}$ satisfy
\begin{align} \log (D^{\max}_n(\lfloor t_c\rfloor))/\log n &=\frac{ (\tau-2)^{\frac{\bnb+\dnr}{2}}}{\tau-1} (\tau-2)^{\frac{T_b-T_r-1}{2}} (\tau-2)^{\{t_c\}}(1+o_{\Pv}(1)), \label{eq::dmax-tc} \\
\log (\wur_{\lfloor t_c\rfloor}) / \log n &= \frac{(\tau-2)^{\frac{\bnb+\dnr}{2}}}{\tau-1} (\tau-2)^{\frac{T_b-T_r-1}{2}} (\tau-2)^{-\{t_c\}}(1+o_{\Pv}(1)). \label{eq::wur-tc}
\end{align}

Neglecting the $(1+o_{\Pv}(1))$ terms, we get that the logarithm of \eqref{eq::dmax-tc} divided by \eqref{eq::wur-tc}, divided  by $|\log (\tau-2)|$, i.e.,  the distance between the two exponents is given by
 \be\label{eq::dtc2} 2\{t_c\} = 2\left\{\frac{T_b-T_r+1}{2} + \frac{\bnb-\dnr}{2}\right\}. \ee
Clearly, $2\{t_c\} \in [0,2)$. %
   At time $T_r + \lfloor t_c \rfloor+1$, when $2\{t_c\}>1$, the maximal degree of blue is increased by a factor $1/(\tau-2)$ in the exponent, while, when $2\{t_c\}<1$, the vertices `slightly below'\footnote{`Slightly below' here means $\wur_{\lfloor t_c \rfloor}(C\log n)^{-2} $: the number of vertices between degrees $\wur_{\lfloor t_c \rfloor}(C\log n)^{-2} $ and $\wur_{\lfloor t_c \rfloor}$ is then large enough so that whp at least one of them will get color blue at time $\lfloor t_c \rfloor+1$.} $\wur_{\lfloor t_c \rfloor}$ are reached both by color red and blue at the same time (at time $\lfloor t_c \rfloor+1$), and hence each of these vertices their color is red and blue with equal probability. Hence, the  maximal degree of blue can be described as
\be\label{eq::max-degree-final} \frac{\log D_n^{\max}(\infty)}{(\tau-1)^{-1}\log n}=(\tau-2)^{\frac{\bnb+\dnr}{2}} (\tau-2)^{
\frac{T_b-T_r-1}{2}} (\tau-2)^{(\{t_c\}-1)\ind\{2\{t_c\}>1\} - \{t_c\}\ind\{2\{t_c\}<1\}  }(1+o_{\Pv}(1)).\ee
Next we analyse $\{t_c\}$, and then substitute the results into this formula.
Since $T_r,T_b$ are integers, the first term in \eqref{eq::dtc2} contributes either $0$ or $1/2$ to the fractional part, depending on the parity, while the second term is between $(-1/2,1/2)$.  Hence we can distinguish four cases depending on the parity of $T_r+T_b$ and $\bnb-{\dnr} <0$ or $\bnb-{\dnr} >0$.

From \eqref{def::alpha} and \eqref{def::delta} it follows that the condition $\bnb-{\dnr}>0$ is equivalent to
\be\label{cond::alpha<bn2} \tau-1>(\tau-2)^{\bnb}+(\tau-2)^{\bnr},\ee corresponding to Cases $(E_>)$ and $(O_>)$ below. The region at which this is satisfied is above the red curve in Fig~\ref{fig::contourplot}.

 Here we list what happens with the maximal degree of blue in the four cases:
\begin{description}
\item[\namedlabel{caseodd1}{$(O_{>})$}]$T_b-(T_r+1)=2k+1$ and $\bnb>\dnr$. In this case both fractional parts in \eqref{eq::dtc2} are at least $1/2$, thus $\{t_c\}=(1+\bnb-\dnr)/2\ge 1/2$, and $\lfloor t_c \rfloor=(T_b-T_r)/2$, $2\{t_c\}\ge 1$.
The right hand side of \eqref{eq::max-degree-final} equals $(\tau-2)^{\bnb} (\tau-2)^{(T_b-T_r-2)/2}$. Note that $(T_b-T_r-2)/2\ge 0$ if $T_b-T_r\ge 2$.
See Fig~\ref{fig::odd1}.

\item[\namedlabel{caseeven1}{$(E_{>})$}]$T_b-(T_r+1)=2k$ and $\bnb> \dnr$. In this case $2\{t_c\}=\bnb-{\dnr}<1$, and $\lfloor t_c \rfloor=(T_b-T_r+1)/2$. The right hand side of \eqref{eq::max-degree-final}  becomes $ (\tau-2)^{\dnr} (\tau-2)^{(T_b-T_r-1)/2}$. Note that $(T_b-T_r-1)/2$ is integer-valued and $\ge 0$ if $T_b-T_r\ge 3$. See Fig~\ref{fig::even1}.

\item[\namedlabel{caseodd2}{$(O_<)$}]$T_b-(T_r+1)=2k+1$ and $\bnb<\dnr$. In this case $\{t_c\}=(1+\bnb-\dnr)/2$ and $\lfloor t_c \rfloor=(T_b-T_r)/2$, $2\{t_c\}<1$.
The right hand side of \eqref{eq::max-degree-final}  is $(\tau-2)^{\dnr} (\tau-2)^{(T_b-T_r-2)/2}$. Note that $(T_b-T_r-2)/2\ge 0$ if $T_b-T_r\ge 2$.
See Fig~\ref{fig::odd2}.

\item[\namedlabel{caseeven2}{$(E_<)$}] $T_b-(T_r+1)=2k$ and $\bnb<\dnr$. In this case $\{t_c\}=1+(\bnb-\dnr)/2$ and $\lfloor t_c \rfloor=(T_b-T_r-1)/2$, and $2\{t_c\}>1$.
The right hand side of \eqref{eq::max-degree-final}  becomes $(\tau-2)^{\bnb} (\tau-2)^{(T_b-T_r-1)/2}$. Note that $(T_b-T_r-1)/2\ge0$ if $T_b-T_r\ge 3$.
See Fig~\ref{fig::even2}.
\end{description}
\begin{figure}[t]
\includegraphics[width=12cm]{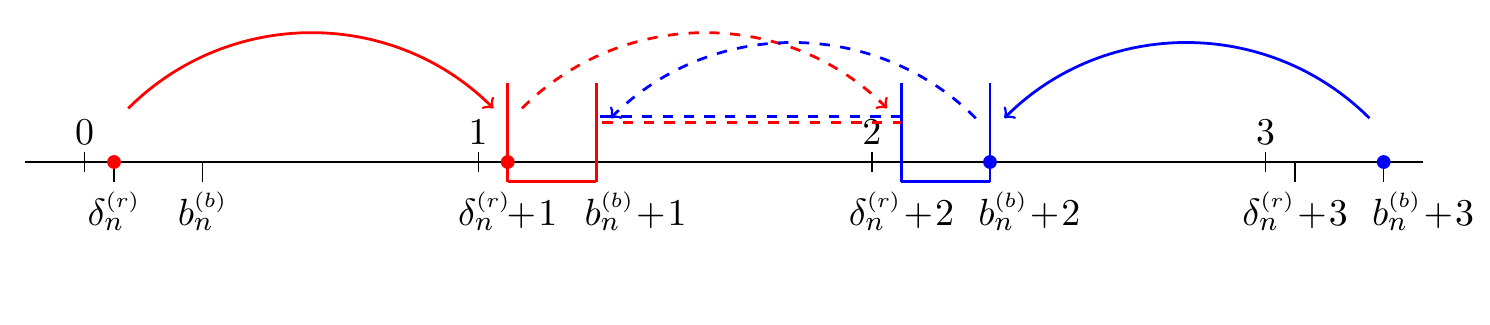}
\caption{Case $(O_>)$:  $\bnb>\dnr$ and $T_b-(T_r+1)$ is odd ($=3$ in the picture). The picture shows the exponent of $\tau-2$ in the formula for $\log$(degrees)/$\log n$. At time $T_r+1$, blue is at $\bnb+3$ and red is at $\dnr$ in the picture; the full and dashed arrows indicate their jump at time $T_r+2$ and $T_r+3$, respectively. Since $\bnb>\dnr$ and $T_b-T_r-1=3$ is odd, the distance between the two colors just before merging (at time $T_r+2$) is $\bnb +2-(\dnr+1)>1$. So, at time $T_r+3$, there is a region between them where they jump to vertices at the same time (indicated by dashed horizontal lines). Hence, the exponent of $\tau-2$ in the maximal degree of blue becomes $\bnb+(T_b-T_r-2)/2$, which is $\bnb+1$ in the figure.}\label{fig::odd1}
\end{figure}
\begin{figure}[t]
\includegraphics[width=14cm]{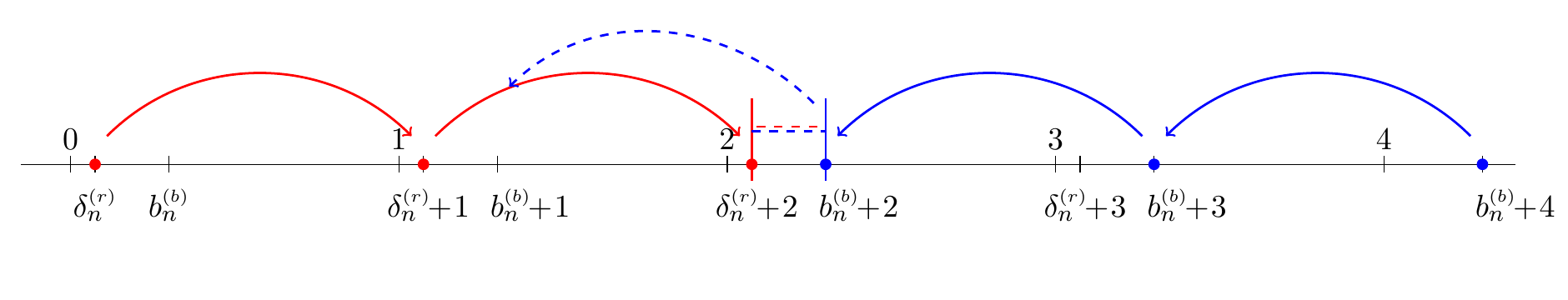}
\caption{Case $(E_>)$: $\bnb>\dnr$ and $T_b-(T_r+1)$ is even ($=4$ in the picture).
The picture shows the exponent of $\tau-2$ in the formula for $\log$(degrees)/$\log n$. At time $T_r+1$, blue is at $\bnb+4$ and red is at $\dnr$ in the picture; the full arrows indicate their jumps at time $T_r+2$ and $T_r+3$. Since $\bnb>\dnr$ and $T_b-T_r-1=4$ is even, the distance between the two colors just before merging (at time $T_r+3$) is less than $1$. So, at time $T_r+4$, there is a region where they jump to vertices at the same time (indicated by dashed horizontal lines), i.e., the exponent of $\tau-2$ in  the maximal degree of blue is $\dnr+(T_b-T_r-1)/2$, which is $\dnr+2$ in the figure.}
\label{fig::even1}
\end{figure}
\begin{figure}[ht]
\includegraphics[width=12cm]{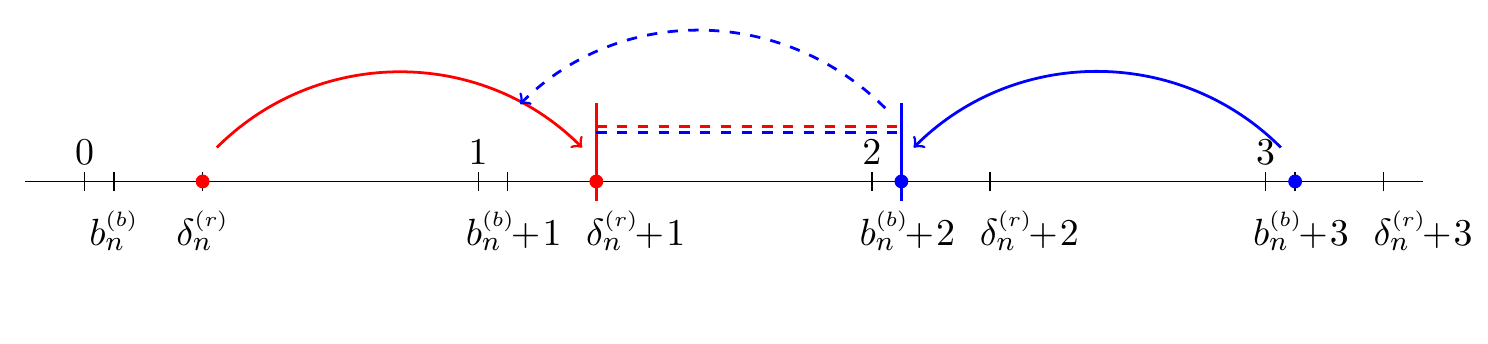}
\caption{Case $(O_<)$: $\bnb<\dnr$ and $T_b-(T_r+1)$ is odd ($=3$ on the picture). The picture shows the exponent of $\tau-2$ in the formula for $\log$(degrees)/$\log n$. At time $T_r+1$, blue is at $\bnb+3$ and red is at $\dnr$ in the picture; the full arrows indicate their jumps at time $T_r+2$. Since $\bnb<\dnr$ and $T_b-T_r-1=3$ is odd, the distance between the two colors just before merging (at time $T_r+2$) is less than $1$. So, at time $T_r+3$, there is a region where they jump to vertices at the same time (indicated by dashed horizontal lines), i.e., the exponent of $\tau-2$ in  the maximal degree of blue is $\dnr+(T_b-T_r-2)/2$, which is $\dnr+1$ in the figure.}\label{fig::odd2}
\end{figure}
\begin{figure}[t]
\includegraphics[width=14cm]{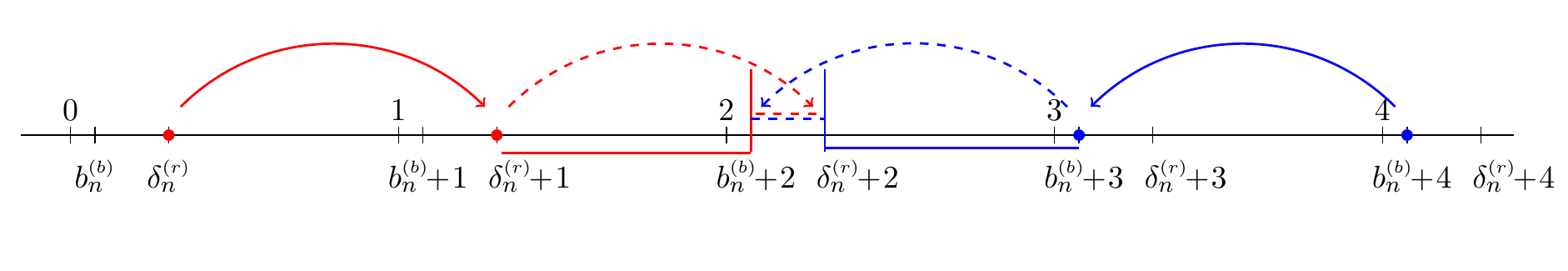}
\caption{Case $(E_<)$: $\bnb<\dnr$ and $T_b-(T_r+1)$ is even ($=4$ on the picture). The picture shows the exponent of $\tau-2$ in the formula for $\log$(degrees)/$\log n$. At time $T_r+1$, blue is at $\bnb+4$ and red is at $\dnr$ in the picture; the full and dashed arrows indicate their jumps at time $T_r+2$ and $T_r+3$. Since $\bnb<\dnr$ and $T_b-T_r-1=4$ is even, the distance between the two colors just before merging (at time $T_r+2$) is larger than $1$. So, at time $T_r+3$, there is a region where they jump to vertices at the same time (indicated by dashed horizontal lines), i.e., the exponent of $\tau-2$ in  the maximal degree of blue is $\bnb+(T_b-T_r-1)/2$, which is $\bnb+2$ in the figure.}\label{fig::even2}
\end{figure}
We are left to calculate the expressions for the maximal degree using \eqref{eq::Tj}:
\be\label{eq::t2-t1exp}(\tau-2)^{(T_b-T_r)/2} = \sqrt{\Ybn/\Yrn} (\tau-2)^{\frac{\bnr-\bnb}{2}}.\ee
Combining this with the maximal degree listed in the four cases above, it is not hard to see that
\be\label{eq::max-deg-2} \frac{\log \left(D_n^{\max}(\infty)\right)}{\log n} = \sqrt{\Ybn/\Yrn}  \frac{1}{\tau-1} h_n(\Yrn, \Ybn)(1+o_{\Pv}(1)),\ee
with
\[\ba h_n(\Yrn, \Ybn) =& \ind_{ E_< \cup O_>} (\tau-2)^{(\bnr+\bnb-1 - \ind_{O_>})/2} +\\
&+\ind_{ E_> \cup O_<} (\tau-2)^{(\bnr-\bnb-1 - \ind_{O_<})/2} ((\tau-1)-(\tau-2)^{\bnr}),   \ea \]
exactly as in \eqref{eq::h1} in the statement of Theorem \ref{thm::maxdegree}. Note that in all cases, using the representations listed in the four cases, the maximal degree is either $(\tau-2)^{\bnb+\ell}/(\tau-1)\le 1$ or $(\tau-2)^{\dnr+\ell}/(\tau-1)\le 1$ for some integer $\ell\ge 1$ when $T_b-(T_r+1)\ge 2$. Hence, the maximal degree in all cases is a random power of $n$ that is at most $(\tau-2)/(\tau-1)$.

It remains to investigate the cases when $T_b-T_r\in \{0,1\}$.

\begin{enumerate}
\item[(2)] If $T_b-T_r=0$, then in Section \ref{sc::peak} we have established that red and blue paint every vertex in the set $\wgab_1$ with equal probability.
 As a result, the maximal degree of blue whp tends to the maximal degree in $\CMD$, that is, it is of order $n^{1/(\tau-1)(1+o_{\Pv}(1))}$ (see Fig. \ref{fig::cross-1}).
\item[(3)(a)] If $T_b-T_r=1$ but $\Ybn/\Yrn<\tau-2$, then we have seen in Section \ref{sc::peak} that at time $T_r+1$, blue arrives to a few vertices with $\log$(degree)/$\log n$ that is $(\tau-2)^{\bnb}/(\tau-1)(1+o_{\Pv}(1))$, while red occupies all vertices with $\log$(degree)/$\log n$ that is $\anr=(\tau-2)^{\dnr}/(\tau-1)$.  Two scenarios are possible:
\begin{enumerate}
\item[(3)(a.1)] If $T_b-T_r=1$ and $\bnb>\dnr$ then at time $T_b=T_r+1$ red is still at higher degree vertices than blue. At $T_r+2$, blue  can increase its exponent to some vertices with $\log$(degree)/$\log n$ at most $(\tau-2)^{\dnr}/(\tau-1).$ Since $\Ybn/\Yrn>\tau-2$ implies that $\dnr>\dnb$ in this case, the vertices blue tries to occupy via crossing the mountain are already all painted red. Note that the maximal degree is a power of $n$ that is less then $1/(\tau-1)$. This case can be merged into Case (1), $(E_>)$ above.
\item[(3)(a.2)] If $T_b-T_r=1$ and $\bnb<\dnr$ then red has already occupied lower degree vertices, so the maximal degree of the blue remains at exponent $(\tau-2)^{\bnb}/(\tau-1)$. Further, again, since $\Ybn/\Yrn>\tau-2$ implies that $\dnr>\dnb$ in this case, the vertices blue tries to occupy via crossing the mountain are already all painted red. Note again that this is a power that is less than $1/(\tau-1)$. This case can be merged into Case (1), $(E_<)$ above.
\end{enumerate}
\item[(3)(b)] If $T_b-T_r=1$ and $\Ybn/\Yrn>\tau-2$. Note that $\Ybn/\Yrn>\tau-2$ implies $\dnb>\dnr$, and hence the maximal degree vertex blue can paint depends on the fact if $\bnb>\dnr$ or $\bnb<\dnr$:
The same argument works here as for Cases (3)(a.1) and (3)(a.2), i.e., if $\bnb>\dnr$, then the maximal degree of blue increases its exponent to $(\tau-2)^{\dnr}/(\tau-1)$, while if $\bnb<\dnr$, then the maximal degree of blue remains at the exponent $(\tau-2)^{\bnr}/(\tau-1)$.  Geometrically, this means that the line where blue jumped from to the mixed area on Figure \ref{fig::cross-2} can be both above or below the bottom of the all-red area on the top of the mountain.
\end{enumerate}

Summarizing, if $\Ybn/\Yrn <\tau-2$, then the maximal degree of blue is always less than $1/(\tau-1)$, and is described by the cases $E_<, E_>, O_<, O_>$ above. Dividing by $h_n(\Yrn, \Ybn)$ we see that
\[ \frac{\log (D_n^{\max}(\infty))}{(\tau-1)^{-1}h_n(\Yrn, \Ybn)\log n} = \sqrt{\Ybn/\Yrn}  (1+o_{\Pv}(1)),
  \] and the right hand side tends to $\sqrt{Y_b/Y_r}$ in distribution.

Note that we did not investigate one issue, namely, what part of the proof was an upper and what part was a lower bound. First, Section \ref{sc::climbup} establishes the existence a blue vertex at time $t(n^{\vr}) +i$ of degree at least  $u_i^{\sss{(b)}}$. On the other hand, Lemma \ref{lem::badpaths} states that whp blue does not paint  any vertex with degree higher than $\wit u_i^{\sss{(b)}}$ at the same time. Hence, on the blue part, the estimates we used were both upper and lower bounds, up to the $(1+o_{\Pv}(1))$ factors that exactly describe the ratio $\log (\widehat u_i^{\sss{(b)}}/u_i^{\sss{(b)}})$, and they tend to zero when divided by $\log n$ in the statement of the theorem.

One might imagine that the avalanche might actually roll down `faster' than described by the layers $\wgar_\ell$s. To eliminate this problem we argue as follows: even though it might happen that the red avalanche occupies at time $T_r+\ell$  some lower degree vertices than $\wur_\ell$, but if so, only a small number of them, not all of them.\footnote{A similar argument could be used as in Lemma \ref{lem::red-intervals} to show that the number of these vertices is small, at least proportionally to the number of vertices in any given layer. However, the argument given here is easier.}
If it would meet blue earlier than it should (i.e., at time $\lfloor t_c \rfloor$ or $\lfloor t_c \rfloor+1$ as described by the four cases $E_<, E_>, O_<, O_>$ above), then the path formed by vertices in the red fast avalanche and the blue climbing path together would establish a path to the top of the mountain for blue that violates Lemma \ref{lem::badpaths}. Hence, this will not happen whp. As a result, our bounds on the maximal degree of blue are both upper and lower bounds, i.e., they hold whp.
\end{proof}

\section{Coexistence when $\Ybn/\Yrn>\tau-2$.}\label{sc::coexistence}
\subsection{Excursion to absolute continuity and some general estimates}
Before we move on to the proof of coexistence, we need some preliminaries. In Proposition \ref{prop::no-gap} below, we collect some results on the generating function of $D, D^\star$ and on the limiting random variable $Y$ of the branching process described in in Definition \ref{def::limit-variables}.
To show coexistence in the competition model, we will later need that  $Y$  has an absolutely continuous distribution with support including some interval $(0, K)$, for some constant $K>0$. Hence the following assumption:

\begin{assumption}\label{assume::convex}
Let us write $h^{\star}(s):=\sum_{k=1}^\infty \Pv(D^{\star}\!=k) s^k$ where $D^\star$ follows the distribution \eqref{def::size-biased1}.
We assume that $f(t):=-\log( 1- h^\star(1-\e^{-t}) )$ is convex or concave on $\R^+$.
 \end{assumption}

\begin{proposition}\label{prop::no-gap} Recall the degree distribution $D$ from \eqref{eq::F} and its size-biased version $D^\star$ from \eqref{def::size-biased1}.
Then

(i) the generating functions $h(s):=\sum_{k=2}^\infty \Pv(D=k) s^k$ and $h^{\star}(s):=\sum_{k=1}^\infty \Pv(D^{\star}=k) s^k$ satisfy
\be\label{eq::gen-func}  1-h(s)=(1-s)^{\tau-1} L(1/(1-s)), \quad 1-h^{\star}(s)=(1-s)^{\tau-2} L^\star(1/(1-s)), \ee
where $0<L(x), L^\star(x)<\infty$ are bounded slowly varying functions.

(ii) Consider $\wit Z_k$, the size of the $k$th generation in a Galton-Watson branching process with offspring distribution $D^\star$. Under Assumption \ref{assume::convex},  $\wit Y=\lim_{k\to \infty} (\tau-2)^k \log \wit Z_k$ has an absolutely continuous distribution function $J(x)$ with full support on $(0, \infty)$, that can be written as
\be\label{eq::jx} J(x)=1-\exp\{  - \Delta(x) \},\ee
where $\Delta(x)$ is continuous, strictly increasing, and can be written as
 \[ \Delta(x) =   \int_0^x p(t) \mathrm dt,\]
 where $p(t)$ is strictly positive; increasing if $f$ is convex, and decreasing if $f$ is concave. Further, $\lim_{x\to \infty} \Delta(x)/x=1$, in agreement with \eqref{eq::exp-tails-Y}.
\end{proposition}

\begin{proof}The proof of (i) can be found in general in \cite{Feller71}, and the fact that the slowly varying functions are bounded is a consequence of the condition in \eqref{eq::F}, i.e., the slowly varying function hidden in the distribution function is bounded.
Part (ii) is a rewrite of \cite[Theorem C, Theorem 4]{Sene74}.
\end{proof}
The rest of this section holds under the milder condition that Assumption \ref{assume::abs-cont} holds, but we will use the notations $J(x)$ and $\Delta(x)$ as in Proposition \ref{prop::no-gap}. Note that already Assumption \ref{assume::abs-cont} implies that $J(x)$ is strictly increasing, hence it can also be written in a form as in \eqref{eq::jx}, with a strictly increasing, though not necessarily continuous, $\Delta(x)$.

Note that if we would like to approximate the exploration of the graph from a uniformly chosen vertex $w$ in $\CMD$, then the root of the approximating BP has offspring distribution $D$, and all the further individuals have offspring from distribution $D^\star$. We denote the number of individuals in this BP by $Z_k^w$. To identify the limit random variable $Y^w:=\lim_{k\to \infty} (\tau-2)^k \log Z_k^w$, we use \cite[Lemma 2.4]{BarHofKom14}, stating
$Y^w=(\tau-2) \max_{i=1, \dots D_w} \wit Y^{(i)}$. From this representation and Proposition \ref{prop::no-gap} it is obvious that under Assumption \ref{assume::abs-cont} $Y^w$ also has full support on $(0, K)$, with strictly increasing distribution function $J^w(x)$, since
\[ J^w(x):=\Pv(Y^w \le x) = \sum_{k=2}^\infty \Pv(D=k) J(x)^k,\]
by dominated convergence and the fact that each term is strictly increasing by Assumption \ref{assume::abs-cont}. Hence,
$\Delta^w(x) := - \log (1- J^w(x))$ is strictly increasing.
Further, using \eqref{eq::jx} and \eqref{eq::gen-func},
\[ J^w(x) = \Ev[ (1- \e^{-\Delta(x/(\tau-2))})^D  ]= 1- \e^{ -(\tau-1) \Delta(x/(\tau-2)) } L(\e^{\Delta(x/(\tau-2))}),\]
where $L$ is defined in \eqref{eq::gen-func}.
Since $L(x)$ is a bounded positive function, $Y^w$ has also exponential decay with exponent
\be\label{eq::jw-tail} \lim_{x\to \infty} \Delta^w(x)/x=(\tau-1)/(\tau-2). \ee
Next, we investigate the relationship between the sum of the degrees and the maximum of the degrees in each generation in the branching process with root $w$. We write $\CG_k$ for the set of individuals in the $k$-th generation. Clearly $|\CG_k|=Z_k^w$.
\begin{claim}\label{cl::max-sum-relation}
Let $M_k:=\max_{i\in \CG_k} D_i^\star$. On the event $\lim_{k\to \infty}(\tau-2)^k \log Z_k^w = Y^w$, the limit  $\lim_{k\to \infty}(\tau-2)^k \log M_k = Y^w/(\tau-2)$ holds.
\end{claim}
\begin{proof} Intuitively, the statement of the lemma should hold since $M_k\approx Z_{k+1}^w$. In a bit more detail,
similarly as in the proof of Claim \ref{cl::stoch-dom-alpha-stable}, we can pick $0\le b_1\le b_2$ so that  $b_1+ X^{\sss{(1)}}\  {\buildrel {d}\over \le } \  D^\star\  {\buildrel {d}\over \le } \  b_2+ X^{\sss{(2)}}$, where the random variable $X^{\sss{(1)}}$ has distribution function $1-c_1^\star/ x^{\tau-2}$ on $[0, \infty)$, and $X^{\sss{(2)}}$ has distribution function $1-C_1^\star/x^{\tau-2}$ on $[0, \infty)$.
Both of these random variables are totally asymmetric stable distributions with skewness $\kappa=1$, shifts $b_1$ and $b_2$, and some scale parameters $c^{\sss{(1)}}$ and $c^{\sss{(2)}}$, respectively.
Then for any $m\in \N$, $x\in \R_+$
\be\label{eq::stoch-dom-above}  \Pv\left( \max_{i=1}^m X_i^{\sss{(1)}} \le m^{1/(\tau-2)} x \right) = \left(1- \frac{c_1^\star} {x^{1/(\tau-2)} m}\right)^m \le \exp\{ -c_1^\star/ x^{1/(\tau-2)}   \}, \ee
while
\be\label{eq::stoch-dom-below}  \Pv\left( \max_{i=1}^m X_i^{\sss{(2)}} \le m^{1/(\tau-2)} x \right) = \left(1- \frac{C_1^\star} {x^{1/(\tau-2)} m}\right)^m \ge \exp\{ -\tfrac12 C_1^\star x^{1/(\tau-2)}   \}, \ee
for large enough $x\in \R$.
Let us denote random variables with distribution function given by the right hand side of \eqref{eq::stoch-dom-above} and \eqref{eq::stoch-dom-below} by $M^{\sss{(1)}}$ and $M^{\sss{(2)}}$, respectively. Then
\[ \frac{b_1}{m^{1/(\tau-2)}} + M^{\sss{(1)}}\  {\buildrel {d}\over \le }\  \frac{\max_{i=1}^m D_i^\star}{m^{1/(\tau-2)}}\  {\buildrel {d}\over \le } \  \frac{b_2}{m^{1/(\tau-2)}}+ M^{\sss{(2)}}. \]
On the other hand, using the method in the proof of Claim \ref{cl::stoch-dom-alpha-stable}, the stochastic domination in \eqref{eq::sum-domination} can be used to estimate the moment generating function of
$\sum_{i=1}^m D_i^\star/m^{1/(\tau-2)}$ which yields
\[ b_1+ X^{\sss{(1)}}\  {\buildrel {d}\over \le }\  \frac{\sum_{i=1}^m D_i^\star}{m^{1/(\tau-2)}}\  {\buildrel {d}\over \le }\  b_2+ X^{\sss{(2)}}\  \]
Combining the last two estimates yields
\be\label{eq::ratio-domination} \frac{M^{\sss{(1)}}}{b_2+ X^{\sss{(2)}}}\ \ {\buildrel {d}\over \le }\ \frac{M_k}{Z_{k+1}}\  {\buildrel {d}\over \le }\  \frac{M^{\sss{(2)}} + b_2/m^{1/(\tau-2)}}{b_1+X^{\sss{(1)}}}, \ee
where all random variables are positive a.s. Hence, writing
\[ (\tau-2)^k\log M_k = (\tau-2)^k \log \left(\frac{M_k}{Z_{k+1}}   \right) + \frac{1}{\tau-2}(\tau-2)^{k+1} \log Z_{k+1},  \]
 we conclude that the first term tends to zero by the stochastic dominations in \eqref{eq::ratio-domination}, and the second term tends to $\wit Y^w / (\tau-2)$ by Theorem \ref{thm::davies}.
\end{proof}

\subsection{Coexistence}
Next we turn to the proof of coexistence when $\Ybn/\Yrn>\tau-2.$

To understand the proportion of vertices that blue eventually paints, we use the usual trick that
\be\label{eq::empirical-blue} \frac{\CB_\infty}{n}= \frac{1}{n}\sum_{v\in [n]} \ind_{\{ v \text { is eventually blue} \}},\ee
which can be interpreted as the empirical measure of blue. To show coexistence it is thus enough to show that this expression is strictly positive with positive probability. We do this via the first and second moment method. The first moment gives the probability (conditional on $n, \Yrn, \Ybn$) that a uniformly chosen vertex is eventually painted blue, while for the second moment we need to investigate the probability that two uniformly chosen vertices are both eventually blue.

We have seen in Section \ref{sc::slopedown} that when $\Ybn/\Yrn>\tau-2$ there is a `mixed' avalanche occupying lower and lower degree vertices. In both cases (i.e., $T_b=T_r$ or $T_b=T_r+1$), with each additional time unit, there is a new interval on the log-log scale that gets colored, namely $\CM_\ell \union \CR ed_\ell$ at time $T_b+\ell$. In both cases vertices in the `top' part of this interval (in $\CM_\ell$) are red and blue with equal probability (Lemma \ref{lem::independence}), while almost all vertices in the `bottom' part of the interval (in $\CR ed_\ell$) are painted red (Lemma \ref{lem::red-intervals}), see Fig. \ref{fig::avalanches}.

The idea to prove coexistence is as follows: we have to avoid the problem of the degrees getting too small (and as a result, our estimates getting too noisy). Hence, we only `run' this mixed avalanche as long as it stays in relatively high vertices, the `core' of the graph, that we call $\mathrm{Core}_n$. We choose this $\mathrm{Core_n}$ in a way that `coincides' with the layers of the avalanche. After the avalanche reaches the boundary of the core, we stop it. As a result, every vertex in the core has been painted whp.

Then, we investigate how the neighbourhood of a random vertex $w$ `enters' the painted core of the graph: we approximate the neighborhood of the vertex by a branching process that is described in Definition \ref{def::limit-variables} and we let this BP grow until the random generation when it first hits the core of the graph. Depending on what the degrees are, the vertices in this last generation might be red, red or blue with equal probability, or uncolored. This gives a partial coloring of the last generation of the BP.

These colors then `percolate' through the BP tree towards the root, following the rules of the coloring scheme, that is, if an uncolored vertex at any time has neighbors of only one color, then it takes that color in the next time step, while if it has neighbors of more than one color, then it picks a color with equal probability. We will show that in this random coloring scheme, the root of the BP gets both colors with strictly positive probability that depends on the ratio $q=\Ybn/\Yrn$, and the probability that the root is blue tends to zero as $q \searrow (\tau-2)$. After this have been shown, a second moment method finishes the proof.

We start to investigate the first moment. In this section we condition on $\Yrn$ and $\Ybn$. Recall the definition of $\wur_\ell, \wub_\ell$ from \eqref{eq::wideui_recursion} and $\wgar_\ell, \wgab_\ell$ from \eqref{eq::wgar}. Throughout, we write $\wit\ind=\ind_{\{T_b=T_r\}}$ and $\simp$ for whp equality up to factor of finite powers of $\log n$.

Note that if $\ell=\nu\log\log n / |\log(\tau-2)| +1 + x_n$, for some $\nu<1$, and $x_n$ is chosen so that the expression is an integer, then \eqref{eq::wideui_recursion} gives, for $j=r,b$,
\be \label{eq::core-ell} \widetilde u_\ell^{\sss{(j)}} = \e^{(\log n)^{1-\nu}\alpha_n^{\sss{(j)}} (\tau-2)^{x_n}} (C \log n)^{1/(3-\tau)} (1+o_{\Pv}(1)).\ee
 Also, by Lemma \ref{lem::red-intervals}, the probability that a vertex in a red interval is not red is of order $\simp \wur_{\ell+\wit\ind}/ \wub_{\ell}$. We would like to keep this probability small, and we also would like that $\wub_{\ell}\simp (\wur_{\ell+\wit\ind})^{\anb (\tau-2)^{\wit\ind-1} / \anr}$ holds.
This holds as long as $0<\nu<1$. Note that for any fixed positive $\nu<1$, $\wur_\ell$ and $\wub_\ell$ are sub-polynomial in $n$.

So, let us fix a $0\!<\!\nu\!<\!1$ and then set $\ell_{\max}:=\lfloor \nu \log\log n / |\log (\tau-2)|\rfloor$, and define
\be\label{def::core-n} \mathrm{Core}_n=\{ v\in \CMD: D_v> \wur_{\ell_{\max}} \}.\ee
From now on we simply write $Q:=\wur_{\ell_{\max}}.$ Note that the bottom of the last mixed interval is at degree $\wub_{\ell_{\max}-\wit\ind}\simp Q^{\anb(\tau-2)^{\wit\ind-1}/\anr}$, which implies that almost all vertices with degree in the interval
$[Q, \simp\!Q^{\anb(\tau-2)^{\wit\ind-1}/\anr} )$ are painted red,  while  vertices with degree in the interval $[\simp Q^{\anb(\tau-2)^{\wit \ind-1}/\anr}, \simp\!Q^{1/(\tau-2)})$ are colored red and blue with equal probability.
We will simply write
\be\label{eq::color-rule}  [Q, Q^{\gamma}) \in \CR ed,  \quad  [Q^{\gamma}, Q^{1/(\tau-2)}) \in \mathcal Mix,  \ee
with \be\label{def::gamma}\gamma:=\anb(\tau-2)^{\wit \ind-1}/\anr.\ee
Note that $\gamma \in (1,(\tau-2)^{-1}).$
By Lemma \ref{lem::red-intervals}, the proportion of potentially blue vertices in $\CR ed_{\ell_{\max}}=[Q, Q^{\gamma}) $ is
\be\label{eq::p-error} p_e:= \frac {\wur_{\ell_{\max}}}{\wub_{\ell_{\max} - \wit \ind}} (1+o_{\Pv}(1))\simp Q^{1-\gamma}.\ee
In this section, we write $\Pv_\gamma(.):= \Pv(.| \gamma, Q, \Yrn, \Ybn), \Ev_\gamma[.]:= \Ev[.| \gamma, Q, \Yrn, \Ybn]$.

\begin{remark}\label{rem::q-gamma-relation}\normalfont
 It is intuitively clear that $q\to \Ybn/\Yrn\approx(\tau-2)$ if and only if $\gamma\to\tfrac{1}{\tau-2}$, while $\Ybn/\Yrn \approx 1$ implies $\gamma \approx 1$.
  To see this, let $\ve, \ve', \ve'', \ve'''$ be small positive numbers.
Recall that $\anb, \anr \in (\tfrac{\tau-2}{\tau-1}, \tfrac{1}{\tau-1}]$ (see \ref{def::alpha}).

 When $T_r=T_b$, then $\gamma =\anb/\anr= \tfrac{1-\ve}{\tau-2}$ is only possible if $\anb= \tfrac{1-\ve'}{\tau-1}$, $\anr\approx  \tfrac{\tau-2+\ve''}{\tau-1}$ and this in turn implies $\bnb=1-\ve'$ and $\bnr = \ve''$. By \eqref{eq::key-2}, this is only possible if $\Ybn/\Yrn = (\tau-2)^{1-\ve' -\ve'' }$ and $n$ is so that $\tfrac{\log\log n - \log Y_j^{\sss{n}}}{ |\log (\tau-2)}$ is close to an integer for both $j=r,b$.

 When $T_b=T_r+1$, then since $\gamma=\anb/(\anr(\tau-2)) $ in this case, and $\anb<\anr$, $\gamma = \tfrac{1-\ve}{\tau-2}$ is only possible if $\anb = \anr-\ve'$, that is, if $\bnb = \bnr- \ve''$, which, combined with $T_r=T_b+1$ again implies that this is only possible if  $\Ybn/\Yrn \approx (\tau-2)^{1-\ve'''}$.

Summarizing, we see that if $\gamma\approx 1/(\tau-2)$ then $ q=\Ybn/\Yrn\approx \tau-2$. The other direction can be treated similarly, as well as the equivalence between $\gamma \approx 1$ and $q=\Ybn/\Yrn\approx 1$.
\end{remark}
 To investigate the color of a uniform vertex $w$, we couple its local neighborhood to an independent copy of a  branching process described in Section \ref{sc::BP}, independent of the blue and red BPs. We will denote this BP by $\mathrm{BP}_w$ and run it until the stopping time when its maximum degree reaches $Q$.
 This coupling can be achieved by extending \cite[Lemma 2.2]{BarHofKom14} to three vertices, where first we couple the degrees of the local neighbourhoods of the source vertices to the red and blue BP, determining $\Ybn, \Yrn$, then couple the degrees of the third vertex to a BP, till it reaches maximum degree $Q$, and then we can continue with everything else. Note that since $Q=o(n^\ve)$ for every $\ve>0$, this can easily be done, since the maximal degree is the same order of magnitude as the total number of vertices in the next generation by Claim \ref{cl::max-sum-relation}, and hence the size of the $\mathrm{BP}_w$ at stopping will be still sub-polynomial in $n$. Hence, the coupling can be done with coupling error that tends to zero with $n$.

 Recall that every vertex except the root has degree from distribution $F^\star$ from \eqref{def::size-biased1}. We write $D^\star_x$ for the degree of vertex $x$ in $\mathrm{BP}_w$. Denote the set of vertices in generation $i$ of $\mathrm{BP}_w$ by $\CG_i$ and define the stopping time $\kappa$ as
\[\kappa:= \inf\{j: \max_{x \in \CG_j} D_x^\star \ge Q\}.\]
This definition of $\kappa$ ensures that $\kappa$ is the first generation where the local neighbourhood of vertex $w$ meets colored vertices, i.e., the shortest path from $w$ to  $\mathrm{Core}_n$ and hence \emph{to any of the  colored vertices} is of length $\kappa$:
\be\label{eq::shortest-kappa} \CD(w, \mathrm{Core}_n)=\kappa. \ee
As a consequence, vertex $w$ will be coloured \emph{exactly} $\kappa$ time unit later than the coloring of the last interval in $\mathrm{Core}_n$ (that is, $w$ is colored at time $T_r+\ell_{\max}+ \kappa$).

We color vertices in $\CG_{\kappa}$ so that their coloring in $\mathrm{BP}_w$ corresponds to the coloring of $\mathrm{Core}_n$. To provide an upper and a lower bound on the number of blue vertices in $\CG_\kappa$, we describe two colorings that we couple together:
  By Lemmas \ref{lem::independence} and \ref{lem::red-intervals} we color vertices of $\CG_\kappa$ \emph{independently} of each other according to following rules, where $p_e$ is from \eqref{eq::p-error}:

\begin{framed}
\emph{Starting Rule 1:} a vertex $x\in \CG_\kappa$ gets color
\begin{enumerate}[(i)]
\item  red when $D_x \in [Q, Q^{\gamma})$,
 \item red or blue with equal probability when $D_x\ge Q^{\gamma}$,
 \item neutral when $D_x<Q$.
\end{enumerate}	
\end{framed}
\begin{framed}
\emph{Starting Rule 2:}
 a vertex $x\in \CG_\kappa$ gets color
\begin{enumerate}[(i)]
\item  red with probability $1-p_e$ and blue with probability $p_e$ when $D_x \in [Q, Q^{\gamma})$,
 \item red or blue with equal probability when $D_x\ge Q^{\gamma}$,
 \item neutral when $D_x<Q$.
\end{enumerate}
\end{framed}
Neutral color means no coloring, that is, the vertex is not in the core of the graph.

The two rules can be naturally coupled to the coloring of the core and also to each other: we can use a $p_e$-coin flip in Rule 2 to decide what happens when $D_x\in [Q, Q^{\gamma})$. Under this coupling, the number of blue vertices in Rule 1 and in Rule 2  is a stochastic lower and upper bound on the blue vertices in the intersection of the last generation of  $\mathrm{BP}_w$ and  $\mathrm{Core}_n$, respectively, due to Lemmas \ref{lem::independence} and \ref{lem::red-intervals}.

After we colored $\CG_\kappa$, we color the vertices independently of each other in subsequent generations $\kappa-1, \kappa-2, ...., 1, 0$ as follows, (for both starting rules):

\begin{framed}
\emph{Flow rule:}
If a vertex $x\in \CG_i$ has children in $\CG_{i+1}$ that
\begin{enumerate}[(i)]
\item are all neutral, vertex $x$ gets color neutral,
\item  are all either neutral or red, then it gets color red,
\item  are all either neutral or blue, then it gets color blue,
\item have both colors red and blue, then it gets color red and blue with equal probability.
\end{enumerate}	
\end{framed}

Note that this rule exactly corresponds to the rule that we set at the beginning of the paper for the spread of each color: if a vertex gets color $\CC$ at time $t$, it colors its not yet colored neighbors to color $\CC$ at time $t+1$, and succeeds to do so for each not yet colored neighbor $z$ unless $z$ has a neighbor of the other color as well, in which case $z$ gets color red and blue with equal probability. We can rewrite this rule from the point of view of $z$: $z$ stays neutral as long as it has no colored neighbors; and if some neighbors become colored at time $t$, $z$ takes the same color at time $t+1$ if it is a unique color (only red or only blue) while picks a color with equal probability otherwise.

	
Let us introduce some notation: we write $Z_i:= |\CG_i|$,  $M_i:=\max_{x \in \CG_i} D_x^\star$,
so that the definition of $\kappa$ is equivalent to $\kappa=\inf\{ i: M_i\ge Q\}$.

To start with, let us introduce \be\label{eq::Y_k^w} Y_k^w:=(\tau-2)^{k+1} \log M_{k},\ee
and then we have  $Y_k^w\toas Y^w$ by Claim \ref{cl::max-sum-relation}. Rewrite this formula for $k=\kappa$, and compare it to the value $Q$:
\be\label{eq::mkappa} M_{\kappa}=\exp\{  Y_{\kappa}^w (\tau-2)^{- (\kappa+1)}\}\ge Q. \ee The solution to this inequality, conditioned on $Y_\kappa^w$, is given by
\be\label{eq::kappa}\kappa=\left\lceil\frac{\log\log Q-\log Y_\kappa^w}{|\log (\tau-2)|}-1\right\rceil := \frac{\log\log Q-\log Y_\kappa^w}{|\log (\tau-2)|}-c_Q, \ee
with $c_Q:=\left\{\frac{\log\log Q-\log Y_\kappa^w}{|\log (\tau-2)|}\right\}\in[0,1).$

In what follows, we analyse the structure of the BP under the conditioning that we stopped it at $\kappa$.
This conditioning implies that
\begin{enumeratei}
\item the maximal degree vertex in $\CG_\kappa$ has degree $M_{\kappa}\ge Q$;
\item all other vertices in $\CG_\kappa$ have degree $\le M_{\kappa}$;
\item all vertices in $\CG_i, \  i<\kappa$ have degree $<Q$.
 \end{enumeratei}
Without loss of generality we can assume that the maximal degree vertex $v^\star$ is unique in $\CG_{\kappa}$ (if not, we pick the `leftmost' one).
Let us call the unique path to the root from $v^\star$ the \emph{ray}, and let us number its vertices backwards, that is, $v^\star:=v_0$, $v_1$ is the parent of $v^\star$, $v_2$ is the parent of $v_1$, etc, finally, $v_{\kappa}:=w$.
 \begin{figure}[ht]
\includegraphics[width=\textwidth]{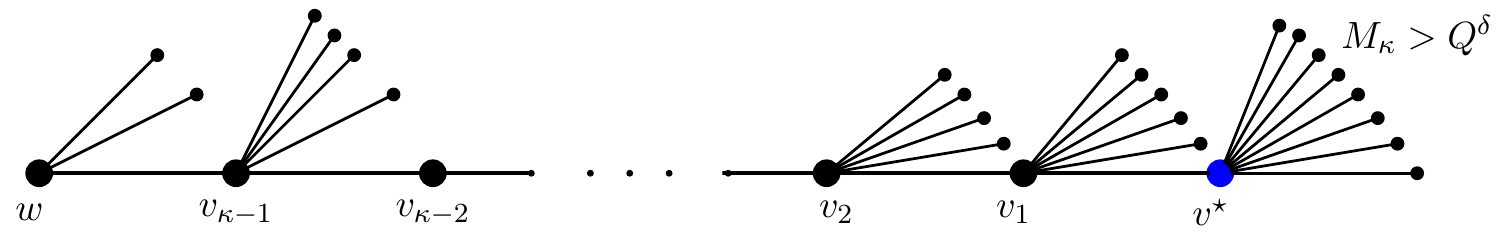}
\end{figure}

We would like to show that $\Pv(w \text{ is red})$ and $\Pv(w \text{ is blue})$ are both strictly positive.
For the first, it is enough to show that the probability that there are no blue vertices at all in $\CG_\kappa$ is strictly positive. For the second, it is enough to show that the probability $r_{\kappa}$ that there is a blue sibling of $v_{\kappa-1}$ is strictly positive, since then, $w$ gets color blue with probability at least $r_{\kappa}/2$.


As a preliminary result, we show that under both starting rules, both $\Pv(v^\star \text{ is blue})$ and $\Pv(v^\star \text{ is red,\ } D_{v^\star} <Q^\gamma)$ are strictly positive (and hence both strictly less than $1$). Note that under Starting Rule 1, the latter event implies that \emph{all} vertices in $\CG_\kappa$ are colored red, hence, $\Pv(v^\star \text{ is red}, D_{v^\star} < Q^\gamma)$ gives a lower bound on $\Pv(w \text{\ is red})$.
  \begin{lemma} \label{lem::q-delta} Let $\gamma\in (1,1/(\tau-2))$ and apply Starting Rule 1 or 2 for coloring the generation $\kappa$ of the stopped BP_w. Then there exist strictly positive numbers $0<q_{\gamma}, q_\gamma^{\text{red}}< 1$ such that
 \be\label{eq::q-delta-lem}  \Pv_{\gamma}(v^\star \text{ is blue\,}) \ge q_{\gamma}, \quad \Pv_{\gamma}(v^\star \text{ is red,\ } D_{v^\star} <Q^\gamma) \ge q_{\gamma}^{\text{red}}.\ee
Further, $q_\gamma \searrow 0$ and $q_\gamma^{\text{red}} \nearrow 1$ as $\gamma\nearrow 1/(\tau-2)$.
  \end{lemma}

  \begin{proof}
Under Starting Rule 1, $ \{ D_{v^\star}<Q^\gamma\} \subset \{v^\star \text{is red}\}$, hence,
\[ \Pv(D_{v^\star} < Q^\gamma) =  \Pv_{\gamma}(M_{\kappa} < Q^{\gamma}), \quad \Pv(v^\star \text{ is blue}) \ge  \Pv_{\gamma}(M_{\kappa} \ge Q^{\gamma})/2. \]
Using \eqref{eq::mkappa} and \eqref{eq::kappa} we get that $M_{\kappa}=Q^{(\tau-2)^{c_Q-1}}(1+o_\Pv(1))$, hence $M_{\kappa} < Q^{\gamma}$ if
\[ (\tau-2)^{c_Q-1}<\gamma.\]
This holds if and only if \be\label{eq::crit} c_Q > 1-\frac{\log \gamma}{|\log (\tau-2)|}:=\wit\gamma. \ee
Note also that $c_Q\in [0,1)$ by definition, and since $\gamma\in [1, 1/(\tau-2))$ whp, $\wit \gamma \in [0,1)$ whp.
By the definition of the fractional part,
\be\label{eq::only-red}\ba  \Pv_{\gamma}(M_\kappa <Q^\gamma) &= \Pv_{\gamma}( c_Q \in (\wit\gamma, 1) )\\
&= \Pv_{\gamma}\left(  Y_\kappa^w \in \bigcup_{m\in \N}\left((\tau-2)^{m+1} \log n , (\tau-2)^{m+\wit \gamma} \log n\right) \right). \ea\ee
Similarly,
\be\label{eq::not-only-red} \ba \Pv_{\gamma}(M_\kappa \ge Q^\gamma ) &= \Pv_{\gamma}( c_Q \in [0, \wit\gamma) )\\
&= \Pv_{\gamma}\!\left( \! Y_\kappa^w \in\! \bigcup_{m\in \N}\!\left((\tau-2)^{m+\wit \gamma} \log n , (\tau-2)^{m} \log n\right) \!\right)\\
&= \Pv_{\gamma}(Y_\kappa^w \in \bigcup_{m\in \N}\!\! A_{n,\wit\gamma}(m)). \ea\ee
We need to show that both \eqref{eq::only-red} and \eqref{eq::not-only-red} are strictly positive for all $n$ and $\gamma\in(1, (\tau-2)^{-1})$.

We only deal with $ \eqref{eq::not-only-red}$, but \eqref{eq::only-red} can be handled analogously. Note that since $Y_\kappa^w=Y_{\kappa(n)}^w \toas Y^w$ as $n\to \infty$, \eqref{eq::not-only-red} can be rewritten in the form
$\Pv(Y_n \in A_{n,\wit\gamma})$, where $Y_n=Y_{\kappa(n)}^w$,  and $A_{n,\wit\gamma}:=\cup_{m\in \N} A_{n,\wit\gamma}(m)$. The problem is that $A_{n,\wit\gamma}\subset \R^+$ does not converge with $n$.
We can overcome this issue as follows: first of all, recall that one of our assumptions was that $(0,K) \subset \mathrm{supp} (Y^w)$ for some $K>0$. Let us assume wlog that $K>1$: if this does not hold, replace $1$ in the following proof by any number that is smaller then $K$. Let us restrict our interest to the interval $[0,1]$, i.e., we investigate only those $A_{n,\wit\gamma}(m)$ for which $(\tau-2)^m\log n<1$.
We can then write
\be\label{eq::yn-decompose} \Pv_{\gamma}(Y_n \in A_{n, \wit\gamma} \ind_{[0,1]}) \ge  \Pv_{\gamma}(Y^w \in A_{n, \wit \gamma} \ind_{[0,1]} ) - \Pv_{\gamma}(Y^w \in A_{n, \wit \gamma}\ind_{[0,1]}, Y_n \notin A_{n, \wit \gamma}\ind_{[0,1]}). \ee
Our goal is to show that the second term is sufficiently small for large enough $n$. Below in \eqref{eq::q-delta} we aim to show  $\Pv_{\gamma}(Y^w \in A_{n, \wit \gamma} \ind_{[0,1]} )\ge r_{\wit\gamma}>0$. Given this number $r_{\wit\gamma}$, we show
\be \label{eq::second-error}  \Pv_{\gamma}(Y^w \in A_{n, \wit \gamma}\ind_{[0,1]}, Y_n \notin A_{n, \wit \gamma}\ind_{[0,1]}) \le r_{\wit \gamma}/2. \ee
For this, let us introduce for $\ve>0$
\be\label{eq::and-epsilon}\ba A_{n, \wit\gamma}^{-\ve}(m):&= ((\tau-2)^{m+\wit\gamma-\ve} \log n, (\tau-2)^{m+\ve} \log n), \\
A_{n, \wit\gamma}^{-\ve}&:= \bigcup_{m\in \N}A_{n, \wit\gamma}^{-\ve}(m).\ea\ee
Clearly, $A_{n, \wit\gamma}^{-\ve}(m)\subset A_{n, \wit\gamma}(m).$
Using this notation, we have the upper bound
\be\label{eq::error-n-d} \ba \Pv_{\gamma}( Y^w \in A_{n, \wit \gamma}\ind_{[0,1]},& Y_n \notin A_{n, \wit \gamma}\ind_{[0,1]})  \le \Pv_{\gamma}( Y^w \in \ind_{[0,1]} A_{n, \wit\gamma}\setminus A_{n, \wit\gamma}^{-\ve} ) \\
&+\!\!\!\!\!\!\!\!\!\! \sum_{\substack{m\in \N\\ (\tau-2)^m\log n\le1 } }\!\!\!\!\!\!\Pv_{\gamma}( Y_n \notin A_{n, \wit \gamma}(m)| Y^w \in A_{n, \wit \gamma}^{-\ve}(m) ) \Pv_{\gamma}( Y^w \in A_{n, \wit\gamma}^{-\ve}(m)). \ea\ee
Now pick $\ve$ small enough so that $\Pv_{\gamma}( Y^w \in \ind_{[0,1]} A_{n, \wit\gamma}\setminus A_{n, \wit\gamma}^{-\ve} )\le r_{\wit\gamma}/4$. This can be done uniformly in $n$ by the absolute continuity of the measure of $Y^w$ and the compactness of the set $[0,1]$: the set $A_{n, \wit \gamma}\ind_{[0,1]}$ can be generalised to the form
\[  B_{\wit\gamma,x}:= \bigcup_{m\in \N} (x(\tau-2)^{m+\wit\gamma}, x(\tau-2)^m), \quad x \in [0,1].\]  Note that the set $B_{\wit\gamma,x}$ only depends on the initial point $x$ and on the `width' given by the factor $(\tau-2)^{\wit \gamma}$. Now, we can define $B_{\wit\gamma,x}^{-\ve}(m)$ similarly as for $A_{n, \wit\gamma}(m)$ for all $m\in \N$, and pick
an $\ve=\ve(\wit\gamma)$ so small that the total length of the union of intervals $B_{\wit\gamma, x}(m)\setminus B_{\wit\gamma,x}^{-\ve}(m)$ is small enough even for the worst starting point $x$.  Then by the absolute continuity of the measure (see \cite[Theorem 3.5]{Foll84}) it follows that $\Pv_{\gamma}(Y^w \in B_{\wit\gamma,x}\setminus B_{\wit\gamma,x}^{-\ve}) ) \le r_{\wit\gamma}/4$ for all possible initial point $x$.

For the second term on the right hand side of \eqref{eq::error-n-d}, a simple calculation shows that $ \Pv_{\gamma}( Y_n \notin A_{n, \wit \gamma}(m)| Y^w \in A_{n, \wit \gamma}^{-\ve}(m) ) $ implies either $Y_n/Y^w > (\tau-2)^{-\ve}$ or $Y_n/Y^w\le (\tau-2)^\ve$, hence,
\[  \Pv_{\gamma}\Big( Y_n \notin A_{n, \wit \gamma}(m)\Big| Y^w \in A_{n, \wit \gamma}^{-\ve}(m) \Big) \le \Pv_{\gamma}\Big( \log Y_n - \log Y^w \notin ( -\wit \ve, \wit \ve ) \Big), \]
where $\wit\ve:= \ve | \log (\tau-2)|$. Now since $Y_n=Y_{\kappa(n)} \toas Y^w$, clearly, $\log Y_n \toas \log Y^w$ and hence the latter probability tends to zero for each fixed $\ve$ as $n\to \infty$. Pick $n_0=n_0(\ve)$ large enough so that this probability is at most $r_{\wit\gamma}/4$ for all $n\ge n_0(\ve)$. Then, combining the estimates for the two terms and the fact that $A_{n,\wit\gamma}^{-\ve} \subset A_{n, \wit\gamma}$, we get
\[ \Pv_{\gamma}(  Y^w \in A_{n, \wit \gamma}\ind_{[0,1]}, Y_n \notin A_{n, \wit \gamma}\ind_{[0,1]}) \le r_{\wit\gamma}/4 + \Pv_{\gamma}(Y^w \in A_{n, \wit \gamma}\ind_{[0,1]} ) r_{\wit\gamma}/4 \le r_{\wit\gamma}/2.\]

To finish the proof that \eqref{eq::yn-decompose} is at most $r_{\wit\gamma}/2$, we are left to give the uniform lower bound $r_{\wit\gamma}$ on  $\Pv(Y^w \in A_{n, \wit\gamma}\ind_{[0,1]})$.
Clearly,
\be\label{eq::indicator-Y-in-01} \ba \Pv_{\gamma}(Y^w \in A_{n, \wit\gamma}\ind_{[0,1]}) &=\!\!\!\!\!\!\!\!\! \sum_{\substack{m\in \N \\ (\tau-2)^m\log n< 1}} \!\!\!\!\!\!\!\!\exp\{- \Delta^w( (\tau-2)^{m+\wit\gamma} \log n)\} - \exp\{- \Delta^w( (\tau-2)^{m} \log n )\}\\
&\ge \exp\Big\{- \Delta^w( (\tau-2)^{1-c_n+\wit\gamma})\Big\} - \exp\Big\{- \Delta^w( (\tau-2)^{1-c_n+\wit\gamma})\Big\}
\ea  \ee
where we took the first $m$ that satisfies the criterion, i.e., $m=\lfloor\log\log n / | \log (\tau-2)|\rfloor + 1$, and introduced $c_n:=\{\log\log n / | \log (\tau-2)|\}\in [0,1)$.
Using that $(\tau-2)^{1-c_n} \in (\tau-2, 1)$ and $\Delta^w$ is strictly increasing (i.e., the difference is strictly positive), the right hand side is at least
 \be\label{eq::q-delta} r_{\wit\gamma}:=\min_{x \in [\tau-2,1]} \left\{ \exp\big\{\!-\! \Delta^w( x(\tau-2)^{\wit\gamma})\big\}
 -\exp \big\{ \!-\!\Delta^w( x)\big\} \right\}.
 \ee
 Note that this minimum exists and is strictly positive as well, since if it would be zero, that would imply the presence of an interval of size at least $(\tau-2)(1-(\tau-2)^{\wit\gamma})$ where $\Delta^w$ is a constant, which contradicts the fact that it is strictly increasing.\footnote{Without further assumption on the form of the generating function of $D^\star$ it is not possible to determine where this minimum is taken. E.g. it is not hard to determine using Proposition \ref{prop::no-gap} Part (2)  that the minimum is taken at $(\tau-2)$ if $f$ is convex and at $1$ if $f$ is concave in Proposition \ref{prop::no-gap}.}
  Further and more importantly, $r_{\wit\gamma}$ only depends on $\wit\gamma$, but not on $n$, hence, it provides a uniform lower bound on the probability that $v^\star$ is painted blue.
Combining everything, we get that
\[ \Pv_{\gamma}(v^\star \text { is blue}) = \Pv_{\gamma}( M_\kappa > Q^{\gamma})/2 = \Pv_{\gamma}( c_Q \in (0, \wit\gamma))/2 \ge r_{\wit\gamma}/4.\]
Setting $q_\gamma:= r_{\wit\gamma}/2$ finishes the proof for $\Pv(v^\star \text{\ is blue\,})$.
Further, since $\wit\gamma=1- \log \gamma/ |\log (\tau-2)|$, $\wit\gamma\searrow 0$ if $\gamma\nearrow 1/(\tau-2)$. So, the interval $[0, \wit\gamma]$ vanishes, and hence the lower bounds for \eqref{eq::not-only-red} vanish
by the continuity of measure of $Y^w$. As a result, $q_\gamma=r_{\wit\gamma}/4\to 0$.
 A lower bound on \eqref{eq::only-red} can be given in a similar manner, defining $q_\gamma^{\text{red}}$.

  To see that $q_\gamma^{\text{red}} \to 1$ uniformly in $n$ as $\gamma\to \tfrac{1}{\tau-2}$, we also need to show that \eqref{eq::not-only-red} tends to $0$ when $\wit\gamma\to 0$, uniformly in $n$.
  We fix a constant $L>0$ to be chosen later and write
\be\label{eq::upper-bound-on-and} \Pv_{\gamma}(Y_n \in A_{n, \wit\gamma}) \le  \Pv_{\gamma}(Y^w > L) + \Pv(Y^w\in A_{n, \wit\gamma}\ind_{[0,L]}) + \Pv_{\gamma}(Y^w \notin A_{n, \wit \gamma}, Y^w \le L,  Y_n \in A_{n, \wit \gamma}). \ee
 In the second term, the total Lebesque measure of the union intervals $A_{n, \wit \gamma}(m)\ind_{[0,L]}$ is at most
 \[ (1-(\tau-2)^{\wit\gamma}) L  \sum_{m\ge 1} (\tau-2)^m =(1-(\tau-2)^{\wit\gamma}) L \frac{\tau-2}{3-\tau}.  \]
Now, if $L^2(1-(\tau-2)^{\wit\gamma})<1$, then the last term is at most $\frac{\tau-2}{3-\tau}/L$. Hence, we can pick  $L=L(\wit\gamma)$ satisfying this, and use the continuity of the measure \cite[Theorem 3.5]{Foll84} of $Y^w$ on $[0,L]$ to get $\Pv( Y^w \in A_{n, \wit \gamma} \ind_{[0,L]})\le C/L.$ The first term is at most $\exp\{ - \Delta(L)\}$, hence it is sufficiently small when $L$ is large enough. This can be guaranteed even when $L^2(1-(\tau-2)^{\wit\gamma})<1$, when $\wit\gamma$ is sufficiently close to $0$.

The problem with the third term in \eqref{eq::upper-bound-on-and} is that $Y^w\notin A_{n, \wit \gamma}$ is a very likely event. But, we can apply a similar trick as in \eqref{eq::and-epsilon} and define for an $\ve>0$
\be\label{eq::and+epsilon}\ba A_{n, \wit\gamma}^{+\ve}(m):&=
 ((\tau-2)^{m+\wit\gamma+\ve} \log n, (\tau-2)^{m-\ve} \log n), \\
A_{n, \wit\gamma}^{+\ve}&:= \bigcup_{m\in \N}A_{n, \wit\gamma}^{+\ve}(m).\ea\ee
Clearly, $A_{n, \wit\gamma}^{+\ve}(m) \supset A_{n, \wit\gamma}(m)$.
We can then estimate the third term in \eqref{eq::upper-bound-on-and} as
\[\ba   \Pv_{\gamma}(Y^w \notin A_{n, \wit \gamma}, Y^w \le L,  Y_n \in A_{n, \wit \gamma}) &\le \Pv(Y^w \in A_{n, \wit\gamma}^{+\ve}\setminus A_{n, \wit\gamma}, Y\le L ) \\
 &\ \ +\Pv_{\gamma}(Y^w \notin A_{n, \wit \gamma}^{+\ve},Y_n \in A_{n, \wit \gamma}, Y^w\le L). \ea\]
Notice that the first term is small when $\ve$ is small enough by the continuity of the measure of $Y^w$ and the fact that $[0,L]$ is a bounded interval. The second term is  small if $n\ge n_0(\ve)$ since it implies that $Y_n/Y^w \notin ((\tau-2)^\ve, (\tau-2)^{-\ve})$, and $Y_n\toas Y^w$.
This shows that $\Pv_\gamma( w \text{ is red}) \to 1$ uniformly in $n$ as $\gamma \to 1/(\tau-2)$.

To finish up the proof of Lemma \ref{lem::q-delta}, we are left to show what error bound we have if we had applied Starting Rule 2 instead of Starting Rule 1. First of all note that the estimates on $\Pv(v^\star \text{ is blue})$ and
$\Pv(M_\kappa< Q^\delta)$ do not change. However, it might happen that even though $M_\kappa < Q^\delta$, there are some blue vertices in $G_\kappa$. This fact only might ruin $\Pv( v^\star \text{ is red, } D_{v^\star}<Q^\gamma)$.
However, note that the probability that any vertex is mis-colored, $p_e\to 0$ as $Q\to \infty$, which is true due to the choice of $Q$.
\end{proof}
As an immediate corollary of Lemma \ref{lem::q-delta}, we get the following.
\begin{corollary}\label{cor::red-first-moment} For both starting rules, the probability that the root is red is strictly positive for all $\gamma\in (1, 1/(\tau-2))$. Further, this probability tends to $1$ when $\gamma\nearrow 1/(\tau-2)$.
\end{corollary}
\begin{proof}
Under Starting Rule 1, $\{ w \text{ is red} \} = \{ D_{v^\star} < Q^\delta \}$, since on this event  \emph{all} vertices get color red in $\CG_\kappa$. As a result, $\Pv(w \text{ is red}) \ge q_\gamma$, and the result follows from Lemma \ref{lem::q-delta}.

It is left to investigate the effect of Starting Rule 2. Recall that in the interval $[Q^\gamma, Q^{1/(\tau-2)}]$ all vertices are painted red and blue with equal probability, and in the interval $[Q, Q^\gamma)$ they are painted red with probability $1-p_e$ and blue with $p_e$, where $p_e\to 0$ as $n \to \infty$. When $\gamma=1$, then $\Pv(w \text{ is red}) =1/2$ by symmetry. Then, a simple coupling argument and monotonicity of the Flow Rule implies that the event $\{w \text{ is red}\}$ is an \emph{increasing event}\footnote{If we write $\omega_i=0,1$ if the $i$th individual in $\CG_\kappa$ is blue or red, respectively, and two colorings $\underline\omega\le \underline\omega'$ iff $\omega_i\le \omega_i'$ for all  $i\in \CG_\kappa$, then $\Pv(w \text{ is red } | \CG_\kappa, \underline\omega ) \le  \Pv(w \text{ is red } |\CG_\kappa, \underline\omega')$.} in the number of red vertices in $\CG_\kappa$.
This implies that $\Pv(w \text{ is red}) \ge 1/2$ for all $\gamma \in [1,1/(\tau-2)).$

Thus, the statement that $\Pv(w \text{ is red})$ strictly positive also holds under Starting Rule 2. It is left to investigate what happens when $\gamma\nearrow 1/(\tau-2)$. Let us denote the number of differently colored vertices in the two colourings by $\mathrm{Blue}_\gamma$, that is, the number of blue vertices in $\CG_\kappa$ with degree in $ [Q, Q^\gamma)$.
We would like to show that
\[ \Pv(\mathrm{Blue}_\gamma\ge 1) \to 0 \]
as $\gamma \nearrow (\tau-2)^{-1}$.
Note that $\mathrm{Blue}_\gamma\  {\buildrel d \over \le }\  \mathrm{Bin}(\#\{x\in G_\kappa, D_x\ge Q\} , p_e)$, hence, first we aim to give an estimate on $|\CG_\kappa|=Z_\kappa$, then on the proportion of vertices with degree $D_x\ge Q^\gamma$.
Using the stochastic domination argument in the proof of Claim \ref{cl::max-sum-relation}, we get
\[ X^{\sss{(1)}}\ {\buildrel {d} \over \le }\ \frac{Z_{i+1}}{Z_i^{1/(\tau-2)}} \ {\buildrel {d} \over \le }\  b_2+X^{\sss{(2)}},  \]
for some positive random variables $X^{\sss{(1)}}, X^{\sss{(2)}}$ and shift $b_2>0$. This fact, combined with Claim \ref{cl::max-sum-relation} and the \emph{proof} of Claim \ref{cl::stoch-dom-alpha-stable} gives that for some logarithmic correction term, if $Q$  is large enough, whp
\[ Z_\kappa \le (\log Q) (M_\kappa)^{(\tau-2)} = Q^{(\tau-2)^{c_Q}} \log Q. \]
Recall that in the $\kappa$th generation of the stopped BP, all vertices have i.i.d.\  degrees with distribution $D^\star$ conditioned on being $\le M_\kappa$.
Hence, the probability that the degree of a vertex $v\neq v^\star$ falls in the interval $[Q, Q^{\gamma})$ is at most
\[ \Pv(D^\star \ge Q | D^\star \le M_\kappa, M_\kappa) \le \frac{1-F^\star(Q)}{F^\star(M_\kappa)}\le 2 C_1^\star Q^{2-\tau}.\]
According to Starting Rule 2, each vertex in this interval gets color blue with probability $\simp Q^{1-\gamma}$ independently of each other,
hence, $B_\gamma$ can be stochastically dominated by a binomial random variable, and so
\[ \Ev[\mathrm{Blue}_\gamma] \le C Q^{(\tau-2)^{c_Q} + 2-\tau+ 1-\gamma} \log Q.\]
Writing $\gamma:=1/(\tau-2) - x$, and recalling that $c_Q\in [0,1)$,
the exponent of $Q$ is at most $-(3-\tau)^2 +x$, with equality if $c_Q=0$.
Hence, if $\gamma\nearrow 1/(\tau-2)$, $x\searrow 0$ and so we can pick a small enough $x$ so that the exponent is negative.
By Markov's inequality, $\Pv(\mathrm{Blue}_\gamma\ge 1) \to 0$ in this  case, and so the coloring in Starting Rule 2 is whp the same as the coloring in Starting Rule 1. For the latter, $\Pv_\gamma(w \text{ is red\,}) \to 1$ has been already shown in Lemma \ref{lem::q-delta}.
\end{proof}

The next lemma is the crucial ingredient for the proof of coexistence, and it shows that the probability that the blue color reaches the root is uniformly positive in $n$. It also implies Proposition \ref{prop::BP-color}.
\begin{lemma}\label{lem::blue-first-moment} Let $\gamma\in (1, 1/(\tau-2))$.
For both starting rules, the probability that the root $w$ of the branching process is painted blue is at least
\[ \Pv(w \text{ is blue}) \ge \frac12 \e^{-\gamma} q_{\gamma(1+\ve)}, \]
where $\ve>0$ is  such that $\gamma(1+\ve)<1/(\tau-2)$ holds, and  $q_{\gamma}$ is from Lemma \ref{lem::q-delta}.
\end{lemma}
\begin{proof}[Proof of Proposition \ref{prop::BP-color}].
 The lower bound in statement of the proposition directly follows from Lemma \ref{lem::blue-first-moment} with $c(\gamma):=\frac12 \e^{-\gamma} q_{\gamma(1+\ve)}$. Further, recall that $q_\gamma\searrow 0$ as $\gamma \nearrow 1/(\tau-2)$ holds, see the statement of Lemma \ref{lem::q-delta}. The upper bound in the proposition follows from Lemma \ref{lem::q-delta} and Corollary \ref{cor::red-first-moment}, with $C(\gamma):=1-q_\gamma^{\text{red}}$.
\end{proof}
\begin{proof}[Proof of Lemma \ref{lem::blue-first-moment}]
Throughout the proof, we analyse the worst starting scenario for blue, that is, Starting Rule 1. Recall that the maximal degree vertex in generation $\CG_\kappa$ is denoted by $v^\star$, and $v_{\kappa-1}$ is the child of the root $w$ so that the subtree of $v_{\kappa-1}$ contains $v^\star$.

 Lemma \ref{lem::q-delta} implies that there is a positive chance that there are blue vertices in generation $\CG_\kappa$ of the BP.  Note that by the same argument, for any small $\ve$ such that $\gamma(1+\ve) < 1/(\tau-2)$, we also have $\Pv(M_\kappa> Q^{ \gamma(1+\ve)})\ge r_{\wit\gamma-\log(1+\ve)/|\log(\tau-2)|}/4 >0$, with $\wit\gamma$ defined in \eqref{eq::crit}.
Hence, let us fix an $\ve$ for which $\wit\gamma-\log(1+\ve)/|\log(\tau-2)| \ge \wit\gamma/2$, and we aim to show that \emph{conditioned on}  $M_\kappa> Q^{\gamma(1+\ve)}$, the probability that a sibling of $v_{\kappa-1}$ is blue is strictly positive, i.e., $r_\kappa>0$.
Let us write $\CT^{(i)}$ for the subtree of the $i$th sister of $v_{\kappa-1}$, and let
\[ \CT_{-r}:= \bigcup_{i=1}^{D_w-1} \CT^{(i)}\]
The reason to restrict our attention to the subtrees of the siblings of $v_{\kappa-1}$ is that the degrees in these subtrees are conditionally independent of each other and also of $D_w$.\footnote{On the other hand, $D_w$ is not independent of $\kappa$ and $M_\kappa$, and also the degrees of vertices on the ray, starting from $v_{\kappa-i}$ for $1\le i\le \kappa-1$ are not conditionally independent: the conditioning that $v_{\kappa-i}$ leads to a maximal degree vertex influences the degree of $v_{\kappa-i}$.} By the definition of $\kappa$, conditioned on $M_\kappa$ and the position of the ray, the  degrees of vertices in every generation in $\CT_{-r}$ are \emph{independent}, and the degrees in $\CG_\kappa \cap \CT_{-r}$ have distribution $D^\star| D^\star\le M_\kappa$, while in every earlier generations the degrees have distribution $D^\star| D^\star<Q$.

  From now on, we work under the assumption that $M_\kappa\ge Q^{\gamma(1+\ve)}$,  (otherwise, we get that $s_0$ below is zero). Then, for any vertex in $\CG_\kappa\cap \CT_{-r}$, using \eqref{eq::size-biased2},
\be\label{eq::s0} \ba \Pv_{\gamma}(x \in \CG_{\kappa} \cap \CT_{-r} \text{ is blue}\, | M_\kappa ) &= \Pv_{\gamma}(D^\star> Q^{\gamma}| D^\star \le M_\kappa, M_\kappa)/2\\
&\ge \frac{c_1^\star}{2} Q^{-\gamma(\tau-2)}\frac{1- Q^{-\gamma\ve(\tau-2)}}{1-c_1^\star Q^{-\gamma(1+\ve)(\tau-2)}} \\
&\ge \frac{c_1^\star}{4} Q^{-\gamma(\tau-2)}=:s_0,\ea \ee
where we have used that the last ratio is at least $1/2$ if $Q$ is large enough (which clearly holds since $Q\ge \exp\{ (\log n)^{1-\nu-c}  \}$ for some small $c>0$). Recall that the Flow rule ensures that a vertex in $\CG_{\kappa-1}$ gets color blue with probability at least $1/2$ if it has at least one blue child. Also,
recall that the degree of any vertex in $\CG_{\kappa-1}\cup \CT_{-r}$ has distribution $D^\star| D^\star < Q$.
Then, the probability that any vertex in $\CG_{\kappa-1}$ has a blue child is at least
\be\label{eq::before-s1} \Pv_{\gamma}(x\in \CG_{\kappa-1}  \cap \CT_{-r} \text{ is blue} ) \ge \frac12 \Ev_{\gamma}[ 1- (1-s_0)^{D^\star} | D^\star<Q]=\frac12(1-\wih h(1-s_0)),\ee
where $\wih h$ is the generating function of  $D^\star | D^\star<Q$. We calculate using \eqref{eq::gen-func} from Proposition \ref{prop::no-gap} that
\[  \wih h(s)=\frac{\sum_{k=1}^\infty \Pv(D^\star=k)s^k - \sum_{k=Q}^\infty \Pv(D^\star=k)s^k}{\Pv(D^\star< Q)} \
\le \frac{1-(1-s)^{\tau-2} L^\star(\tfrac{1}{1-s}) }{1-C_1^\star  Q^{-(\tau-2)}  }.\]
Hence, \eqref{eq::before-s1} becomes
\be\label{eq::at-s1} \Pv_{\gamma}(x\in \CG_{\kappa-1}  \cap \CT_{-r} \text{ is blue} )\ge \frac{1}{2} \frac{s_0^{\tau-2}L^\star(\frac{1}{s_0})- C_1^\star  Q^{-(\tau-2)}}{1-C_1^\star  Q^{-(\tau-2)}} \ge \frac{c_2^\star}{4} s_0^{\tau-2}=:s_1,
\ee
 where the last inequality is true since $L^\star(\cdot)>c_2^\star$ is a strictly positive bounded function by Proposition \ref{prop::no-gap}, further, $\gamma\in(1, 1/(\tau-2))$ implies that $s_0^{\tau-2}\gg  Q^{-(\tau-2)}$ (see \eqref{eq::s0}).
For any $i\ge 1$, any vertex in generation $\CG_{\kappa-(i+1)}$ is blue with at least a probability $1/2$ if it has at least one blue child in generation $\CG_{\kappa-i}$. Hence, writing
\[  \Pv_{\gamma}(x\in \CG_{\kappa-i}  \cap \CT_{-r} \text{ is blue} )\ge s_i,
 \] we can repeat \eqref{eq::before-s1} and \eqref{eq::at-s1} using $s_i$ instead of $s_0$. Note that the condition $s_i^{\tau-2}\gg  Q^{-(\tau-2)}$ is also satisfied for all $i$. This yields that
that $s_i$ satisfies the recursion $s_{i+1}=\frac{c_2^\star}{4} s_{i}^{\tau-2}$, hence
 \be\label{eq::s-kappa-1}s_{\kappa-1}= s_0^{(\tau-2)^{\kappa-1}} \left(\frac{c_2^\star}{4}\right)^{(1-(\tau-2)^{\kappa-1})/(3-\tau)}.\ee
By definition, $s_{\kappa-1}=\Pv( x \in \CG_1 \cap \CT_{-r} \text{ is blue})$. Also, since every vertex has degree at least $2$, $D_w-1\ge1$ and hence the root has at least one child that is not on the ray. As a result,
\[ \Pv(w \text{ is blue}) \ge s_{\kappa-1} /2.\]
Using the definition of $\kappa$ in \eqref{eq::kappa}, we calculate
\[ (\tau-2)^{\kappa-1}=\frac{Y_\kappa^w}{\log Q} (\tau-2)^{-c_Q-1}.\]
Combining this, \eqref{eq::s-kappa-1} and the value of $s_0$ from \eqref{eq::s0}, conditioned on the value of $\kappa$,
\be\label{eq::s-kappa} \ba s_{\kappa-1}&\ge \left(\frac{c_2^\star}{4}\right)^{\frac{1}{3-\tau}} \left(\frac{c_1^\star 4^{\frac{\tau-2}{3-\tau}} }{(c_2^\star)^{\frac{1}{3-\tau}}}\right)^{ (\tau-2)^{\kappa-1}} \exp\{-(\log Q) \gamma(\tau-2) \frac{Y_\kappa^w}{\log Q} (\tau-2)^{-c_Q-1}\} \\
 &\ge  C \exp\{ -\gamma(\tau-2)^{-c_Q} Y_\kappa^w\} \ge C \exp\{ -\gamma Y_\kappa^w\},  \ea\ee
where we used that $c_Q \in (0,1).$

Note that this expression is conditioned on the value of $Y_\kappa^w$, hence, it is left to evaluate its expectation (conditioned on $\gamma$), but recall, that we also have assumed that $M_\kappa \ge Q^{\gamma(1+\ve)}$, which event is the same as  $c_Q \in (0, \wit\gamma- \log(1+\ve)/|\log (\tau-2)| )$ by \eqref{eq::crit}.
Hence, combining \eqref{eq::s-kappa} with this conditioning, we get the probability that the root  $w$ is blue can be bounded from below by

\[  \Ev_{\gamma}\left[  \exp\{ -\gamma Y_\kappa^w\} \ind_{ \left\{ \left\{\frac{\log\log n - \log Y_\kappa^w}{|\log (\tau-2)|}\right\} \in (0, \wit\gamma- \frac{\log(1+\ve)}{|\log (\tau-2)|}) \right\}}\right]\]

 We can get a lower bound on this expression by restricting $Y_\kappa^w$ to $[0,1]$, in which case, the factor before the indicator is at least $\exp\{ - \gamma \}$, while the expectation of the indicator is treated in the \emph{proof of} Lemma \ref{lem::q-delta}, and is at least $r_{\wit\gamma - \frac{\log(1+\ve)}{|\log (\tau-2)|} }$ (where $r_{\wit\gamma}$ is defined in \eqref{eq::q-delta}, see also \eqref{eq::indicator-Y-in-01}).
 Combining these we arrive at
 \[ \Pv_{\gamma}(w \text{ is blue }) \ge \frac12 \e^{-\gamma} r_{\wit\gamma - \frac{\log(1+\ve)}{|\log (\tau-2)|} }, \]
 and since $\ve$ satisfies that $\wit\gamma-\frac{\log(1+\ve)}{|\log (\tau-2)|}\ge 0$, this finishes the proof.
 \end{proof}
\begin{proof}[Proof of Theorem \ref{thm::main2}]
In the proof of this theorem, we work conditionally on $q=\Ybn/\Yrn \in (\tau-2, 1)$. Of course, if $\Ybn/\Yrn \in (1, 1/(\tau-2))$, then the same statements are true, with the role of red and blue exchanged. Recall that $\Ev_\gamma[\cdot] = \Ev_[ \cdot | \gamma, Q, \Yrn, \Ybn]$, where $\gamma, Q$ were defined in \eqref{def::gamma} and just after \eqref{def::core-n}, respectively.

We aim to show that there exist constants $c(q), c(q)^{\text{red}}>0$, so that $\CB_\infty/n\ge c(q)$ and $\CR_\infty/n\ge c(q)^{\text{red}}$ hold whp. The upper bound in the statement of the theorem is then obvious with $C(q):=1- c(q)^{\text{red}}$ since $\CB_\infty + \CR_\infty =n$.

We only show this statement for $\CB_\infty/n$, but $\CR_\infty/n\ge q_{\gamma}^{\text{red}}/2:=c(q)^{\text{red}}$ can be treated analogously, using Lemma \ref{lem::q-delta} and Corollary \ref{cor::red-first-moment} by setting  $\gamma$ in these to be equal to the value in \eqref{def::gamma}.
Let us write $2p_{\gamma}:= \exp\{ - \gamma \} r_{\wit \gamma- \frac{\log(1+\ve)}{|\log (\tau-2)|}}/2$.
Lemma \ref{lem::blue-first-moment} ensures that, with $w$ being a uniformly chosen vertex,
\[ \Ev_{\gamma}\left[\CB_\infty/n\right] = \sum_{v\in [n]} \Ev_{\gamma}[ \ind_{\{ v \text{ is blue}\}}]/n = \Pv_{\gamma}(w \text{ is blue}  ) \ge 2 p_{\gamma}. \]
Then we use Chebishev's inequality to get
\[ \Pv_{\gamma}( \CB_\infty/ n \le p_{\gamma} ) \le \Pv_{\gamma}( |\CB_\infty/ n  -  \Ev_{\gamma}\left[\CB_\infty/n\right]| \ge p_{\gamma} ) \le \frac{\Var_{\gamma}[ \CB_\infty/n  ]}{ p_{\gamma}^2 }.  \]
To prove the statement of the theorem, we need to show that the rhs tends to zero with $n$.
We write
\[\ba  \Var_{\gamma}[\CB_\infty /n] &\le  \frac{1}{n^2} \sum_{w \in [n]} \Pv_{\gamma}(w \text{ is blue}) \\&+  \frac{1}{n^2} \sum_{w\neq z \in [n]} \Pv_{\gamma}(w,z \text{ is blue}) - \Pv_{\gamma}(w \text{ is blue})\Pv_{\gamma}(z \text{ is blue}). \ea\]
Clearly, the first sum is at most $n$, hence, the first term is at most $1/n$.
For the second term we need to show that for a uniformly chosen pair of vertices $w,z$
\[ \Pv_{\gamma}(w,z \text{ is blue}) \to \Pv_{\gamma}(w \text{ is blue})\Pv_{\gamma}(z \text{ is blue}).\]
This statement is a consequence of the coupling to independent branching processes. Indeed, if we have two uniformly chosen red and blue source vertices $u,v$, and two other uniformly chosen pair of vertices $w,z$, then one can generalise
\cite[Lemma 2.2]{BarHofKom14}: the local neighbourhoods of these four vertices can be coupled to four independent branching processes, where in each of them the degree of the root is an i.i.d copy of $D$, and all other forward degrees are distributed as $D^\star$. This coupling can be achieved by extending \cite[Lemma 2.2]{BarHofKom14} to four vertices in the following way: first we finish coupling the degrees of the local neighbourhoods of the source vertices $u,v$ to the red and blue BP up to total size $n^{\vr^{\sss{(r)}}}$, determining $\Ybn, \Yrn, Q, \gamma$. Then, we imaginarily stop the spreading of these two colors at this point, and we couple the degrees in the exploration process of the local neighbourhood of $w,z$ to a collection of independent random variables, forming another two independent BPs, till their maximum degree reaches $Q$, and then we can continue with everything else.
If we pick $\vr$ in Section \ref{sc::climbup} sufficiently small, the total number of vertices explored in the four processes together is still at most of order of magnitude $n^{\vr/\tau-2}< n^{\vr'}$ (where $\vr'$ is from \cite[Lemma 2.2]{BarHofKom14}), hence, the coupling of the forward degrees is still valid by \cite[Lemma 2.2]{BarHofKom14}.
Clearly, conditioned on the value $Q$ and $\gamma$, (determined by the red and blue BPs), the probability that the roots of the two independent BP-s are blue under this independent coupling is $\Pv_{\gamma}(w\text{ is blue}\,)\Pv_{\gamma}(z\text{ is blue}\,)$.
Hence, the probability that $\Pv_{\gamma}(w,z \text{\ is blue\,}) \neq \Pv_{\gamma}(w\text{ is blue\,})\Pv_{\gamma}(z\text{ is blue\,})$ is exactly the probability that the coupling fails, which tends to zero as  $n\to \infty$ (an estimate on order of magnitude of the coupling error can be found in \cite[Appendix A.2]{BHH10}, where it is shown to be $O(n^{-\ve})$ for some small but positive $\ve$.)

This finishes the proof of Theorem \ref{thm::main2} with $c(q):= p_\gamma$ and $C(q):=1- q_\gamma^{\text{red}}/2.$
Note that $ p_\gamma \to 0, q_\gamma^{\text{red}} \to 1$ as $\gamma \nearrow 1/(\tau-2)$ exactly implies $c(q), C(q) \to 0$ as $q\searrow (\tau-2)$, see Remark \ref{rem::q-gamma-relation}.
\end{proof}

\section{Number of maximum degree vertices}\label{sc::mbn}
In what follows, we aim for the proof of Theorem \ref{thm::main} if $\Ybn/\Yrn<\tau-2$, but we still need plenty of preparation for that.
Recall from Section \ref{sc::slopedown}, page 19 that Case (1) stands for $T_b-T_r\ge 2$, while Case (3)(a) means $T_b-T_r=1$ with $\Ybn/\Yrn <\tau-2$.
We have seen in the proof of Theorem \ref{thm::maxdegree} (see page 30-33), that $D^{\max}_n(\infty)$, the degree of the highest degree vertex that blue can paint, can be expressed by distinguishing four cases $O_>, E_>, O_<, E_<$ (see page 31) representing $T_b-(T_r+1)$ being odd or even and $\tau-1>(\tau-2)^{\bnr} + (\tau-2)^{\bnb}$ holds or not. Recall that $h_n(\Yrn, \Ybn)$ captures the oscillating part that depends on these 4 cases for the normalising oscillating random variable in Theorem \ref{thm::maxdegree}.

In this section we investigate \emph{how many} maximum degree vertices are reached by blue. Here, by `maximum degree' vertex we mean any vertex $w$ that satisfies  $\log (\deg_w) = \log (D^{\max}_n(\infty)) (1+ o_{\Pv}(1))$.  We show that  in Cases $E_>,O_<$, the number of these vertices is in fact so large that it corresponds to an additional factor for the total number of half-edges in maximum degree vertices of blue.

More precisely, let us denote the set of outgoing half edges from the maximal degree vertices by $\CMBN$, and its size by $\MBN$.
 \begin{lemma}\label{lem::verticeswithmaxdegree}
The number of outgoing half-edges from the set of maximal degree vertices, i.e. the sum of the forward degrees of vertices for which \eqref{eq::max-deg-2} holds, satisfies
\[ \frac{\log\MBN}{(\log n)\!\cdot\! (\tau-1)^{-1} h_n^{\text{half-edge}}( \Yrn, \Ybn)  } \toindis \sqrt{\frac{Y_b}{Y_r}},  \]
where  $h_n^{\text{half-edge}}( \Yrn, \Ybn)$ stochastically dominates  $h_n(\Yrn, \Ybn)$, and is a bounded random variable given below in formula \eqref{def::gfunction}.
\end{lemma}

\begin{proof}
The proof is analogous to that of \cite[Lemma 6.3]{BarHofKom14}, but some details differ. Hence, we work out the proof here.
To start with, recall from the proof of Theorem \ref{thm::maxdegree} that in Cases $E_<, O_>$, blue finishes its last jump at a certain layer $\Gamma_{i}^{\sss{(b)}}$. See the argument after \eqref{eq::dtc2} for this, and the four cases on page 31.
Thus, in Case $E_<, O_>$ the statement is a direct consequence of Lemma \ref{lem::numberofverticesinGamma}, since blue is stuck with its maximal degree at layer $\Gamma_{T_r+\lfloor t_c \rfloor-t(n^{\vr})+1}^{\sss{(b)}}$, and hence $\MBN= A_{T_r+\lfloor t_c \rfloor-t(n^{\vr})+1}^{\sss{(b)}} D_n^{\max}(\infty)$. Let us write from now on in this proof $i_{\max}:=T_r+\lfloor t_c \rfloor-t(n^{\vr})+\ind_{E_< \cup O_>}$. Recall that $i_{\star \sss{(b)}}$ denotes the total number of layers blue can go through if red would not be present (see \eqref{eq::b-definitions}), so we clearly have $i_{\max}\le i_{\star\sss{(b)}}$.
 Then, for Cases $E_<, O_>$, we get the bound
\[\log \MBN \le \log D_n^{\max}(\infty)+ \log  A_{i_{\max}}^{\sss{(b)}},\] and since $i_{\max} < i_{\star \sss{(b)}}$ is whp a bounded random variable, the last term  disappears when we divide by $\log n$.

We are left with handling the cases were the last jump of blue is not a full layer, i.e.,\ Cases $E_>, O_<$. In these cases, after reaching layer $\Gamma_{i_{\max}}^{\sss{(b)}}$, blue still jumps up, but not a full layer:
 due to the presence of red color the forward degrees are truncated at $\wit u_{\lfloor t_c \rfloor}^{\sss{(r)}}$.

First, we apply Lemma \ref{lem::numberofverticesinGamma} to see that $\log A_{i_{\max}}^{\sss{(b)}}$ in the last `full' layer $\Gamma_{i_{\max}}^{\sss{(b)}}$ is small.
Let us recall the notation $u_{i_{\max}}^{\sss{(b)}}=D_n^{\max}(T_r+\lfloor t_c \rfloor)$, and the fact that the `size' of the last jump of blue in the exponent of $(\tau-2)$ is $2\{t_c\}<1$  at time $T_r+\lfloor t_c \rfloor +1$ for Case $E_<, O_>$ (see \eqref{eq::dmax-tc} and \eqref{eq::wur-tc} to get \eqref{eq::dtc2}). Let us introduce as the extra factor of the $\log$(degrees)/$\log n$ reached at time $T_r+\lfloor t_c \rfloor+1$ by blue
\be\label{def::xi} \xi:=(\tau-2)^{-2\{t_c\} }, \ee  and then we introduce a new layer
\[ \Gamma^\diamond:=\left\{ v\in \CMD:   d_v \ge  \frac{(u^{\sss{(b)}}_{i_{\max}})^\xi}{(C\log n)^{1/(\tau-2)} } \right\},\]
and we denote the number of half-edges in this set by $\CE_{\xi}$.

By Lemma \ref{lem::badpaths}, whp blue does not reach higher degree vertices than $\widehat u_{i_{\max}}^{\sss{(b)}}$ at time $T_r+\lfloor t_c \rfloor$. So,
 conditioned that there are $A_{i_{\max}}^{\sss{(b)}}$ many vertices in layer $\Gamma_{i_{\max}}^{\sss{(b)}}$, the number of vertices in $\Gamma^\diamond$ to which blue is connected to is dominated by
 \be\label{eq::bindom3}\mathcal B \cap \Gamma^\diamond \  {\buildrel {d}\over{ \le }}\     {\sf Bin} \left( A_{i_{\max}}^{\sss{(b)}} \widehat u_{i_{\max}}^{\sss{(b)}}, \frac{\CE_{\xi}}{\CL_n (1+o(1))}  \right), \ee
where we recall that $\CL_n$ is the total number of half-edges in $\CMD$. We can bound
\[ \CE_\xi\le n (u^{\sss{(b)}}_{i_{\max}})^{-\xi(\tau-2)} C\log n\]  by using Claim \ref{claim::Sbound}.
 Thus, conditioned on $A_{i_{\max}}^{\sss{(b)}}$, on the event $\CL_n/n\in \{\Ev[D]/2, \Ev[D]\}$ the expected value of the binomial variable in \eqref{eq::bindom3} is whp bounded from above by
 \[  \frac{2 C_1^2}{c_1} (C \log n) A_{i_{\max}}^{\sss{(b)}} \widehat u_{i_{\max}}^{\sss{(b)}}  (u_{i_{\max}}^{\sss{(b)}})^{-\xi(\tau-2)}
  = \frac{2 C_1^2}{c_1} (C \log n) A_{i_{\max}}^{\sss{(b)}}\frac{\widehat u_{i_{\max}}^{\sss{(b)}}}{u_{i_{\max}}^{\sss{(b)}}} (u_{i_{\max}}^{\sss{(b)}})^{1-\xi(\tau-2)}.  \]
This gives an upper bound on the number of \emph{vertices} with degree at least $(u^{\sss{(b)}}_{i_{\max}})^\xi/ (C \log n)^{1/(\tau-2)}$, thus the total number of \emph{half-edges} going out from maximal degree vertices can be bounded from above by
 \be\label{eq::addgamma} \frac{2 C_1^2 (C \log n)^{\frac{\tau-3}{\tau-2}}}{c_1} A_{i_{\max}}^{\sss{(b)}}\frac{\widehat u_{i_{\max}}^{\sss{(b)}}}{u_{i_{\max}}^{\sss{(b)}}} (u_{i_{\max}}^{\sss{(b)}})^{1+\xi(3-\tau)}.\ee
  Since $i_{\max}\le i_{\star\sss{(b)}}$ is bounded, we can use \eqref{eq::aifinal}, \eqref{eq::uibar} with \eqref{def::ui} to see that
  \[ A_{i_{\max}}^{\sss{(b)}}\frac{\widehat u_{i_{\max}}^{\sss{(b)}}}{u_{i_{\max}}^{\sss{(b)}}}\le (A_{i_{\max}}^{\sss{(b)}})^2 \le (A_{i_{\star \sss{(b)}}}^{\sss{(b)}})^2\] is some bounded power of $C \log n$, and thus it disappears when taking logarithm and dividing by $\log n$.
   Hence, we get that in Cases $E_<, O_>$ with $\xi=(\tau-2)^{-2\{t_c\}}$, whp
   \be\label{eq::xi-estimate-1}\log \MBN\le(1+\xi(3-\tau))\log (u_{i_{\max}}^{\sss{(b)}} ).\ee
The \emph{lower bound} on the number of blue vertices in $\Gamma^\diamond$ is easier: we can establish a similar binomial domination argument (now from below), using  that there is at least one blue vertex in $\Gamma_{i_{\max}}^{\sss{(b)}}$ with degree at least $u_i^{\sss{(b)}}$, and that the number of half edges $\CE_\xi\  {\buildrel d \over \ge y_n} \mathrm{Bin}(n, c_1 y_n^{1-\tau})$, with $y_n= (u^{\sss{(b)}}_{i_{\max}})^\xi/(C\log n)^{1/(\tau-2)} $, and using the concentration of the binomial random variables. The estimate gives the same order of magnitude as the rhs of \eqref{eq::xi-estimate-1}, with small error probabilities: the details are left for the reader.

 Recall the definition of $\xi$ from \eqref{def::xi} and that in Cases $E_<, O_>$
  \[ \log(D_n^{\max}(\infty)) = \xi \log u_{i_{\max}}^{\sss{(b)}} (1+o_{\Pv}(1)) \]
  holds. Comparing this with \eqref{eq::xi-estimate-1} we arrive at
  \be\label{eq::logmbn-to-logdmax}  \log \MBN = (\xi^{-1} + (3-\tau) )  \log D_{\max}^{\sss{(b,n)}}(\infty)  (1+o_{\Pv}(1)).\ee
  To get the final expression for total number of half-edges at the last up-jump, we need to multiply the function $h_n(\Yrn, \Ybn)$ in \eqref{eq::h1} by $(\xi^{-1} + (3-\tau))$ in Cases $O_<, E_>$, using that $2\{t_c\}=\bnb-\dnr$ for Case $E_>$, while $2\{t_c\}=1+\bnb-\dnr$ for Case $O_<$. As a result, we obtain
 \[ \frac{ \log \MBN} {(\tau-1)^{-1}\log n} = \sqrt{\frac{\Ybn}{\Yrn}} h^{\text{half-edge}}_n( \Yrn, \Ybn) (1+o_{\Pv}(1)),\]
 with
 \be\label{def::gfunction} \ba &h^{\text{half-edge}}_n( \Yrn, \Ybn):=\ind_{ E_< \cup O_>} (\tau-2)^{(\bnr+\bnb-1 - \ind_{O_>})/2} +\\
&+\ind_{ E_> \cup O_<} (\tau-2)^{(\bnr-\bnb-1 - \ind_{O_<})/2} ((\tau-1)-(\tau-2)^{\bnr})(3-\tau)\\
&+\ind_{ E_> \cup O_<} (\tau-2)^{(\bnr+\bnb- \ind_{E_>})/2}. \ea\ee Dividing by $h^{\text{half-edge}}_n( \Yrn, \Ybn)$ and using that $(\Yrn, \Ybn)\toindis (Y_r, Y_b)$ finishes the proof of Lemma \ref{lem::verticeswithmaxdegree}.
\end{proof}

Before moving on to the next section, let us introduce the time when the maximal degree is reached, which is nothing else but the time of the last possible up-jump of blue, i.e., for all four cases $O_<, O_>, E_<, E_>$ it is
\be\label{def::t_b}\ba t_b&:=T_r+\lfloor t_c \rfloor+1.
\ea \ee
\section{Path counting methods for blue}\label{sc::path-counting} We have seen in Section \ref{sc::mbn} that if $\Ybn/\Yrn<(\tau-2)$, blue has $\MBN$ many half-edges of highest degree at time $t_b$. At this time, only $o(n)$ vertices are reached by red and blue together\footnote{This statement needs verification, but it follows simply from the fact that typical distances in $\CMD$ are $2 \log \log n/|\log (\tau-2)|+O_{\Pv}(1)$ while $t_b= \log \log n / |\log (\tau-2)| + O_{\Pv}(1)$, hence most vertices have  not been reached by any color at time $t_b$ yet.}
-- most of the vertices are still not colored. Thus, it still remains to determine how many vertices blue can reach \emph{after} time $t_b$. We do this via giving a matching upper and a lower bound on how many vertices blue occupies in this last phase. This part is a direct application of the methods described in \cite[Section 7]{BarHofKom14}, thus, we only describe the idea and check that each condition in the lemmas there is satisfied.

For the upper bound, the idea is that we count the close neighborhood of the half-edges that are just occupied at time $t_b$. Since the red avalanche continues to be in its avalanche phase and occupies all vertices at lower and lower degrees as time passes, the spreading of blue is more and more restricted, so this local neighborhood is quite small. We call this the \emph{optional cluster} of blue. We give a concentration result on its size. (That is, a concentrated upper bound on what blue can get.)

For the lower bound, we estimate how many vertices red might `bite out' of this optional cluster. This can happen since even a constant degree vertex might be close to both colors. We show that this intersection of the clusters is negligible compared to the size of the optional cluster.

We start describing the upper bound -- the optional cluster of blue -- in more detail. At time $t_b$, the half-edges in the set $\CMBN$ start their own \emph{exploration clusters}, i.e., an exploration process from the half-edge to not-yet
 occupied vertices. At time $t_b+ j$, we color every vertex $v$ blue, whose distance is exactly $j$ from some half-edge $h$ in $\CMBN$, and the degrees of vertices on the path from $h$ to $v$ are less than what red occupied by that time.  That is, the degree of the $j+1$st vertex on the path must be less than $\wit u_{ t_b-T_r+ j-1}^{\sss{(r)}}$.
 We do this via estimating the number of paths with degree restrictions from $\CMBN$ and call this the \emph{optional cluster of blue}, and denote the set by $\opt_{\max}$ and its size by $\ops_{\max}$. Corollary \ref{cor::chebisev} below determines its asymptotic behavior.

On the other hand, not just half edges in $\CMBN$ can gain extra blue vertices: from half edges in $\CA_{i_{\max}-z}^{\sss{(b)}},\  z=0,1,2\dots$ the explorations start earlier (at time $t_b-z$) towards small degree vertices. Let us denote the vertices reached via half-edges from layer $\CA_{i_{\max}-z}^{\sss{(b)}}\setminus \CA_{i_{\max}-z+1}^{\sss{(b)}}$ by $\opt_{-z}$ and its size by $\ops_{-z},\, z\ge 0$. At time $t_b- z+ j$, we color every vertex $v$ blue, whose distance is exactly $j$ from a half-edge $h$ in $\CA_{i_{\max}-z}^{\sss{(b)}}$, and the degrees of vertices on the path from $h$ to $v$ are less than $u_{i_{\max}-z+j}^{\sss{(b)}}$, and also what red has already occupied at that moment, i.e. the degree of the $j$-th vertex on the path must be less than $\min\{u_{i_{\max}-z+j}^{\sss{(b)}}, \wit u_{ t_b-z+ j-1-T_r}^{\sss{(r)}}$\}.  This extra truncation is needed since we do not want to count vertices explored from $\CA_{i_{\max}-z}^{\sss{(b)}}$ towards $\CA_{i_{\max}-z+1}^{\sss{(b)}}$, hence the additional restriction. We show that the total number of optional blue vertices in lower layers, $\sum_{z\ge 0}\ops_{-z}$ with these additional explorations is at most the same order as $\ops_{\max}$ in Lemma \ref{lem::opt-z-lemma}. 

For the lower bound of what blue can occupy after $t_b$, note that not every vertex in $\opt_{\max}$ will be occupied by blue: red can still bite out some parts of these vertices by simply randomly being close to some parts of the blue cluster. We estimate the number of vertices in the intersection of $\opt_{\max}$ and red, and then subtracting the gained estimate from the lower bound on $\ops_{\max}$ gives a lower bound on what blue occupies from the graph after $t_b$, see Lemma \ref{lem::red-blue-intersection} below. Now we turn to the calculations.

As before, we use $\Pv_{Y,n}(\cdot), \Ev_{Y,n}[\cdot]$ defined in \eqref{def::py-ey}.

We introduce the expected truncated degree of a vertex that is distance $j$ away from the set $\CMBN$ by
\be\label{def::nuj} \nu_j:=\Ev_{Y,n}\left[D^\star \ind_{\left\{D^\star< \wit u_{t_b+j-1-T_r}^{\sss{(r)}}\right\}} \right],  \ee
Then \eqref{eq::d-star-indicator} yields an upper bound on $\nu_j$, and the same expression with $C_1^\star$ replaced by $c_1^\star$ serves as a lower bound.
Let us also define
\[ \kappa_j:=\frac{1}{\Ev[D]}\Ev\left[D(D-1)(D-2)\ind_{\left\{D< \wit u_{t_b+ j-1-T_r}^{\sss{(r)}}\right\}}\right] \]
Then, again by \eqref{eq::F},
\[ \frac{c_1}{\Ev[D]} \left(\wit u_{t_b+ j-1-T_r}^{\sss{(r)}}\right)^{4-\tau}\le \kappa_j \le  \frac{C_1}{\Ev[D]} \left(\wit u_{ t_b+ j-1-T_r}^{\sss{(r)}}\right)^{4-\tau}.\]
Recall the definition of a path from page 13. This time, let us call a path of length $k$ from $\CMBN$ with vertices $\left(\pi_j\right)_{j\le k}$ \emph{good} if $\pi_j\le \wit u_{ t_b+ j-1-T_r}^{\sss{(r)}}$, and \emph{good-directed}  if  $\wit u_{ t_b+ j-T_r}^{\sss{(r)}}\le \pi_j\le \wit u_{ t_b+ j-1-T_r}^{\sss{(r)}}$.
\begin{lemma}\label{lem::wihZk} For $k\ge 0$, denote by $\ops_{\max}(k), \ops_{\max}^{\mathrm{d}}(k)$ the number of vertices that are on good and good-directed paths of distance $k$ away from $\CMBN$, respectively.  Then,
\be \label{eq::wihZk_expected}   \MBN \cdot \prod_{j=1}^{k} \nu_j \le \Ev[\ops_{\max}(k)\mid \MBN ]\le  \MBN \cdot \prod_{j=1}^{k} \nu_j \cdot \left(1+O\left(\frac{k^2}{n}\right)\right)\ee
and
\be \label{eq::wihZk_expected2}    \MBN \cdot \prod_{j=1}^{k} (\nu_j-\nu_{j+1}) \le \Ev[\ops_{\max}^{\mathrm{d}}(k)\mid \MBN ]\le  \MBN \cdot \prod_{j=1}^{k} (\nu_j-\nu_{j+1}) \cdot \left(1+O\left(\frac{k^2}{n}\right)\right)\ee
while for the variance of the latter:
\be \ba \label{eq::wihZk_variance}&\Vv[\ops_{\max}^{\mathrm{d}}(k)|\MBN] \le \Ev[\ops_{\max}^{\mathrm{d}}(k)|\MBN] \\
&\ + \overline{\Ev[\ops_{\max}^{\mathrm{d}}(k)|\MBN]}^2 \cdot \left(\frac{\nu_{k-1}}{(\nu_{k-1}-1)}\frac{\kappa_1}{\nu_{1}^2}\left(\frac{1}{\MBN} + \frac{2}{\CL_n}\right)+ \frac{\nu_{k-1}^2}{(\nu_{k-1}^2-1)^2}\frac{\kappa_1^2}{\nu_{1}^4} \frac{2}{\MBN  \CL_n}+e_{k,n}\right), \ea\ee
where  $\overline{\Ev[\ops_{\max}^{\mathrm{d}}(k)|\MBN]}$ means the upper bound on $\Ev[\ops_{\max}^{\mathrm{d}}(k)|\MBN]$ in \eqref{eq::wihZk_expected2}, and the error term $e_{k,n}$ is given by
\be\ba\label{def::ekn} e_{k,n}&=\left( \prod_{i=1}^{k} \frac{\CL_n-2i+1}{\CL_n-2i-2k+1}-1\right)\\
&\ \ +   \left(1+\frac{\kappa_1\nu_{k-1}}{\nu_1^2}\frac{1}{\MBN} \right)\left(1+\frac{\kappa_1 \nu_{k-1}}{\nu_1^2 }\frac{1}{c\CL_n}\right)\frac{k}{\nu_{k-1}-1}\left(e^{k^2 \kappa_1^2 \nu_{k-1}/(\nu_{1}^4\CL_n)}-1\right).\ea\ee
\end{lemma}
 The proof of this lemma uses path counting methods and is similar to that of \cite[Lemma 5.1]{Janson10}. Similar techniques can also be found in \cite[Volume II.]{H10}. The detailed proof can be found in  \cite[Appendix]{BarHofKom14}, where $\la=1$ can be set.

Now we state the immediate corollary of Lemma \ref{lem::wihZk}. Recall the definition of $t_b$ from \eqref{def::t_b}.
\begin{corollary}[Chebyshev's inequality for blue vertices]\label{cor::chebisev} Take $c_3\le (1-\ve) |\log(\tau-2)|^{-1}$ and any $k\le c_3\log\log n.$
Then, conditioned on the number of blue half-edges $\MBN$ at time $t_b$, the number of vertices optionally occupied by blue up to time $t_b+  k$ satisfies that, conditionally on $\MBN$,
\[\frac{\log (\ops_{\max}(k))}{\log \MBN + \sum_{i=1}^{k-1} \log \nu_i} \toinp 1. \]
\end{corollary}

\begin{proof}
 In this proof below, all expectations and probabilities are conditional wrt.\ $\MBN$: since $\MBN$ is a deterministic function $n,\Yrn, \Ybn,$ we express this conditioning by using the $\Pv_{Y,n}(\cdot), \Ev_{Y,n}[\cdot]$ notation.  Let us write
\[ J:=\Pv_{Y,n}\left( \big|\ops_{\max}(k) -\Ev_{Y,n}[\ops_{\max}(k) ]\big|\ge \frac12 \Ev_{Y,n}[\ops_{\max}(k)] \right).\]
Then by a simple triangle inequality,
\be\label{eq::triangle} \ba J &\le \Pv_{Y,n}\left(|\ops_{\max}(k) - \ops_{\max}^{\text{d}}(k)| \ge \Ev_{Y,n}[\ops_{\max}(k)]/6 \right) \\
&+\Pv_{Y,n}\left( \big|\ops_{\max}^{\text{d}}(k) -\Ev_{Y,n}[\ops_{\max}^{\mathrm{d}}(k) ]\big|\ge \Ev_{Y,n}[\ops_{\max}(k)]/6 \right)\\
&+\Pv_{Y,n}\left( |\Ev_{Y,n}[\ops_{\max}^{\text{d}}(k) - \Ev_{Y,n}[ \ops_{\max}(k)] |\ge  \Ev[\ops_{\max}(k)]/6 \right)\\
&:= J_1+J_2+J_3. \ea\ee
We can apply Chebyshev's inequality on the second term, using the lower bound in \eqref{eq::wihZk_expected}, the upper bound in \eqref{eq::wihZk_expected} as an upper bound on $\overline{\Ev[\ops_{\max}^{\mathrm{d}}(k)|\MBN]}$, and the variance formula  in Lemma \ref{lem::wihZk}:
\be\label{eq::J}\ba J_2&\le \frac{9 \Vv[\ops_{\max}^{\text{d}}(k)| \MBN]}{\Ev[\ops_{\max}(k)|\MBN]^2}  \\
&\le \left( \frac{1}{\MBN } \frac{\kappa_1}{\nu_1^2}\frac{\nu_{k-1}}{(\nu_{k-1}-1)} +\frac{ \kappa_1^2}{\nu_{1}^4 \CL_n}  \frac{2k^4 \nu_{k-1}}{\nu_{k-1}-1} \Big(1+ O\big(\frac{1}{\MBN } \frac{\kappa_1}{\nu_1^2}\big)\Big) \right) \left(1+ O(\tfrac{k^2}{n})\right).  \ea\ee
The term containing $\kappa_1^2/\nu_{1}^4 \CL_n$ comes from the Taylor expansion of the exponential factor in the formula for $e_{k,n}$.
We have to verify that the rhs to $0$. For this we need $\kappa_1/(\nu_1^2 \MBN) \to 0$ and also $\kappa_1^2/(\nu_1^4 \CL_n) \to 0$.

For $\kappa_1/(\nu_1^2 \MBN)$ note that $\MBN\ge D_n^{\max}(\infty)$, since it counts the number of half-edges with maximal degree $D_n^{\max}(\infty)$. Further, $\kappa_1/\nu_1^2= (\wit u_{ t_b -T_r}^{\sss{(r)}})^{\tau-2}=(\wit u_{\lfloor t_c \rfloor+1}^{\sss{(r)}})^{\tau-2}=o( D_n^{\max}(\infty))$, since it is not hard to see that at time $t_b$, the degree above which red occupies everything (i.e., $\wit u_{\lfloor t_c\rfloor+1}^{\sss{(r)}}$) is already less than $D_n^{\max}(\infty)$,
 otherwise blue could have still increased its maximal degree at $t_b+1$ by an extra jump.
  (Alternatively, compare the exact values of $D_n^{\max}(\infty)$ in \eqref{eq::max-deg-2}, and compare it to that
 of $(\wit u_{  t_c \rfloor+1}^{\sss{(r)}})^{(\tau-2)^2}$, which can be derived by multiplying the rhs of \eqref{eq::wur-tc} by $(\tau-2)^2$.)

Similarly,  the second term, $\kappa_1^2/(\nu_1^4 \CL_n)=(\wit u_{\lfloor t_c\rfloor+1}^{\sss{(r)}})^{2(\tau-2)}/\CL_n$ is small as long as $(\wur_{ \lfloor t_c\rfloor+1})^{\tau-2}=o(\sqrt{n})$. This is always true, since already $(\wur_1)^{\tau-2}=o(\sqrt{n})$.

To handle $J_1$ in  \eqref{eq::triangle}, we Markov's inequality, conditioned on $\MBN$:
\[ J_1 \le \frac{\Ev_{Y,n}[\ops_{\max}(k)|\MBN]-\Ev_{Y,n}[\ops_{\max}^{\text{d}}(k)|\MBN]}{6^{-1}\Ev_{Y,n}[\ops_{\max}(k)|\MBN]} \le 6(1+O(\tfrac{k^2}{n}))\Big(1- \prod_{j=1}^k \big(1-\frac{\nu_{j+1}}{\nu_j}\big)\Big). \]
We use $\nu_i/(\wur_{t_b+ i-1 -T_r})^{3-\tau}\in [c_1, C_1]$ hence the last factor on the rhs is at most
\[ \sum_{j=1}^k \frac{\nu_{j+1}}{\nu_{j}} \le \sum_{\ell=t_b -T_r}^{t_b+ k-1 -T_r }\frac{C_1}{c_1} \left(\frac{\wit u_{\ell+1}}{\wit u_\ell}\right)^{3-\tau} \le \sum_{\ell=[t_b -T_r]}^{t_b+k-1 -T_r} \frac{C_1(C\log n)^{3-\tau}}{c_1(\wur_\ell)^{(3-\tau)^2}},\]
where we have used the recursion $\wur_{\ell+1}=C\log n\,(\wur_{\ell})^{\tau-2}$ in \eqref{eq::wideui_recursion}.
Again, by the same recursion, for some large enough constant $C'$, the sum on the rhs is at most
\[ \frac{C_1(C\log n)^{3-\tau}}{c_1}\frac{C'}{(\wur_{t_b+k -1-T_r} )^{(3-\tau)^2}},\]
which is small as long as $ (\wur_{t_b+k -1-T_r})^{3-\tau}$ is of larger order than $C\log n$. Note that this holds for an appropriate choice of $k=k(n)$ by \eqref{eq::core-ell}.

Finally, notice that $J_3$ in \eqref{eq::triangle} is deterministic under the measure $\Pv_{Y,n}(\cdot)$. It is not hard to see that this inequality always holds whenever
\[ \frac{\Ev_{Y,n}[\ops_{\max}^{\text{d}}]}{\Ev_{Y,n}[\ops_{\max}^{\text{d}}]}\ge \prod_{i=1}^{k} \big(1- \frac{\nu_{j-1}}{\nu_j}\big) \big(1+ O(\tfrac{k^2}{n})\big)^{-1}\ge 5/6. \]
By the same argument as the one used for the term $J_1$, this holds for large enough $n$.
\end{proof}

We cite the next lemma without its proof from  \cite[Lemma 7.3]{BarHofKom14}, showing that $\sum_{z\ge 0}\ops_{-z}(k)$, the vertices reached via half-edges from vertices in the sets $\CA_{i_{\max}-z}\setminus \CA_{i_{\max}-z+1}$, (and not going through a higher layer), are at most the same order of magnitude as  $\ops_{\max}(k)$.
\begin{lemma}[\cite{BarHofKom14}]\label{lem::opt-z-lemma} With the notation introduced before,
\[ \log \Big(\sum_{z\ge 0}\ops_{-z}(k) \Big)\le \log (\ops_{\max}(k)) (1+o_{\Pv}(1)). \]
\end{lemma}

Having analysed the size of the optional cluster of blue,  we are ready to finish the upper bound of Theorem \ref{thm::main} by combining the previous results.
\begin{proof}[Proof of the upper bound in Theorem \ref{thm::main}.]
First, fix $k=k(n)\to \infty$ so that $k(n)=o(\log\log n)$.
Then, Lemma \ref{lem::opt-z-lemma} implies that the logarithm of the  total number of vertices that blue paints in the last phase is at most $\log \ops_{\max}(k) (1+o_{\Pv}(1))$. Corollary \ref{cor::chebisev} says that the order of magnitude of $\log (\ops_{\max}(k))=\log \MBN + \sum_{j=1}^{k-1}\log  \nu_j + o_{\Pv}(1)$, where $\MBN$ is the number of blue half-edges in the highest layer that blue can reach.
Further, Lemma \ref{lem::verticeswithmaxdegree} determines the order of magnitude of $\log \MBN$, which equals
\be\label{eq::order-mbn}\log \MBN=\sqrt{\Ybn/\Yrn} \log n \cdot \frac{h_n^{\text{half-edge}}(\Yrn, \Ybn)}{\tau-1} (1+o_{\Pv}(1)),\ee and hence converges in distribution to $(Y_b/Y_r)^{1/2}$ when divided by the second two factors.

Thus, to get the asymptotic behavior of $\log (\ops_{\max}(k))$, it remains to calculate  $\sum_{j=1}^k \log \nu_j$ and compare it to the order of $\log \MBN$. For this recall  the definitions of $\nu_j$ in \eqref{def::nuj}, $t_b$ in \eqref{def::t_b}, $t_c$ in \eqref{eq::tc2}, $\wur_\ell$ in \eqref{eq::ul}, and the upper bound on $\nu_j$ in \eqref{eq::nujbound}.
 \begin{align} \label{eq::prodnuj1}\sum_{j=1}^{k} \log \nu_j \le& \sum_{j=1}^k \log\left( C_1 \left(\wit u_{ t_b+ j-1-T_r\rfloor}\right)^{3-\tau}\right)\\
  \le& \sum_{j=0}^{k-1} \left\{\left(\anr\log n \!+\! \bnr \log (C\log n)\right)(\tau-2)^{\lfloor t_c \rfloor+j}  (3\!-\!\tau)\right\} \!+\!k \log  (C_1 C\log n).\nonumber \end{align}
Note that this is simply a geometric sum.  Using \eqref{def::delta} we obtain
\be\label{eq::orderoflognuj} \sum_{j=1}^{k} \log \nu_j\le\ (1+o_{\Pv}(1))\cdot  \log n
\frac{(\tau-2)^{\dnr}}{\tau-1}(\tau-2)^{\lfloor t_c \rfloor}(1- (\tau-2)^{k}).
\ee
Recall that $\log (\wur_\ell)/\log n=(\tau-2)^{\dnr+\ell-1}/(\tau-1)(1+o_{\Pv}(1))$, (by \eqref{eq::ul} and \eqref{def::delta}), and hence the expression in \eqref{eq::orderoflognuj} is essentially $\log \wur_{\lfloor t_c \rfloor+1}$, multiplied by by a factor of $(1-(\tau-2)^k)$ in the exponent. Using  the rhs of \eqref{eq::wur-tc}, \eqref{eq::t2-t1exp} and the value of $\{t_c\}$ in the four cases $(E_<, O_>, E_>, O_<)$,
\be\label{eq::log-nuj}   \sum_{j=1}^{k} \log \nu_j\le\ (1+o_{\Pv}(1))\cdot  \log n  \left(\frac{\Ybn}{\Yrn}\right)^{1/2}
\frac{(\tau-2)^{\dnr + \frac{\bnr-\bnb}{2}  }}{\tau-1}(\tau-2)^{(\ind_{E_>}-\ind_{E_<})/2}(1- (\tau-2)^{k}).  \ee
Recall that $(\tau-2)^{\dnr}=((\tau-1)-(\tau-2)^{\bnr})/(\tau-1)$, hence, let us  introduce
 \be \label{def::dnY} h^{\text{paths,k}}_n( \Yrn, \Ybn):=(\tau-1-(\tau-2)^{b_n^{(r)}}) (\tau-2)^{(\bnr-\bnb+\ind_{E_>} - \ind_{E_<})/2 }(1-(\tau-2)^k).\ee
Let \be \label{def::Cpath}h^{\text{paths}}_n( \Yrn, \Ybn):=\lim_{k\to \infty} h^{\text{paths,k}}_n( \Yrn, \Ybn).\ee
  Now, recall again that $\log \ops_{\max}(k)=\log \MBN+ \sum_{j=1}^k \nu_j+ o_{\Pv}(1)$ by Corollary \ref{cor::chebisev}, and combine \eqref{eq::log-nuj} with \eqref{eq::order-mbn}
  \[  \frac{\log \ops_{\max}(k) }{
\log n \!\cdot\! (\tau-1)^{-1}\left(  h_n^{\text{half-edge}}(\Yrn, \Ybn) + h^{\text{paths}}_n( \Yrn, \Ybn)\right)} \le \sqrt{ \frac{\Ybn}{\Yrn}}(1+ o_{\Pv}(1)). \]
Now we can finally use $\Yrn\toindis Y_r, \ \Ybn\toindis Y_b$ (by Theorem \ref{thm::davies}
). Hence, the right hand side converges to $\sqrt{Y_b/Y_r}$. Note that $f_n(\Yrn, \Ybn):=h_n^{\text{half-edge}}(\Yrn, \Ybn) + h^{\text{paths}}_n( \Yrn, \Ybn)$ after some elementary rearrangements and simplifications exactly gives \eqref{eq::main-f}, finishing the proof of the upper bound.
\begin{remark}\normalfont It is not entirely obvious from the formulas that the factor multiplying $\log n$, $\sqrt{\Ybn/\Yrn} f_n(\Yrn, \Ybn) (\tau-1)^{-1}$, is always strictly less than $1$. We investigate this issue below in Lemma \ref{lem::no-coexistence}. \end{remark}
\end{proof}

Recall that the vertices reached via half-edges from layer $\CA_{i_{\max}-z}^{\sss{(b)}}\setminus \CA_{i_{\max}-z+1}^{\sss{(b)}}$ are denoted by $\opt_{-z}, z\ge 0.$
For the lower bound of the proof of the first part of Theorem \ref{thm::main}, let us introduce the notation $\opt(k):=\bigcup_{z\ge0} \opt_{-z}(k) \cup \opt_{\max}(k)$, and set $\opss(k):=|\opt(k)|$,
where $k$ stands for the length of the paths we are counting.
The next lemma shows that essentially all the vertices in $\opt(k)$ for some $k=k_n=o(\log\log n)$ will indeed be painted blue, i.e., red cannot accidentally bite out too much from this set:

\begin{lemma}\label{lem::red-blue-intersection}
Set $k=k(n)=o(\log\log n)$.
The expected number of vertices in the intersection of $\CR_{t_b+ k}$ and $\opt(k)$ is small:
\[ \Ev_{Y,n}\left[|\opt(k)\cap  \CR_{t_b+ k}|\right]= o_{\Pv}\left(\ops_{\max}(k)\right),\]
hence $|\opt(k)\setminus(\opt(k)\cap  \CR_{t_b+ k})|=\ops_{\max}(k)(1+o_{\Pv}(1))$.
\end{lemma}
\begin{proof}
The proof of this lemma is identical to that of \cite[Lemma 7.4]{BarHofKom14}, with setting $\la=1$, hence, we refer the reader there.
\end{proof}

\begin{proof}[Proof of the lower bound of Theorem \ref{thm::main}]
The proof is identical for the proof of the lower bound of \cite[Theorem 1.2]{BarHofKom14}, hence we again refer the reader there.
\end{proof}

\begin{lemma}\label{lem::no-coexistence}
If $\Ybn/\Yrn<\tau-2$, then there is no coexistence, i.e. $\CB_\infty=o_{\Pv}(n)$.
\end{lemma}
\begin{proof}
To show the statement, we have to recall from the proof of Lemma  \ref{lem::verticeswithmaxdegree} that
\[  \frac{\log \MBN}{(\tau-1)^{-1}\log n} =\left(  \ind_{ O_< \cup E_>} ((\tau-2)^{2 \{t_c\} } +  (3-\tau)) + \ind_{O_> \cup E_< } \right)\frac{\log (D_n^{\max} (\infty))}{(\tau-1)^{-1}\log n}, \]
and now use the representation of the maximal degree that is listed in the four cases at the bottom of page 31 to get that in Cases $O_<, E_>$
\be\label{eq::mbn-ole-ege} \frac{\log \MBN}{(\tau-1)^{-1}\log n}  = (\tau-2)^{(T_b - T_r-1-\ind_{O_< })/2} (\tau-2)^{\dnr} \left( (\tau-2)^{\bnb - \dnr + \ind_{O_<}} + (3-\tau) \right),  \ee
while in Cases $O_>, E_<$
\be\label{eq::mbn-oge-ele} \frac{\log \MBN}{(\tau-1)^{-1}\log n}  = (\tau-2)^{(T_b - T_r-1-\ind_{O_> })/2} (\tau-2)^{\bnb}.   \ee
We also have to recall that the path counting method gives \eqref{eq::orderoflognuj}, combined with the observation below \eqref{eq::orderoflognuj}, namely that the rhs of  \eqref{eq::orderoflognuj} equals the right hand side of \eqref{eq::wur-tc} multiplied by $\tau-2$. Hence,
\[ \frac{\sum_{j=1}^\infty \log \nu_j }{(\tau-1)^{-1} \log n} = (\tau-2)^{(T_b-T_r + 1)/2}  (\tau-2)^{(\bnb+\dnr)/2 - \{t_c\}}.   \]
Combining this again with the value of $\{t_c\}$ in the four cases listed at the bottom of page 31, we get
\be\label{eq::lognuj-rewrite} \frac{\sum_{j=1}^\infty \log \nu_j }{(\tau-1)^{-1} \log n} = (\tau-2)^{(T_b-T_r + \ind_{E_>} - \ind_{E_<})/2} (\tau-2)^{\dnr}.  \ee
To show that $\CB_\infty = o_{\Pv}(n)$ we need to show that the sum of the rhs of \eqref{eq::mbn-ole-ege} or \eqref{eq::mbn-oge-ele} plus the rhs of \eqref{eq::lognuj-rewrite} is less than $(\tau-1)$.
We analyse the four cases separately.

\emph{Case $O_>$}. Note that by the definition of the event $O_>$, $T_b-T_r-1$ is odd, hence $T_b-T_r\ge 2$ is even. The sum of the right hand sides of \eqref{eq::mbn-oge-ele} and \eqref{eq::lognuj-rewrite} in this case equals
\[ (\tau-2)^{(T_b-T_r-2)/2} \left( (\tau-2)^{\bnb} + (\tau-2)^{\dnr+1}  \right) < (\tau-2)^0 + (\tau-2)^1 = \tau-1, \]
where we have used that  $T_b-T_r-2\ge 0$, hence the first factor is at most $1$, and also $\bnb\ge 0, \dnr>0$ by \eqref{def::delta}.

\emph{Case $O_<$}. Note that by the definition of the event $O_>$,  $T_b-T_r-1\ge 1$ is odd. The sum of the right hand sides of \eqref{eq::mbn-ole-ege} and \eqref{eq::lognuj-rewrite} in this case equals
\[ (\tau-2)^{(T_b-T_r-1)/2} \left( (\tau-2)^{\bnb+1} + (\tau-2)^{\dnr} \right) <  (\tau-2)^1 + (\tau-2)^0 = \tau-1,\]
since $T_b-T_r-2\ge 0$, and again, $\bnb\ge 0, \dnr>0$ by \eqref{def::delta}.

\emph{Cases $E_<$ and $E_>$}. Note that by the definition of the event $E_<, E_>$, in these cases $T_b-T_r- 1\ge 0$ is even. The sum of the right hand sides of \eqref{eq::mbn-oge-ele} or \eqref{eq::mbn-ole-ege} and \eqref{eq::lognuj-rewrite} in both cases equals
\[ (\tau-2)^{(T_b-T_r-1)/2} \left( (\tau-2)^{\dnr} + (\tau-2)^{\bnb} \right). \]
Clearly, if $T_b\ge T_r + 3$, then this expression is at most $2(\tau-2) < \tau-1$ since $\tau<3$.
Now, if  $T_b= T_r+1$, then the first factor is $1$  and the second factor equals $\tau-1- (\tau-2)^{\bnr}  +(\tau-2)^{\bnb}$, which is at most $\tau-1$ if and only if $\bnb>\bnr$. Recall the analysis of the crossing the peak in Section \ref{sc::peak}, Case (3): we see that if $T_b=T_r+1$, then $\bnb>\bnr$ if and only if $\Ybn/\Yrn<\tau-2$, which is exactly what we assume here throughout. Hence, the expression is less than $\tau-1$ in these cases.
\begin{remark}\normalfont
We now comment on why we did not prove the statement of the lemma directly using the formula for $f_n(\Yrn, \Ybn)$. Note that by \eqref{eq::t2-t1exp}, $(\tau-2)^{(T_b-T_r)/2} = \sqrt{\Ybn/\Yrn} (\tau-2)^{(\bnr-\bnb)/2}$. Hence, for instance in Cases $E_<, E_>$,
\[ \frac{\log \CB_\infty}{\log n} = \sqrt{\Ybn/\Yrn} (\tau-2)^{(\bnr-\bnb-1)/2} \left( (\tau-2)^{\dnr} + (\tau-2)^{\bnb} \right).\]
Notice that the first factor on the right hand side is at most  $\sqrt{\tau-2}$, hence we need that the other factors are at most $(\tau-1) / \sqrt{\tau-2}$ in these cases. This is however only true under the extra information that if $\sqrt{\Ybn/\Yrn}$ is close to $\tau-2$, so that this formula holds (not the ones for $O_<, O_>$), then necessarily $\bnb\ge \bnr$, and $\bnb\searrow \bnr$  as $\Ybn/\Yrn \nearrow \tau-2$. On the other hand, if $\Ybn/\Yrn\ll\tau-2$, then the relationship between $\bnr, \bnb$ is not necessarily the same, and hence, e.g. on the event $\bnb \nearrow 1$ the maximum of $f_n(\Yrn, \Ybn)$ is $(\tfrac{2\tau-3}{3})^{3/2} 2/(\tau-2)$ which is in fact strictly larger than $(\tau-1)/\sqrt{\tau-2}$.
\end{remark}

\end{proof}

\section{Acknowledgement}
The work of RvdH and JK was supported in part by the Netherlands Organisation for Scientific Research (NWO) through VICI grant 639.033.806.

The work of JK was supported in part by NWO through the STAR cluster, the VENI grant and the work of RvdH was supported in part by NWO through Gravitation grant 024.002.003. JK thanks the Probability Group at The University of British Columbia for their hospitality while working on this project.

\appendix
\section{Comparison of typical distances}\label{appendix::distances}
In this appendix we compare the result of Theorem \ref{thm::distances} to \cite[Theorem 1.2]{HHZ07} in more detail. Throughout, we write $\CD_n:=\CD_n(u,v)=\CD_n(\CR_0, \CB_0)$ the graph distance between two uniformly chosen vertices.

Here we argue that the two formulations are indeed the same, by describing the core idea of the proof of \cite[Theorem 1.2]{HHZ07}, and relate  quantities (events, random variables, etc.) appearing in the proof to quantities in our paper. The proof of \cite[Theorem 1.2]{HHZ07} goes through a minimisation problem, where the two clusters or red and blue should connect the first time such that a coupling should be maintained. More precisely, suppose $k_1$ is a random variable that is measurable w.r.t.\ $\{Z_{s}^{\sss{(j)}}\}_{s=1}^m$, for some $m$ (in this paper we take $m=n^\vr$).  Suppose we run color red for $k_1$ steps, color blue for $k-k_1-1$ steps. There are $Z_{k_1+1}^{\sss{(r)}}, Z_{k-k_1}^{\sss{(b)}}$ many half edges attached to the vertices in the two colored clusters, respectively. The distance between the two source vertices is then larger than $k$ if these sets of half-edges do not connect to each other, and the probability of this event is approximately
\be \label{eq::distance-tail} \Pv(H_{k_1+1}^{\sss{(r)}}\cap H_{k-k_1}^{\sss{(b)}}=\varnothing) \approx \exp\{ - c Z_{k_1+1}^{\sss{(r)}} Z_{k-k_1}^{\sss{(b)}} / \CL_n\} \ee
  A branching process approximation similar to the one in Section \ref{sc::BP} is performed to approximate the numerator in the exponent. However, this BP approximation is only valid until none of the colors have more half-edges than $n^{(1-\ve) / (\tau-1)}$ for some small $\ve>0$, i.e., they do not go over the top of the mountain. This criterion is established in \cite[Proposition 3.2]{HHZ07}. The set $\CT_{m}^{i,n}$ in equation (3.3) in \cite{HHZ07} exactly describes those values of $\ell$ for which $\{\ell\le T_j+1\}$, $j=r,b$ holds (where $T_j$, defined in \eqref{eq::k*+i*}, is the time to reach the top vertices). The $+1$ is added to $T_j$ since the half-edges attached to vertices in the colored cluster at time $T_j$ can be described as $Z_{T_j+1}^{\sss{(j)}}$.

Now, from \eqref{eq::distance-tail}, we see that $\{\CD_n>k\}$ happens whp if $Z_{k_1+1}^{\sss{(r)}}\cdot Z_{k-k_1}^{\sss{(b)}} = o(n)$ and also that both  $k_1+1 \in \CT_{m}^{r,n}$ and $k-k_1 \in \CT_{m}^{b,n}$ holds: this is described in the event $\CB_n$ in equation (4.57) in \cite{HHZ07}.
Using the BP approximation, this becomes
\[ \Pv(\CD_n >k) \approx \max_{k_1 \in \CB_n}\exp \{ - C \exp\{(\tau-2)^{-(k_1+1)}\Yrn + (\tau-2)^{-(k-k_1)} \Ybn - \log n  \}   \}. \]
With the event \[ \CE_{n,k}:=\{ \exists k_1\in \CB_n, (\tau-2)^{-(k_1+1)}\Yrn + (\tau-2)^{-(k-k_1)} \Ybn< \log n\} \] it is obvious from \eqref{eq::distance-tail} that $\lim_{n\to \infty}\Pv(\CD_n>k| \CE_{n,k}) \to 1$, while $\Pv(\CD_n>k| \CE_{n,k}^c) \to 0$. Hence, we get that
\be\label{eq::distance-tail-BP} \Pv(\CD_n>k) \approx \Pv( \min_{k_1\in \CB_n}(\tau-2)^{k_1+1}Y_r + (\tau-2)^{k-k_1} Y_2 < \log n  )\ee
where again \[  \CB_n=\{ k_1+1 \le T_r+1, k-k_1 \le T_b+1\}.\]
The paper shows that $\min_{k_1\in \CB_n}$ can be replaced by $\min_{k_1\le k}$ in the minimum above.
Next we show that the formulation of \eqref{eq::distance-tail-BP} gives the same distances as our statement for typical distances in Theorem \ref{thm::distances}, that is,
\be\label{eq::our-statement} \CD_n = T_r + T_b + 1 + \ind_{(\tau-2)^{\bnb} + (\tau-2)^{\bnr} < \tau-1}.\ee
From \eqref{eq::k*+i*} it is an elementary calculation to check that that for any $i\in \Z$, $j=r,b$
\be \label{eq::degree-at-tji} \left(\tfrac{1}{\tau-2}\right)^{T_j+1+i} Y_j^{\sss{(n)}} = \log n \frac{(\tau-2)^{\bnr-i}}{\tau-1}.  \ee

First, we check that \eqref{eq::distance-tail-BP} gives $\Pv(\CD_n > T_r+ T_b) =1$.
So, we set $k=T_r+T_b$, and let us write $k_1:=T_r-\ell$ for some $\ell\in \Z$, then $k-k_1=T_b+\ell$. Hence, we can rewrite \eqref{eq::distance-tail-BP} using  \eqref{eq::degree-at-tji} with $i=-\ell$ and $i=\ell-1$, and get
\[ \Pv( \CD_n > T_r+T_b) \approx  \Pv( \min_{\ell} \frac{(\tau-2)^{\bnr+\ell} + (\tau-2)^{\bnb+1-\ell}}{\tau-1} < 1  ). \]
It is  clear now that setting $\ell=0$ gives a solution for all $\bnr, \bnb \in [0,1)$, since the expression after the $\min_\ell$, for $\ell=0$ is at most $(1+(\tau-2))/(\tau-1) = 1$.  Moreover, note that for $\ell=0$ both $k_1+1\le T_r+1$ and $k-k_1 \le T_b+1$ hold, hence, we found a solution in $\CB_n$.

Next, we check that
\be\label{eq::distance-indicator} \Pv(\CD_n > T_r+T_b +1) = \Pv( \ind_{(\tau-2)^{\bnb} + (\tau-2)^{\bnr} < \tau-1}=1  ).\ee
For this, we set $k=T_r+T_b+1$, write again $k_1=T_r-\ell$, then $k-k_1=T_b+1+\ell$. Hence, we can rewrite \eqref{eq::distance-tail-BP} using \eqref{eq::degree-at-tji} with $i=-\ell$ and $i=\ell$ and get
 \[ \Pv( \CD_n > T_r+T_b+1) \approx  \Pv( \min_{\ell} \frac{(\tau-2)^{\bnr+\ell} + (\tau-2)^{\bnb-\ell}}{\tau-1} < 1  ). \]
It is clear now that setting $\ell=0$ yields \eqref{eq::distance-indicator},  moreover, note that for $\ell=0,$ the event $\CB_n$ holds as well. We argue that no other choice of $\ell$ gives a smaller answer. Wlog we can assume $\ell\ge 1$, the case when $\ell\le -1$ can be treated similarly. Then,
we need to show that for all $\ell\ge 1$,
\be\label{eq::needed}  (\tau-2)^{\bnr+\ell} +  (\tau-2)^{\bnb-\ell} > (\tau-2)^{\bnr} +  (\tau-2)^{\bnb}. \ee
since $\bnr,\bnb \in [0,1)$, the following two inequalities hold:
\[ \ba  (\tau-2)^{\bnb} \left( (\tau-2)^{-\ell }-1\right) &> (\tau-2)^{-\ell+1} -(\tau-2) \\
 (\tau-2)^{\bnr} \left( 1-(\tau-2)^{\ell }\right) &\le 1-(\tau-2)^{\ell}. \ea\]
Hence, \eqref{eq::needed} holds if $1-(\tau-2)^{\ell} \le (\tau-2)^{-\ell+1} -(\tau-2)$ holds. It is elementary to check that this is the case for all $\ell\ge 1$ and $\tau\in(2,3)$.

Finally, we check that \be\label{eq::distance-toolarge} \Pv(\CD_n > T_r+T_b +2) = 0  \ee
holds also. For this, we set $k=T_r+T_b+2$, write again $k_1=T_r-\ell$, then $k-k_1=T_b+2+\ell$. Hence, we can rewrite \eqref{eq::distance-tail-BP} using \eqref{eq::degree-at-tji} with $i=-\ell$ and $i=\ell+1$ and get
 \[ \Pv( \CD_n > T_r+T_b+2) \approx  \Pv( \min_{\ell} \frac{(\tau-2)^{\bnr+\ell} + (\tau-2)^{\bnb-\ell-1}}{\tau-1} < 1  ). \]
We need to show that no $\ell\in \Z$ satisfies this minimisation problem.
For this, we use again that $\bnr,\bnb \in [0,1)$ implies that
\[(\tau-2)^{\bnr+\ell} + (\tau-2)^{\bnb-\ell-1} > (\tau-2)^{\ell+1} + (\tau-2)^{-\ell},\]
and it is elementary to show again that the rhs is at least $\tau-1$ for all $\ell \in \Z$ and $\tau \in (2,3)$.

These calculations show that the statement of Theorem \ref{thm::distances} is identical - though a non-trivial rewrite of - the statement of \cite[Theorem 1.2]{HHZ07}. The final formula of \cite[Theorem 1.2]{HHZ07}, i.e., the distribution of the fluctuation of the typical distance around $2\log\log_n/|\log (\tau-2)|$ is then obtained by solving the minimisation problem on the rhs of \eqref{eq::distance-tail-BP} with $k_1\in \CB_n$ replaced by $k_1 \le k$.

\bibliographystyle{abbrv}
\bibliography{refscompetition}

\end{document}